\documentclass[12pt,reqno,a4paper]{amsart} 

\usepackage{fullpage}
\usepackage[english]{babel}
\usepackage{amsmath}
\usepackage{amsfonts,amssymb}
\usepackage{enumerate}
\usepackage{mathrsfs}

\usepackage{amsthm}
\usepackage{tikz}
\usetikzlibrary{arrows}
\usepackage{mathrsfs}
\usepackage{caption}
\usepackage{color}
\usepackage{pgfplots}

\pgfplotsset{compat=1.14}

\usepackage{hyperref} % [hypertex]
\hypersetup{
    colorlinks=true,       
    linkcolor=red,          
    citecolor=blue,        
    filecolor=magenta,      
    urlcolor=cyan           
}

\newcommand*\circled[1]{\tikz[baseline=(char.base)]{
          \node[shape=circle,draw,inner sep=2pt] (char) {#1};}}

\DeclareMathOperator{\cosech}{cosech}
\DeclareMathOperator{\sech}{sech}
\DeclareMathOperator{\csch}{csch}

\DeclareMathOperator{\cn}{div}

\theoremstyle{plain}
\newtheorem{theorem}{Theorem}[section]
\newtheorem*{theorem*}{Theorem}
\newtheorem{proposition}[theorem]{Proposition}
\newtheorem{lemma}[theorem]{Lemma}
\newtheorem{corollary}[theorem]{Corollary}
\newtheorem{definition}[theorem]{Definition}

\theoremstyle{definition}
\newtheorem{remark}[theorem]{Remark}
\newtheorem*{remark*}{Remark}
\newtheorem{notation}[theorem]{Notation}
\newtheorem*{notation*}{Notation}
\newcommand{\Til}{\mathcal T}
\newcommand{\R}{\mathbb R}
\newcommand{\HS}{S}
\newcommand{\HVS}{\mathbf S}
\newcommand{\HD}{H_D}
\newcommand{\HN}{H_N}
\newcommand{\Tilh}{\mc T_h}
\renewcommand{\P}{{\mathbf P}}

\newcommand{\ST}{\Big\arrowvert_0^T}

\newcommand{\dsigma}{\, d\sigma}
\newcommand{\dalpha}{\, d\alpha}

\newcommand{\dt}{\, d t}
\newcommand{\dx}{\, d x}

\newcommand{\dxdt}{\dx \dt}
\newcommand{\dy}{\, d y}

\newcommand{\dydx}{\, d y \, d x}

\newcommand{\dydxdt}{\, d y \, d x \, d t}

\newcommand\py{\partial_y}

\newcommand{\eps}{\varepsilon}
\newcommand{\evaluate}[1]{\widetilde{{#1}}}
\newcommand{\defn}{\mathrel{:=}}

\def\blA{\bigl\lVert}
\def\brA{\bigr\rVert}
\def\ra{\right\vert}
\def\rA{\right\Vert}
\def\la{\left\vert}
\def\lA{\left\Vert}

\newcommand{\mN}{\mathcal{N}}

\newcommand\mI{\mathcal{I}}

\def\xC{\mathbb{C}}

\def\xR{\mathbb{R}}

\def\xZ{\mathbb{Z}}

\def\uq{\frac{1}{4}}
\def\mez{\frac{1}{2}}

\newcommand{\mc}{\mathcal}

\newcommand{\mbb}{\mathbb}

\newcommand{\mf}{\mathfrak}

\renewcommand{\Re}{\operatorname{Re}}
\renewcommand{\Im}{\operatorname{Im}}
\newcommand{\Rl}{\mbb R}

\newcommand{\Z}{\mbb Z}

\newcommand{\Strip}{S}

\newcommand{\mfH}{{\mf H}}

\numberwithin{equation}{section}
\begin{document}

\title{A Morawetz inequality for water waves}
\author[]{Thomas Alazard, Mihaela Ifrim and Daniel Tataru}

\begin{abstract}
We consider gravity water waves in two space dimensions, with finite or infinite depth.
Assuming some uniform scale invariant Sobolev  bounds for the solutions, we prove local energy decay (Morawetz) estimates globally in time. 
Our result is uniform in the infinite depth limit.
\end{abstract}

\maketitle

%\tableofcontents

%%%%%%%%%%%%%%%%%%%%%%%%%%%%%%%%%%%%%%%%%%%%%%%%%%%%%%%%%%%%%%%%%%%%
\section{Introduction}\label{s:Introduction}%%%%%%%%%%%%%%%%%%%%%%%%%%%%%%%%%%%%%%%%%%%%%%%
%%%%%%%%%%%%%%%%%%%%%%%%%%%%%%%%%%%%%%%%%%%%%%%%%%%%%%%%%%%%%%%%%%%%

Our aim in this paper is to initiate the study of 
Morawetz inequalities for water waves. 
The water-wave equations describe the dynamics of the interface separating air from 
a perfect fluid. This is a system of two coupled equations: the incompressible Euler equation inside the fluid domain, and a kinematic equation 
describing the evolution of the interface. Assuming that the flow is irrotational, 
we thus have two unknowns: the velocity potential~$\phi$, whose gradient gives the velocity, and the free surface elevation~$\eta$, whose graph is the free surface. 

\smallbreak

We consider the 2D-gravity equations, 
without surface tension, and assume that the fluid domain has a flat bottom. 
Then, at time $t$, the 
fluid domain is of the form
$$
\Omega(t)= \{\, (x,y)\in \xR \times \xR\, : \, -h<y<\eta(t,x)\,\},
$$
where $h>0$ is the depth. 
Given a compactly supported bump function $\chi=\chi(x)$, 
we want to estimate the local energy
\[
\frac g 2 \int_0^T \!\!\!\int_\xR\chi(x-x_0) \eta^2(t,x)\, dx dt
+\mez \int_0^T \!\!\!\int_\xR\int_{-h}^{\eta(t,x)} \chi(x-x_0)
\la \nabla_{x,y}\phi(t,x,y) \ra^2 \, dydx dt,
\]
uniformly in time $T$ and space location $x_0$. 

\smallbreak

In the infinite depth case ($h=\infty$), neglecting all nonlinearities, the gravity water-wave equations 
can be written as a fractional Schr\"odinger equation
$$
\partial_t u+i\la D_x\ra^{\mez}u=0,\quad \text{with}\quad 
u=\eta+i\la D\ra^{\mez}\psi.
$$
For this equation, one can obtain a Morawetz inequality by using some standard dispersive tools. 
The first goal of this paper is to extend this linear analysis to the finite depth case, and prove an estimate which is uniform with respect to $h\ge 1$. This problem exhibits some very interesting difficulties at low frequencies, whose analysis requires a careful study of harmonic functions in a strip. 

\smallbreak

The second and main task of this paper is to obtain a 
Morawetz inequality for the nonlinear equations. Our main result extends the linear inequality; it holds provided that some scale invariant norms remain small enough uniformly in time. 
Our nonlinear analysis is highly non-perturbative, since it is a very delicate problem to estimate the nonlinearities by scale invariant norms 
(this can be seen by recalling that one does not even know the existence of weak-solutions in such scale invariant spaces). 

\smallbreak

The proofs combine multiple methods and  ideas in a novel way: $i)$ local conservation laws for the momentum conservation (inspired by Morawetz, and introducing a new momentum density for the water waves equations); 
$ii)$ a systematic use of conformal transformations, $iii)$ 
appropriate Littlewood-Paley decompositions and multilinear estimates to analyze the low-frequency component, $iv)$ a fully nonlinear normal form type modification of the momentum density  to handle the worst nonlinearities.  

In the Appendix, we also complement this analysis 
by showing a Morawetz inequality for possibly large solutions, but at the expense of loosing the uniformity in the depth as well as the control of the low-frequency component of the velocity potential.

\subsection{Morawetz estimates} 
Also known as \emph{local energy decay}, they were originally introduced in Morawetz's paper \cite{Morawetz1968}. In their original form they assert that, for solutions to the linear wave equation,
the local energy of the solutions is bounded, globally in time, by the initial energy. One may view this as a statement about the local decay 
of solutions which is invariant with respect to time translations.

Another interesting example is the Schr\"odinger equation. Unlike the wave equation, where one has a finite speed of propagation, in this case
the group velocity increases to infinity in the high frequency limit.
Because of this, the natural local energy measures a higher regularity
($1/2$ derivative more to be precise) than the initial data energy of the solutions; for this reason the  Morawetz estimates for the Schr\"odinger equation have been originally called \emph{local smoothing}, 
see~\cite{CS-ls},~\cite{Vega}.

 Up to the present time,  the Morawetz estimates have had a rich and complex history, which is too extensive to try to describe here.
 For further references we direct the reader to \cite{MT} for the wave 
 equation, \cite{MMT} for the Schr\"odinger equation and \cite{OR} for 
 other models.  Morawetz estimates  have been proved for linear and nonlinear models,
 and have been used as a key ingredient in many results concerning 
 the long time behavior of solutions in nonlinear dispersive flows.
 One other key development was the introduction of \emph{interaction
 Morawetz estimates} in \cite{CKSTT}, which has played a major role
 in the study of nonlinear Schr\"odinger equations.
 
 We turn our attention now to Morawetz estimates for water waves.
 Here additional challenges arise due to the fact that the equations 
 are not only fully nonlinear, but also nonlocal. Another striking difference is due to the fact 
 that in the high frequency limit the dispersive part of the group velocity goes to zero. Because of this,
 here we have the opposite phenomena to local smoothing, namely a loss
 of $1/4$ derivative in the local energy. Combined with the nonlinear
 and nonlocal character of the equations, this brings substantial difficulties in the low frequency  analysis.

\subsection{The water wave equations}
Consider the incompressible Euler equations for a potential flow 
in a $2D$-domain located between with a free surface and a flat bottom. At time $t$ the 
fluid domain is of the form
$$
\Omega(t)= \{\, (x,y)\in \xR \times \xR\, : \, -h<y<\eta(t,x)\,\},
$$
where $h>0$ is the depth and $\eta$, the free surface elevation, is an unknown function. 
The velocity field is the gradient of a harmonic potential function $\phi=\phi(t,x,y)$, satisfying the Bernoulli equation,
\begin{equation}\label{bernoulli}
\left\{
\begin{aligned}
&\Delta_{x,y}\phi=0\quad\text{in }\Omega(t)\\ 
&\partial_{t} \phi +\mez \la \nabla_{x,y}\phi\ra^2 +P +g y = 0 \quad\text{in }\Omega(t)\\
&\phi_y =0 \quad \text{on }y=-h,
%&\phi_x=0\quad\text{on }x=0 \text{ or }x=L,
\end{aligned}
\right.
\end{equation}
where $g>0$ is the acceleration of gravity, 
$P\colon \Omega\rightarrow\xR$ is the pressure, 
$\nabla_{x,y}=(\partial_x,\partial_y)$ and $\Delta_{x,y}=\partial_x^2+\partial_y^2$. 
Partial differentiations in space are denoted by suffixes so that $\phi_x=\partial_x\phi$ and $\phi_y=\partial_y \phi$.

The water-wave equations are given by two boundary conditions on the free surface: 
firstly an equation describing the deformations of the domain,
\begin{equation}\label{eta}
\partial_{t} \eta = \sqrt{1+\eta_x^2}\, \phi_n \arrowvert_{y=\eta}=\phi_y(t,x,\eta(t,x))-
\eta_x(t,x)\phi_x(t,x,\eta(t,x)),
\end{equation}
and secondly an equation expressing the balance of forces at the free surface. In the present article we consider pure gravity waves, so that this balance of forces reads
\begin{equation}\label{pressure}
P\arrowvert_{y=\eta}=0.
\end{equation}

One can give more explicit evolution equations 
by introducing the trace of the velocity potential at the free surface,
$$
\psi(t,x)=\phi(t,x,\eta(t,x)),
$$ 
as well as the Dirichlet to Neumann operator associated to the fluid domain $\Omega(t)$, defined by
$$
G(\eta)\psi=\sqrt{1+\eta_x^2}\, \phi_n \arrowvert_{y=\eta}=(\phi_y-\eta_x \phi_x)_{\arrowvert_{y=\eta}}.
$$
Then (see \cite{Zakharov1968}), 
with the above notations, the water-wave system reads
\begin{equation}\label{systemT}
\left\{
\begin{aligned}
&\partial_t \eta=G(\eta)\psi \\
&\partial_t \psi+g\eta +\mez \psi_x^2-\mez \frac{(G(\eta)\psi+\eta_x\psi_x)^2}{1+\eta_x^2}=0.
\end{aligned}
\right.
\end{equation}

\subsection{Symmetries and conservation laws}
Introduce the energy $\mathcal{H}$, defined by
\begin{equation}\label{Hamiltonian}
\mathcal{H}=\frac{g}{2}\int_\xR\eta^2\, dx
+\mez\int_\xR\int_{-h}^{\eta(t,x)}\la \nabla_{x,y}\phi \ra^2\, dydx.
\end{equation}
The energy is conserved. Furthermore, 
it is known since Zakharov (\cite{Zakharov1968}) that the water-wave 
system is Hamiltonian. Precisely, we have
\[
\frac{\partial\eta}{\partial t}=\frac{\delta \mathcal{H}}{\delta \psi},\qquad
\frac{\partial\psi}{\partial t}=-\frac{\delta \mathcal{H}}{\delta \eta}.
\]
 
A second conservation law arises by Noether's theorem from the invariance with respect to horizontal translations.
This is the horizontal momentum, which has the form 
\begin{equation}\label{momentum}
\mathcal M = \int_{\R} \eta \psi_x \, dx.
\end{equation}
Together with the energy, this will play a key role in what follows. 

Another symmetry is given by the scaling invariance which holds 
in the infinite depth case (that is when $h=\infty$). 
If $\psi$ and $\eta$ are solutions of the gravity water waves equations \eqref{systemT}, then $\psi_\lambda$ and $\eta_\lambda$ defined by
$$
\psi_\lambda(t,x)=\lambda^{-3/2} \psi (\sqrt{\lambda}t,\lambda x),\quad 
\eta_\lambda(t,x)=\lambda^{-1} \eta(\sqrt{\lambda} t,\lambda x),
$$
solve the same system of equations. The (homogeneous) Sobolev spaces invariant by this scaling correspond to $\eta$ 
in $\dot{H}^{3/2}(\xR)$ and $\psi$ in $\dot{H}^2(\xR)$. 

\subsection{The Cauchy problem} 
The energy, the momentum and the scale invariant norms are super-critical 
for the current local well-posedness results about the Cauchy problem. 
One does not even know the existence of weak-solutions for initial data such that 
these three quantities are finite. 

The local well-posedness for the Cauchy problem with initial data in Sobolev spaces has been extensively studied; we refer the reader to 
\cite{Nalimov,Wu97,Wu99,ChLi,LindbladAnnals,CS,LannesJAMS,SZ,ABZ3,HIT,LannesLivre,ABZ-memoir,Ai}.  The water wave equations are now known to be locally well-posed in suitable function spaces which are $1/2$-derivative\footnote{ Even slightly below that, see \cite{ABZ-memoir,Ai}.} more  regular than the scaling invariance, 
e.g.\ when initially
\[
\eta \in H^s(\xR), \qquad \psi - T_{\phi_y\vert_{ y=\eta} } \eta \in H^{s+\frac12}(\xR), \qquad s \geq 2 , 
\]
where $T_ab$ is the paraproduct decomposition of the product of two functions $a$ and $b$; 
it represents the portion which favours the ``low-high'' interaction when a low-frequency component of $a$ is multiplied with a high-frequency component of $b$. Here the expression $\psi - T_{\phi_y\vert_{ y=\eta}}\eta$ represents the so called \emph{good unknown} of Alinhac (\cite{Ali,Alipara,LannesJAMS,AM}), and is imposed by the non-diagonal quasilinear structure of the equations. Alternatively, one can re-express the second condition in terms of 
the gradient of the velocity potential, namely by requiring that 
$\nabla \phi\arrowvert_{y=\eta}$ belongs to $H^{s-\frac12}(\xR)$. 

Since we are interested in uniform in time estimates, let us recall that 
much less is known concerning the long time dynamics. For data of size $\epsilon$ it is known that solutions persist for at least a cubic  
lifespan $O(\epsilon^{-2})$, see\footnote{As  a historical note, the question of obtaining cubic lifespan bounds first arose in the work of Zakharov~\cite{Zakharov1968} in the context of the NLS approximation for deep water waves; see also \cite{TW} for more recent results in this direction.} \cite{AD-Sobolev,HIT} for the deep water case and 
\cite{H-GIT} for the finite depth case (see also \cite{Wang-3D-finite} for the 3D problem). 
For longer times it is not at all
clear what happens to the solutions, and the blow-up scenario 
in particular has not been excluded (see \cite{CCFGGS,CS-blowup} for large data blow-up). An exception to this is the case when the initial data is not only small but also localized, where there solutions are known to be global, see \cite{Wu2009,IoPu,AD-global,HIT, IT} and also similar results in three dimensions \cite{GMS,Wu2011}.

Rather than trying to study the size of the solutions for longer time, in the present article we take a different track, and assume that we have a solution which stays bounded (small) in a reasonable Sobolev norm on a time interval $[0,T]$, with no a-priori bound on $T$, and ask what can be said about the dispersive properties of the solutions. More precisely, our goal here is to initiate the study of 
Morawetz inequalities for water waves. We consider the case of gravity waves in the present article, 
and the case of gravity-capillary waves in a second article.

\subsection{Function spaces} 
In this paragraph we introduce three spaces: 
a space $E^0$ associated to the energy, 
a space $E^{\frac14}$ associated to 
the momentum, and a uniform in time control norm $\lA \cdot\rA_{X}$ 
which respects  the scaling invariance. 

The above energy $\mathcal{H}$ (Hamiltonian) corresponds to the energy space for $(\eta,\psi)$,
\[
E^0 = g^{-\frac12} L^2(\xR) \times \dot{H}^{\frac12}_h(\xR),
\]
with the depth dependent $H^{\frac12}_h(\xR)$ space defined as
\[
\dot{H}^{\frac12}_h(\xR) = \dot H^\frac12(\xR) + h^{-\frac12} \dot H^1(\xR).
\]

Similarly, in order to measure the momentum, we use the space
$E^{\frac14}$, which is the $h$-adapted linear $H^{\frac14}$-type norm for $(\eta,\psi)$ (which corresponds to the momentum),
\[
\begin{aligned}
&E^{\frac14} := g^{-\frac14} H^{\frac14}_h(\xR) \times g^\frac14   \dot H^{\frac34}_h(\xR)  \\
\end{aligned}
\]
with
\[
H^{\frac14}_h(\xR) := \dot H^\frac14(\xR) \cap  h^{\frac14} L^2(\xR), \qquad   \dot H^{\frac34}_h(\xR) = \dot{H}^\frac34 (\xR)+  h^{-\frac14} \dot H^1(\xR).
\]
For our uniform a-priori bounds for the solutions, ideally one would like to use 
a scale invariant norm, which would correspond to the following Sobolev bounds:
\[
\eta \in H^\frac32_h(\xR), \qquad \nabla \phi\arrowvert_{y=\eta} \in H^1_h(\xR).
\]
Our uniform control norm, denoted by $X$, nearly matches the above ideal scenario. Precisely, we define the homogeneous norm $X_0$ by
\[
X_0 := L^\infty_t H^{\frac{3}{2}}_{h} \times g^{-\mez}L^\infty_t H^1_h,
\]
and then set
\[
\|(\eta,\psi)\|_{X} := \lA P_{\leq h^{-1}}
(\eta,\nabla \phi\arrowvert_{y=\eta})\rA_{X_0}+
\sum_{h^{-1} \leq \lambda\in 2^\xZ} \|P_\lambda (\eta,\nabla \phi\arrowvert_{y=\eta})\|_{X_0}.
\]
Here we use a standard Littlewood-Paley decomposition beginning at frequency $1/h$,
\[
1 = P_{<1/h} + \sum_{1/h < \lambda \in 2^\Z} P_\lambda.
\]

% \blue{OMIT THIS ?
% Namely, we fix a smooth bump function $\varphi$ supported in the annulus $\{ 1/2\leq \la \xi\ra\leq 2\}$ and such that $\sum_{j\in \xZ}\varphi(2^{-j}\xi)=1$ for all $\xi$ in $\xR^*$. 
% Then, given a function $f$ and dyadic number $\lambda\in 2^{\xZ}$, we define 
% $f_\lambda=P_\lambda f$ by $\hat{f_\lambda}=\varphi(\lambda^{-1} \xi)\hat{f}(\xi)$. Then we set 
% $f_{<\lambda}=\sum_{\lambda'<\lambda} f_{\lambda'}$.

% where $\lambda\in 2^{\xZ}$ is a dyadic summation index, $P_\lambda=\varphi(\lambda^{-1}D_x)$ is a Littlewood-Paley Fourier multiplier, where 
% $\varphi \in C^\infty_0(\xR\setminus\{0\})$ satisfies $\varphi(\xi)=1$ when $1/2\leq \la \xi\ra\leq 1$, and $P_{\leq h^{-1}}=\sum_{\lambda<h^{-1}}P_\lambda$.}

Based on the expression \eqref{Hamiltonian} for the
energy, we introduce the following notations for the local energy. Fix an 
arbitrary compactly supported nonnegative function $\chi$. 
Then,  the local energy 
centered around a point $x_0$ is
\[
\| (\eta,\psi)\|_{LE_{x_0}}^2 :=  g \int_0^T \int_\xR\chi(x-x_0) \eta^2\,dx\, dt
+\int_0^T \int_\xR\int_{-h}^{\eta(t,x)} \chi(x-x_0) \la \nabla_{x,y}\phi \ra^2\, dy\, dx\, dt.
\]
%and in the gravity-surface tension case is
%\[
%\| (\eta,\psi)\|_{LE^{\kappa}_{x_0}}^2 := \| (\eta,\psi)\|_{LE_{x_0}}^2 + k \int_0^T \int_\xR \chi(x-x_0)\eta_{x}^2\dx dt.
%\]
%Here $\chi$ is a smooth and positive, compactly supported bump function.
It is also of interest to take the supremum over $x_0$,
\[
\| (\eta,\psi)\|_{LE}^2 :=  \sup_{x_0\in \xR} \| (\eta,\psi)\|_{LE_{x_0}}^2.
\]

\subsection{Main result}
Our main result for gravity water waves is as follows:

\begin{theorem}[Local energy decay for gravity waves]\label{ThmG1}
Let $s>5/2$. There exist $\epsilon_0$ and $C_{0}$ such that the following result holds. 
For all $T\in (0,+\infty)$, all 
$g\in (0,+\infty)$,  all $h\in [1,+\infty)$ and all 
solutions $(\eta,\psi)\in C^0([0,T];H^{s}(\R)\times H^s(\R))$ of the water-wave system \eqref{systemT} satisfying
\begin{equation}\label{uniform}
\| (\eta,\psi)\|_{X} \leq \epsilon_0
\end{equation}
the following estimate holds
\begin{equation}\label{nMg}
\begin{aligned}
\| (\eta,\psi)\|_{LE}^2  \leq C_0 ( \| (\eta, \psi) (0) \|^2_{E^{\frac14}}+ \| (\eta, \psi) (T) \|^2_{E^{\frac14}}).
\end{aligned}
\end{equation}
\end{theorem}

We continue with several remarks concerning the choices of parameters/norms in the theorem.

\begin{remark}
One key feature of our result is 
that it is global in time (uniform in $T$) and uniform in $h \geq 1$. In particular our estimate is uniform in the infinite depth limit.
\end{remark}

\begin{remark}
Another feature of our result is that the statement of Theorem ~\ref{ThmG1} is invariant with respect to the following scaling law  (time associated scaling)  
\[
\begin{aligned}
(\eta (t,x), \psi(t,x)) &\rightarrow (\eta (\lambda t,x), \lambda \psi(\lambda t,x))\\
(g,h) &\rightarrow (\lambda^2 g, h).
\end{aligned}
\]
This implies that the value of $g$ is not important. By scaling one 
could simply set it to $1$ in all the proofs. We do not do that
in order to improve the readability of the article.
\end{remark}

\begin{remark}
As already explained, the uniform control norms in  \eqref{uniform} are below the current local well-posedness threshold for this problem, and are instead what one might view as the 
critical, scale invariant norms for this problem. The dependence on $h$ 
is natural as spatial scaling will also alter the depth $h$. In the infinite depth limit one recovers exactly the homogeneous Sobolev norms.
We also note that, by Sobolev embeddings, our smallness assumption
guarantees that 
\[
|\eta| \lesssim \epsilon_0 h, \qquad |\eta_x | \lesssim \epsilon_0.
\]
\end{remark}

\begin{remark}
The constraint $h \geq 1$ is due to the window size of $1$ in the local energy norm. Of course, once a local energy estimate is obtained
for a window size, the similar bound for all larger window sizes 
also follow. Then bounds for $h < 1$ or for smaller window sizes can also be achieved by scaling; however, the uniformity in $h$ will be lost.
\end{remark}

\begin{remark}
In Appendix~\ref{appendix:nonlinear}, we complement this result 
by showing a similar estimate 
for possibly large solutions (satisfying a smallness assumption which is milder than \eqref{uniform}), but at the expense of loosing the uniformity in the depth as well as the control of the low-frequency component of the velocity potential.
\end{remark}

As in Morawetz' s original paper (\cite{Morawetz1968}), we will obtain
these results by using the multiplier method, based on the momentum conservation law. When doing this, we encounter two difficulties:

\begin{itemize}
\item {\bf High frequency issues} which are due to the fact that 
our equations are quasilinear. 

\item{\bf Low frequency issues} which are due both to the fact that the 
equations are nonlocal, and that they have quadratic nonlinearities.
\end{itemize}

Of these, the low frequency issues are far more delicate. To approach
them we use both the Eulerian coordinates and the holomorphic coordinates. The latter will provide a better setting to understand the fine bilinear and multilinear structure of the equations.

\subsection{Plan of the paper}
In the next section, we review density flux pairs for the momentum.
The density $\eta\psi_x$ implicit in \eqref{momentum} only allows one to control the 
local potential energy, while for the local kinetic energy we introduce
an alternate density and the associated flux.

To exploit the density flux identities we need a good understanding of the Dirichlet problem in a strip, which in turn leads us to the 
holomorphic coordinates. This is discussed in the following section, which also provides the formulation of the equations in holomorphic coordinates and reviews the correspondence between the two settings.

In Section~\ref{s:linear} we use the quadratic versions of the 
above density flux pairs in order to prove the local energy bounds
for the corresponding linear flow. This will be later used to handle
the leading, quadratic part of the nonlinear identities.

Finally, in the last sections we use the nonlinear density flux pairs
to prove the local energy decay bounds in the theorem. Here we 
use the linear analysis for the main quadratic terms, and the bulk of the work is devoted to estimates for the cubic and higher terms.
It is there that a delicate analysis is required in order to handle both the low 
and high frequency contributions. 
In particular, the worst such contribution turns out to be unbounded. However, we discovered that this error term can be balanced using a carefully chosen nonlinear normal form type correction to the momentum density.

\subsection*{Acknowledgements} 
This research was initiated at IH{\'E}S in the spring 2016 during a Trimester on nonlinear waves.  The authors thank the organizers of this semester and the IH{\'E}S for their hospitality. The first author thanks the University of California-Berkeley and the last two authors thank the {\'E}cole Normale Sup{\'e}rieure de Cachan and the University of Orsay, 
where part of this research was carried out. 
The first author was partially supported by the grant ÒANA{\'E}Ó ANR-13-BS01-0010-03. The second author was partially supported by a Clare Boothe Luce Professorship. The third author was partially supported by the NSF grant  DMS-1266182 as well as by a Simons Investigator grant from the  Simons Foundation.

%%%%%%%%%%%%%%%%%%%%%%%%%%%%%%%%%%%%%%%%%%%%%%%%%%%%%%%%%%%%%%%%%%%%
\section{Conservation of momentum and local conservation laws}\label{s:Cl}%%%%%%%%%%%%%%%%%%%%%%%%%%%%%%
%%%%%%%%%%%%%%%%%%%%%%%%%%%%%%%%%%%%%%%%%%%%%%%%%%%%%%%%%%%%%%%%%%%%

This section contains several formal identities related to the conservation of momentum (these formal computations will be made rigorous later on). We prove that there are certain momentum densities that one can use in defining the horizontal momentum. Clever manipulations of such quantities will lead to control of the kinetic and potential components of the local energy.

The fact that the momentum is a conserved quantity comes from the fact
that the problem is invariant with respect to horizontal translation (see Benjamin
and Olver \cite{BO} for studies of the invariants and symmetries of
the water-wave equations). 
To exploit the conservation of the momentum we will use density flux pairs 
$(I,S)$ which by definition must satisfy 
\begin{equation}\label{density}
\mathcal{M}=\int I \, dx ,
\end{equation}
and also the conservation law
\begin{equation}\label{cl}
\partial_t I+\partial_x S=0.
\end{equation}
In what follows $m(x)$ is a positive increasing function. 
Multiplying the identity  \eqref{cl} by $m=m(x)$, integrating 
over $[0,T]\times \xR$ and then integrating by parts yields
\[
\iint_{[0,T]\times \xR} S(t,x) m_x \, dxdt =\int_\xR m(x)I(T,x)\, dx-\int_\xR m(x)I(0,x)\, dx.
\]
Since $m_x$ is nonnegative, the above identity is favorable provided 
that $S$ is also nonnegative.

We begin by writing the momentum as an integral over the whole water domain
\[
\mathcal{M}(t)=\int_\xR\int_{-h}^{\eta(t,x)}\phi_x (t,x,y)\, dy dx.
\]
The next lemma shows that the momentum is an invariant. As already
mentioned, this is a well-known result. For the sake of completeness,
we give, following Longuet-Higgins \cite{LH1983}, a formal
computational proof of the conservation of momentum which is linked to
other computations made below.  We also give three different expressions
for a possible choice of the momentum density. 

\begin{lemma}
We have
$$
\frac{d}{dt} \mathcal{M}=0.
$$
\end{lemma}
\begin{proof}
Set $\Omega(t)=\{(x,y)\ : \ -h<y<\eta(t,x)\}$. To prove this result we first check that, for any function $f=f(t,x,y)$, one has
$$
\frac{d}{dt} \iint_{\Omega(t)} f(t,x,y)\, dy dx
=\iint_{\Omega(t)} (\partial_t +\nabla_{x,y}\phi\cdot \nabla_{x,y})f \, dy dx.
$$
Indeed,
\begin{align*}
\int_\xR (\partial_t \eta)f(t,x,\eta)\, dx &=\int_\xR (\partial_n\phi) f(t,x,\eta) 
\sqrt{1+\eta_x^2}\, dx\\
&=\int_{\partial\Omega(t)}n \cdot (f\nabla_{x,y}\phi) \, d \sigma\\
&=\iint_{\Omega(t)} \cn_{x,y}(f\nabla_{x,y}\phi)\, dydx\\
&=\iint_{\Omega(t)}\nabla_{x,y}\phi\cdot \nabla_{x,y}f \, dydx.
\end{align*}
By applying the previous identity with $f=\phi_x$, we deduce that
$$
\frac{d}{dt} \mathcal{M}=\iint_{\Omega(t)}(\partial_t +\nabla_{x,y}\phi\cdot \nabla_{x,y})\phi_x\, dydx,
$$
so
$$
\frac{d}{dt} \mathcal{M}
=\iint_{\Omega(t)}\partial_x(\partial_t \phi+\mez\la \nabla_{x,y}\phi\ra^2)\, dydx
=-\iint_{\Omega(t)}\partial_x P\, dydx.
$$
Now, we have
$$
\iint_{\Omega(t)}\partial_x P\, dydx=\int_\xR\partial_x \Big( \int_{-h}^{\eta(t,x)}P\, dx \Big)\, dx
-\int \eta_x P\arrowvert_{y=\eta}\, dx=0,
$$
where we used the boundary condition $P\arrowvert_{y=\eta}=0$. This gives the wanted result. 
\end{proof}

In addition to the conservation of momentum, one has local
conservation laws of the form \eqref{cl},  which imply the conservation of momentum. The study of
these conservation laws for water waves was initiated by Benjamin and
we refer to his broad survey paper about impulse conservation in
\cite{Benjamin1984}. Here  we discuss density-flux pairs $(I,S)$ 
for the momentum. These are not unique, and 
in effect there are three such pairs that play a role in our work. The first two pairs are well known in fluid mechanics, but the third one is new 
to the best of our knowledge. 

\begin{lemma}
The expression 
\[
I_1(t,x)=\int_{-h}^{\eta(t,x)}  \phi_x(t,x,y) \, dy,
\]
is a density for the momentum, with associated density flux
\[
S_1(t,x)\defn -\int_{-h}^{\eta(t,x)}\partial_t \phi\, dy -\frac{g}{2}\eta^2+\mez \int_{-h}^{\eta(t,x)}(\phi_x^2-\phi_y^2)\, dy.
\]
\end{lemma}
\begin{proof}
This result follows from the study by Benjamin in \cite{Benjamin1984}.  We give here another proof. 

From now on, given a function $f=f(t,x,y)$, we denote by $\evaluate{f}$ the function 
$$
\evaluate{f}(t,x)=f(t,x,\eta(t,x)).
$$
With this notation, one has
$$
\partial_t I_1=\partial_t \int_{-h}^{\eta}\phi_x\, dy=(\partial_t \eta)\evaluate{\phi_x}+\int_{-h}^\eta \partial_t \phi_x \, dy.
$$
Using the equations for $\eta$ and the velocity $\phi_x$ this gives that
$$
\partial_t I_1=(\evaluate{\phi_y}-\eta_x\evaluate{\phi_x})\evaluate{\phi_x}
-\int_{-h}^{\eta}\partial_x \left(\mez \la\nabla_{x,y}\phi\ra^2+ P\right)\, dy.
$$
We deduce that
$$
\partial_t I_1=\evaluate{\phi_y}\evaluate{\phi_x}+\mez \eta_x\evaluate{\phi_y^2}
-\mez\eta_x\evaluate{\phi_x}^2+\eta_x \evaluate{P}
-\partial_x \int_{-h}^{\eta}\left(\mez \la\nabla_{x,y}\phi\ra^2+ P\right) \, dy.
$$
Using the two equations for the pressure (in the fluid domain and at the free surface), 
we conclude that
$$
\partial_t I_1=\evaluate{\phi_y}\evaluate{\phi_x}+\mez \eta_x\evaluate{\phi_y^2}
-\mez\eta_x\evaluate{\phi_x}^2
+\partial_x \int_{-h}^{\eta}\left(\partial_t \phi+gy\right) \, dy.
$$
Since $\partial_x \int_{-h}^{\eta}gy\, dy=\partial_x(g\eta^2/2)$, to conclude the proof, it remains only to check that
$$
\evaluate{\phi_y}\evaluate{\phi_x}+\mez \eta_x\evaluate{\phi_y^2}
-\mez\eta_x\evaluate{\phi_x}^2=\mez\partial_x \int_{-h}^{\eta}(\phi_y^2-\phi_x^2)\, dy.
$$
This can be verified by a direct computation, noticing that
$$
\int_{-h}^\eta(\phi_x\phi_{yx}-\phi_x\phi_{xx})\, dy=
\int_{-h}^\eta(\phi_x\phi_{yx}+\phi_x\phi_{yy})\, dy=
\int_{-h}^\eta\partial_y (\phi_x\phi_y)\, dy=\evaluate{\phi_x}\evaluate{\phi_y},
$$
where we used the equations for $\phi$ to get $\phi_{xx}=-\phi_{yy}$ 
and $\phi_y(t,x,-h)=0$.
\end{proof}

%The above density-flux pair will not be directly useful because it has a linear component
%in it. However, we will use it as a springboard for the next two density-flux pairs:

\begin{lemma}\label{l:I2}
The expression 
\[
I_2(t,x)=\eta(t,x)\psi_x(t,x)
\]
is a density for the momentum, with associated density flux
\[
S_2(t,x)\defn -\eta\psi_t-\frac{g}{2}\eta^2+
\mez \int_{-h}^{\eta(t,x)}(\phi_x^2-\phi_y^2)\, dy.
\]
\end{lemma}
\begin{proof}
We write that
\begin{align*}
I_1&=\int_{-h}^{\eta(t,x)}\phi_x(t,x,y)\, dy=\partial_x\int_{-h}^{\eta(t,x)}\phi\, dy -\eta_x\psi\\
&=\partial_x \left(\int_{-h}^{\eta(t,x)}\phi\, dy -\eta\psi\right)+\eta\psi_x\\
&=\partial_x \left(\int_{-h}^{\eta(t,x)}\phi\, dy -\eta\psi\right)+I_2.
\end{align*}
This immediately implies that
$\mathcal{M}=\int I_1\, dx =\int I_2\, dx$ and 
$$
\partial_t I_2=\partial_t I_1-\partial_x \left(\int_{-h}^{\eta(t,x)}\partial_t\phi\, dy -\eta\psi_t\right)=-\partial_x\left(S_1+\int_{-h}^{\eta(t,x)}\partial_t\phi\, dy -\eta\psi_t\right),
$$
so that the wanted expression for $S_2$ can be deduced from the previous lemma.
\end{proof}

To define the third pair we introduce two auxiliary functions as follows (we shall later rigorously justify that these functions are well-defined). The function $q$, 
defined inside the fluid domain, is the stream function, or the harmonic conjugate of $\phi$,
and satisfies
\begin{equation}\label{defi:qconj}
\left\{
\begin{aligned}
& q_x = -\phi_y , \quad \text{in }-h<y<\eta(t,x),\\
& q_y = \phi_x, \qquad \text{in }-h<y<\eta(t,x),\\
& q(t,x,-h) = 0.
\end{aligned}
\right.
\end{equation}
The function $\theta$ is the harmonic extension of $\eta$ with Dirichlet boundary condition on the bottom:
\begin{equation}\label{defi:thetainitial}
\left\{
\begin{aligned}
& \Delta_{x,y}\theta=0\quad \text{in }-h<y<\eta(t,x),\\
&\theta(t,x,\eta(t,x))=\eta(t,x),\\
&\theta(t,x,-h) = 0.
\end{aligned}
\right.
\end{equation}
Now the following lemma states that there is another natural density/flux pair for the momentum. 
\begin{lemma}\label{l:I3}
The expression 
\[
I_3(t,x)= \int_{-h}^\eta 
\nabla \theta(t,x,y) \cdot\nabla q(t,x,y) \, dy
\]
is a density for the momentum, with associated density flux
\[
S_3(t,x)\defn  -\frac{g}{2}\eta^2 - \int_{-h}^{\eta(t,x)} \theta_y \phi_t \, dy   
+  \int_{-h}^{\eta(t,x)} \Big(\mez (\phi_x^2 - \phi_y^2)+ \theta_t \phi_y \Big)\, dy.
\]
\end{lemma}
\begin{proof}
We write 
\[
\nabla \theta \cdot\nabla q = \partial_x (\theta q_x) + \partial_y (\theta q_y),
\]
and integrate in $y$,
\[
\begin{aligned}
\int_{-h}^{\eta (x,t)} \nabla \theta \cdot\nabla q \, dy &= 
\int _{-h}^{\eta (x,t)}\big(  \partial_x (\theta q_x) + \partial_y (\theta q_y)\big) 
\, dy \\
&= \partial_x \left(\int_{-h}^{\eta (x,t)} 
 \theta q_x\, dy\right) - \eta_x \evaluate{\theta q_x} + 
 \evaluate{\theta q_y},
\end{aligned}
\]
where we recall that, given $f=f(t,x,y)$, we set 
$\evaluate{f}(t,x)=f(t,x,\eta(t,x))$.

Now we notice that 
$$
-\eta_x \evaluate{q_x} +  \evaluate{q_y}=\eta_x \evaluate{\phi_y} +  \evaluate{\phi_x}=\partial_x \evaluate{\phi}=\psi_x,
$$
so, recalling that $\evaluate{\theta}=\eta$, we conclude that 
\[
I_3 = I_2 + \partial_x \int \theta q_x \, dy.
\]
Hence $I_3$ is also a momentum density. Further, its flux
is
\[
S_3 = S_2 - \partial_t \int_{-h}^{\eta(t,x)} \theta q_x\, dy.
\]
We further expand the last time derivative,
\begin{align*}
 \partial_t \Big(\int_{-h}^{\eta(t,x)} \theta q_x \, dy\Big) &=   \eta_t \evaluate{\theta q_x} + \int_{-h}^{\eta(t,x)} \big(-\theta_t \phi_y + \theta q_{xt}\big) \, dy\\
 &=  -\eta_t \eta \evaluate{\phi_y} 
 + \int_{-h}^{\eta(t,x)} \big(-\theta_t \phi_y - \theta \phi_{yt}\big) \, dy
\\
  &=    - \eta_t \eta \evaluate{\phi_y} - \eta \evaluate{\phi_t}   + \int_{-h}^{\eta(t,x)} \big(-\theta_t \phi_y + \theta_y \phi_{t}\big)\, dy
\\
  &=    -\eta \psi_t    + \int_{-h}^{\eta(t,x)} \big(-\theta_t \phi_y + \theta_y \phi_{t}\big) \, dy.
\end{align*}
The conclusion of the lemma easily follows.
\end{proof}

\subsection{The expressions \texorpdfstring{$\phi_t$}{} and \texorpdfstring{$\theta_t$}{}}
Here we provide a better description of the functions $\theta_t$ and $\phi_t$
arising in the last momentum flux $S_3$. For that we introduce two bounded operators,
$\HD$ and $\HN$, which act on functions on the top and produce their harmonic extension within the fluid domain with zero Dirichlet, respectively Neumann
boundary condition on the bottom (we shall explain 
later on that these operators are defined on a space large enough to contain all the functions that we shall encounter, namely they 
are well defined on uniformly local $L^2$ spaces). As an example of the usage of these notations, we have
\[
\theta = \HD(\eta), \qquad \phi= \HN(\psi),
\]
which means that
\begin{equation*}
\left\{
\begin{aligned}
&\Delta \theta=0 \quad \text{in }-h<y<\eta,\\
&\theta\arrowvert_{y=\eta}=\eta,\\
&\theta\arrowvert_{y=h}=0,
\end{aligned}
\right.
\qquad
\left\{
\begin{aligned}
&\Delta \phi=0 \quad \text{in }-h<y<\eta,\\
&\phi\arrowvert_{y=\eta}=\psi,\\
&\partial_y\phi\arrowvert_{y=h}=0.
\end{aligned}
\right.
\end{equation*}

Recall that, given a function $f=f(t,x,y)$, we set 
$\evaluate{f}(t,x):=f(t,x,\eta(t,x))$.
\begin{lemma}
The function $\phi_t$ is harmonic in the fluid domain, with homogeneous
Neumann boundary condition on the bottom, and can be represented as 
\begin{equation}\label{phit}
\phi_t = -g \HN(\eta) -  \HN\left(\evaluate{|\nabla \phi|^2}\right).
\end{equation}
The function $\theta_t$ is harmonic in the fluid domain, with homogeneous
Dirichlet boundary condition on the bottom, and can be represented as 
\begin{equation}\label{thetat}
\theta_t = \phi_y -  \HD\left(\evaluate{\nabla \theta \cdot\nabla \phi}\right).
\end{equation}
\end{lemma}
\begin{proof}
The equation for $\phi$ follows directly from the Bernoulli equation. To compute the equation for $\theta$, we write that, 
on the top $\{y=\eta(t,x)\}$, 
\[
\evaluate{\theta_t} = \eta_t \big( 1 - \evaluate{\theta_y}\big), \qquad 
\evaluate{\theta_x} = \eta_x\big(1-\evaluate{\theta_y}\big).
\]
Then we deduce that
\begin{equation}
\evaluate{\theta_t-\phi_y} = - \eta_x \evaluate{\phi_x} \big( 1-\evaluate{\theta_y}\big) - 
\evaluate{\phi_y \theta_y}  = - 
\evaluate{\theta_x  \phi_x} 
- \evaluate{\theta_y  \phi_y}
= - \evaluate{\nabla \theta \cdot \nabla \phi}.
\end{equation}
Since $\theta_t-\phi_y$ vanishes on the bottom $\{y=-h\}$, this implies that $\theta_t-\phi_y$ 
is the harmonic extension with the Dirichlet boundary condition of $\nabla\theta\cdot\nabla \phi$.
\end{proof}

%%%%%%%%%%%%%%%%%%%%%%%%%%%%%%%%%%%%%%%%%%%%%%%%%%%%%%%%%%%%%%%%%%%%
%%%%%%%%%%%%%%%%%%%%%%%%%%%%%%%%%%%%%%%%%%%%%%%%%%%%%%%%%%%%%%%%%%%%
%%%%%%%%%%%%%%%%%%%%%%%%%%%%%%%%%%%%%%%%%%%%%%%%%%%%%%%%%%%%%%%%%%%%
%%%%%%%%%%%%%%%%%%%%%%%%%%%%%%%%%%%%%%%%%%%%%%%%%%%%%%%%%%%%%%%%%%%%
%%%%%%%%%%%%%%%%%%%%%%%%%%%%%%%%%%%%%%%%%%%%%%%%%%%%%%%%%%%%%%%%%%%%
%%%%%%%%%%%%%%%%%%%%%%%%%%%%%%%%%%%%%%%%%%%%%%%%%%%%%%%%%%%%%%%%%%%%
\section{Holomorphic coordinates}%%%%%%%%%%%%%%%%%%%%%%%%%%%%%%%%%%%%%%%%%%%%%%%%%%%
%%%%%%%%%%%%%%%%%%%%%%%%%%%%%%%%%%%%%%%%%%%%%%%%%%%%%%%%%%%%%%%%%%%%
%%%%%%%%%%%%%%%%%%%%%%%%%%%%%%%%%%%%%%%%%%%%%%%%%%%%%%%%%%%%%%%%%%%%
%%%%%%%%%%%%%%%%%%%%%%%%%%%%%%%%%%%%%%%%%%%%%%%%%%%%%%%%%%%%%%%%%%%%
%%%%%%%%%%%%%%%%%%%%%%%%%%%%%%%%%%%%%%%%%%%%%%%%%%%%%%%%%%%%%%%%%%%%
%%%%%%%%%%%%%%%%%%%%%%%%%%%%%%%%%%%%%%%%%%%%%%%%%%%%%%%%%%%%%%%%%%%%
%%%%%%%%%%%%%%%%%%%%%%%%%%%%%%%%%%%%%%%%%%%%%%%%%%%%%%%%%%%%%%%%%%%%

\subsection{Harmonic functions in the canonical domain}
\label{ss:laplace}
Here we discuss two classes of harmonic functions in the horizontal strip \(\Strip = \Rl\times(-h,0)\). 

We start by considering solutions to the homogeneous Laplace equation  with 
homogeneous Neumann boundary condition on the bottom,
\begin{equation}\label{MixedBCProblem}
\begin{cases}
\Delta u = 0\qquad \text{in}\ \ \Strip \\[1ex]
u(\alpha,0) = f\\[1ex]
\partial_\beta u(\alpha,-h) = 0.
\end{cases}
\end{equation}
The solution may be written in the form
\[
u(\alpha,\beta)= P_N(\beta, D)f(\alpha): = \frac{1}{{2\pi}}\int p_N(\xi,\beta)\hat
f(\xi)e^{i\alpha\xi}\, d\xi,
\]
where the Fourier multiplier symbol $p_N$ is given by %
\[
p_N(\xi,\beta) = \frac{\cosh((\beta+h)\xi)}{\cosh(h\xi)}.
\]

We are also  interested in the  Dirichlet to Neumann map $\mathcal{D}_N$ defined by
$$
\mathcal{D_N}f=u_{\beta}(\cdot, 0).
$$
This is closely related to the  Tilbert transform, defined by the
formula
\begin{equation}\label{def:Tilbert}
\Tilh f(\alpha) =
-\frac{1}{2h}\lim\limits_{\epsilon\downarrow0}\int_{|\alpha-\alpha'|>\epsilon}
\cosech\left(\frac\pi{2h}(\alpha-\alpha')\right)f(\alpha')\, d\alpha',
\end{equation}
or equivalently, given by the Fourier multiplier 
\[
\Tilh = -i\tanh(hD).
\]
We remark that it
takes real-valued functions to real-valued functions. We denote the
inverse Tilbert transform by \(\Tilh^{-1}\); a-priori this is only defined modulo constants. 

With this notation the Dirichlet to Neumann map for the problem \eqref{MixedBCProblem} is given by
\[
\mathcal{D}_N f=\Tilh \partial_{\alpha} f.
\]

\bigskip

We will also need to consider to similar problem with the Dirichlet boundary condition on the bottom
\begin{equation}\label{DirichletProblem}
  \begin{cases}
    \Delta v = 0\qquad \text{in}\ \ \Strip \\[1ex]
    v(\alpha,0) = g \\[1ex]
     v(\alpha,-h) = 0.
  \end{cases}
\end{equation}
The solution may be written in the form
\[
v(\alpha,\beta)=P_D(\beta, D)g(\alpha) := \frac{1}{{2\pi}}\int p_D(\xi,\beta)\hat
g(\xi)e^{i\alpha\xi}\, d\xi,
\]
where the Fourier multiplier symbol $p_D$ is given by %
\[
p_D(\xi,\beta) = \frac{\sinh((\beta+h)\xi)}{\sinh(h\xi)}.
\]
The Dirichlet to Neumann map $\mathcal{D}_D$ for this problem is given by
\[
v_{\beta} (\alpha ,0)=\mathcal{D}_D g=-\Til_h^{-1}\partial_{\alpha} g.
\]

\bigskip

The solutions to the two problems \eqref{MixedBCProblem} 
and \eqref{DirichletProblem} can be related via harmonic conjugates.
Precisely, given a real-valued solution $u$ to \eqref{MixedBCProblem},
there exists a unique solution $v$ to \eqref{DirichletProblem},
which is harmonic conjugate to $u$, i.e., 
satisfying the Cauchy-Riemmann equations
\[
\left\{
\begin{aligned}
&u_\alpha = -v_\beta \\
&u_\beta = v_\alpha \\
&\partial_\beta u(\alpha,-h)=0.
\end{aligned}
\right.
\]
The Dirichlet data $g$ for $v$ on the top is determined by the 
Dirichlet data $f$ for $u$ on the top via the relation
\[
g = -\Til_h f.
\]
Conversely, given $v$ we could seek a corresponding harmonic conjugate
$u$. The difference in this case is that $u$ will only be uniquely 
determined modulo real constants.

\subsection{A parabolic estimate for harmonic functions}

We are interested in estimates of harmonic functions on vertical lines in terms of the Dirichlet data on the top. 
These are parabolic type estimates for solutions of these elliptic equations. To introduce these estimates, let us consider the Laplace equation in the half space:
\[
\Delta v=0 \text{ in }\beta<0, \quad v\arrowvert_{\beta=0}=g.
\]
By considering the Fourier transform in $\alpha$, one obtains that 
\[
\partial_\beta^2 \hat{v}-\vert \xi\vert ^2\hat{v}=0,
\]
so 
\[
\hat{v}=A(\xi)e^{\beta\la \xi\ra}+B(\xi)e^{-\beta\la\xi\ra}.
\]
Since $\beta<0$, one has necessarily $B(\xi)=0$ so we deduce that $v$ solves a parabolic equation (we see $\beta$ as a time variable)
$$
\partial_\beta v-\la D_\alpha\ra v=0 \quad\text{in }\beta<0, \quad v\arrowvert_{\beta=0}=g.
$$
This is a backward parabolic equation. Namely, the function $w(\alpha,\beta )=v(\alpha,-\beta)$ satisfies $\partial_\beta w +\la D_\alpha\ra w=0$. 
Now, if we perform a standard energy estimate, multiplying the equation by $v$, one obtains that
$$
\lA v(\cdot,\beta)\rA_{L^2_\alpha}^2
+2\int_{\beta}^0 \blA \la D_\alpha\ra^{1/2}v(\cdot,\beta')\brA_{L^2_\alpha}^2\, d \beta'=
\lA g\rA_{L^2_\alpha}^2.
$$
By letting $\beta$ go to $-\infty$, we conclude that
$$
2\lA v\rA_{L^2_\beta\dot{H}^{\frac 12}_\alpha}^2\leq \lA g\rA_{L^2_\alpha}^2.
$$
The following lemma improves this inequality in several directions: it allows to control the $L^\infty$-norm instead of the $\dot{H}^{1/2}$-norm, it allows to consider 
initial data in $H^s$, and it gives a result that is uniform with respect to the depth. 
Our main estimate is as follows
\begin{proposition}
\label{p:extension}
$i)$ Let $s \in \left( -\infty, \frac12\right)$. 
Then the solutions to the equation \eqref{DirichletProblem} satisfy the following bound:
\begin{equation}
\label{est extension}
\Vert \beta ^{-s} v(\alpha, \beta )\Vert_{L_\beta^2 L^\infty_\alpha} \lesssim \Vert g\Vert _{H^s_h}.
\end{equation}
$ii)$ The same result also holds for 
the equation \eqref{MixedBCProblem}.
\end{proposition}

This will transfer easily later on to a similar bound for the Laplace equation in the fluid domain. 
\begin{proof}
As already mentioned, the solution to \eqref{DirichletProblem} is of the form
\[
v(\alpha, \beta) = \frac{1}{{2\pi}}\int p_D(\xi,\beta)\hat
g (\xi)e^{i \alpha \xi}\,d \xi,
\]
where
\[
p_D(\xi,\beta) = \frac{\sinh((\beta+h)\xi)}{\sinh(h\xi)}.
\]
Notice that $\la p_D(\xi,\beta)\ra 
\leq e^{c\beta\la \xi\ra}$ for some positive constant $c$. 

We now consider a Littlewood-Paley decomposition of $g$,
\[
g = g_{\leq 1/h} + \sum_{1/h<\lambda\in 2^{\xZ}} g_\lambda.
\]
By the triangle inequality and Bernstein's inequality applied 
to each corresponding dyadic piece of $v$ we obtain
\[
\| v(\beta)\|_{L^\infty_\alpha} \lesssim 
h^{-\frac12} \| g_{\leq 1/h} \|_{L^2}+
\sum_{\lambda > 1/h}
\lambda^\frac12 e^{c \beta \lambda} \| g_\lambda\|_{L^2}.
\]
For $ s < \dfrac12$ the functions $\beta^{-s} e^{c \beta \lambda}$
are easily seen to be almost orthogonal in $L^2(-h,0)$. Then it follows that
\[
\| \beta^{-s} v\|_{L^2_\beta L^\infty_\alpha}^2 \lesssim 
h^{-2s} \| g_{\leq 1/h} \|_{L^2}^2 +
\sum_{\lambda > 1/h}
\lambda^{2s} \| g_\lambda\|_{L^2}^2,
\]
which completes the proof of $i)$. 

To prove $ii)$, we remark that we above we have only used the fact that 
$\la p_D(\xi,\beta)\ra$ is bounded from above by $e^{c\beta\la \xi\ra}$. 
Since the symbol $p_N$ satisfies the same bound, the same conclusion holds for the solution to \eqref{MixedBCProblem}.
\end{proof}

\subsection{Holomorphic functions in the canonical domain}

Here we consider holomorphic functions $w$ in the canonical domain $S\defn\{\alpha+i\beta\ :\ \alpha\in\xR,~-h\leq \beta\leq 0\}$, which are real on the bottom $\{\R-ih\}$. These functions form a real algebra. Such functions are uniquely determined
by their values on the real line $\{\beta=0\}$,
and can be expressed 
as 
\[
w = u+ iv,
\]
where $u$ and $v$ are harmonic conjugate functions which solve 
the equations \eqref{MixedBCProblem}, respectively \eqref{DirichletProblem}.

By extension we will call functions on the real line holomorphic if they
are the restriction on the real line of holomorphic functions in the
strip and satisfy the above boundary condition on the bottom. This consists of functions $w\colon \xR\rightarrow \xC$ so that there is an holomorphic function, still denoted by $w\colon S\rightarrow \xC$, which satisfies
\[
\Im w = - \Tilh \Re w
\]
on the top.

% as can be seen from a simple
%application of the product formula
%%
%\begin{equation}\label{SummationFormula}
%  u\Tilh[v] + \Tilh[u]v = \Tilh\big[uv - \Tilh[u]\Tilh[v]\big],
%\end{equation}
%
%which follows from the corresponding identity for \(\tanh\xi\). 
The complex conjugates of holomorphic functions are called antiholomorphic. 

\subsection{Holomorphic coordinates and water waves} Given the fluid
domain $\Omega$ at some time $t$ we introduce holomorphic coordinates
$z=\alpha + i\beta$, via conformal maps
\[
Z\colon S \rightarrow \Omega(t),
\]
which associate the top to the top, and the bottom to the bottom.

These maps are uniquely defined up to horizontal translations in $S$. 
Restricted to the real axis this provides a
parametrization for the water surface $\Gamma$. Because of the
boundary condition on the bottom of the fluid domain the function $W$
is holomorphic when $\alpha \in \mathbb{R}$. 

Such a conformal transformation exists by the Riemann mapping theorem, and can be constructed as follows:
\begin{itemize}
\item construct the harmonic function $\beta$ in the fluid domain,
which takes values $0$ on the top, and $-h$ on the bottom.
\item construct the function $\alpha$ in the fluid domain as a harmonic 
conjugate of $\beta$. This is uniquely determined modulo real constants.
\item invert the holomorphic map $x+iy \to \alpha+i\beta$ to obtain
the desired conformal map $Z$.
\end{itemize}

Given such a map $Z$, we denote by
$$
W:=Z-\alpha,
$$
where $W = 0$ if the fluid surface is flat i.e., $\eta = 0$.

Turning our attention to the velocity potential $\phi$, we consider its
harmonic conjugate~$q$ and then the function $Q:=\phi +i q$  taken in 
conformal coordinates is the
holomorphic counterpart of $\phi$. Here $q$ is exactly the stream
function also used in the previous section.

One can model the water wave equations in holomorphic coordinates as
an evolution for $(W, Q)$ within the space of holomorphic
functions defined on the surface. This is described in detail in the papers \cite{HIT} for the infinite depth case, respectively \cite{H-GIT} for the finite depth
case (see also \cite{dy-zak}). We recall the equations:
 
\begin{equation}\label{FullSystem-re}
  \begin{cases}
    W_t + F(1+W_\alpha) = 0\vspace{0.1cm}\\
    Q_t + FQ_\alpha - g\Tilh[W] + \P_h
    \left[\dfrac{|Q_\alpha|^2}{J}\right] = 0,
  \end{cases}
\end{equation}
where
\[
J = |1+W_\alpha|^2, \qquad F = \P_h\left[\frac{Q_\alpha-\bar
    Q_\alpha}{J}\right].
\]
Here $\P_h$ represents the orthogonal projection on the space of
holomorphic functions with respect with the inner product
in the Hilbert space $\mfH_h$ introduced in
\cite{H-GIT}. This has the form
\[
\langle u, v\rangle_{\mfH_h} := \int \left( \Til_h \Re u \cdot \Til_h \Re v  + \Im u \cdot \Im v \right) \, d\alpha, 
\]
and coincides with the $L^2$ inner product in the
infinite depth case.
Written in terms of the real and imaginary parts of $u$, the
projection $\P_h$ takes the form
\begin{equation}\label{def:P}
\P_h u = \frac12 \left[(1 - i \Tilh) \Re u + i (1+ i\Tilh^{-1}) \Im u           \right].
\end{equation}

Since all the functions in the system \eqref{FullSystem-re} are 
holomorphic, it follows that these relations also hold in the 
full strip $S$ for the holomorphic extensions of each term. 

We also remark that in the finite depth case there is an additional gauge freedom in the above form of the equations, in that $\Re F$ is a-priori only uniquely determined up to constants. This corresponds to the similar degree of freedom in the choice of the conformal coordinates,
% * <tataru@math.berkeley.edu> 2018-06-19T20:34:24.407Z:
%
% ^.
and will be discussed in the last subsection.

A very useful function in the holomorphic setting is 
\[
R = \frac{Q_\alpha}{1+W_\alpha},
\]
which represents the ``good variable" in this setting, and  
corresponds to the Eulerian function
\[
R = \phi_x + i \phi_y.
\]
We also remark that the function $\theta$ introduced in the previous section is described in holomorphic coordinates by
\[
\theta =\Im W.
\]
Also related to $W$, we will use the auxiliary holomorphic function
\[
Y = \frac{W_\alpha}{1+W_\alpha}.
\]

Another important auxiliary function here is the advection velocity
 \[
 b = \Re F,
 \]
which represents the velocity of the particles on the fluid surface
in the holomorphic setting. 

It is also interesting to provide the form of the conservation 
laws in holomorphic coordinates.
We begin with the energy (Hamiltonian), which has the form
\[
\mathcal H =  \frac g2 \int |\Im W|^2 (1+\Re W_\alpha ) \, d\alpha
- \frac14\langle Q,\Tilh^{-1}[Q_\alpha]\rangle_{\mfH_h} .
\]
The momentum on the other has the form
\[
\mathcal M = \frac12 \langle W, \Til_h^{-1} Q_\alpha \rangle_{\mfH_h} = \int _{\mathbb{R}}\Til_h \Re W \cdot \Re Q_\alpha \, d\alpha = \int _{\mathbb{R}} \Im W \cdot \Re Q_\alpha \,  d\alpha .
\]

\subsection{Uniform bounds for the conformal map}
In order to freely switch computations between the Eulerian and holomorphic setting it is very useful to verify that our Eulerian 
uniform smallness assumption also has an identical interpretation in the holomorphic setting.

To account for the uniformity in time in the $X$ norm 
it is very convenient to use the language of frequency envelopes. We define a \emph{frequency envelope} for 
$(\eta,\nabla \phi_{|y = \eta})$ in $X$ to be any positive sequence 
\[
\left\{ c_{\lambda}; \quad h^{-1} < \lambda \in 2^\xZ \right\}
\]
with the following two properties:
\begin{enumerate}
\item Dyadic bound from above,
\[
\| P_\lambda (\eta,\nabla \phi_{|y = \eta})\|_{X_0} \leq c_\lambda .
\]
\item Slowly varying,
\[
 \frac{c_\lambda}{c_\mu} \leq \max \left\{ \left(\frac{\lambda}{\mu}\right)^\delta, \left(\frac{\mu}{\lambda}\right)^\delta \right\} .
\]
\end{enumerate}
Here $\delta \ll 1$ is a small universal constant. Among all such frequency envelopes there exists a \emph{minimal frequency envelope}. 
In particular, this envelope has the property that
\[
\| (\eta,\nabla \phi_{|y = \eta})\|_{X} \approx \| c\|_{\ell^1}.
\]
This will play an important role in our analysis:

\begin{definition} \label{def-fe}
By $\{c_\lambda\}_{\lambda \geq 1/h}$ we denote the minimal frequency envelope for $(\eta,\nabla \phi\arrowvert_{y=\eta})$ in $X_0$. We call $\{c_\lambda\}$ the \emph{control frequency envelope}.
\end{definition}
Since in solving the Laplace equation on the strip, solutions at depth $\beta$ are localized at frequencies $\leq \lambda$ where  $\lambda \approx  |\beta|^{-1}$, we will also use the notation 
\[
c_\beta = c_\lambda, \qquad \lambda \approx 
|\beta|^{-1}.
\]
This uniquely determines $c_\beta$ up to a small multiplicative constant, which suffices for our purposes.

We now use the control envelope to transfer the control norm bound 
for $(\eta,\nabla \phi_{|y=\eta})$ to their counterpart $(\Im W,R)$
in the holomorphic coordinates.

\begin{proposition}\label{p:control-equiv}
Assume the smallness condition \eqref{uniform}, and let $\{c_\lambda\}$
be the control envelope as above. Then we have 
\begin{equation}\label{X-fe-hol}
\| P_{\lambda} (\Im W,R)\|_{X_0} \lesssim c_\lambda .
\end{equation}
\end{proposition}
\smallskip
\begin{remark}We remark that this in particular implies the $X$ bound 
\begin{equation}\label{uniform-hol}
\|(\Im W, R)\|_{X} \lesssim \epsilon_0,
\end{equation}
and also, by Bernstein's inequality, the pointwise bound
\begin{equation}\label{W-ppoint}
\| W_\alpha\|_{L^\infty} \lesssim \epsilon_0.
\end{equation}
\end{remark}
This in turn implies that the Jacobian matrix for the change of coordinates stays close to the identity.

\begin{proof}
By a continuity argument, it suffices to prove the desired bounds 
under the additional bootstrap assumption 
\begin{equation}\label{boot}
\|(\Im W, R)\|_{X} \leq \epsilon_1, \qquad \epsilon_0 \ll 
\epsilon_1 \ll 1 .
\end{equation}
 We caution the reader that the two $X$ norms
and their associated frequency envelopes for  $(\eta,\nabla \phi_{|y = \eta})$, respectively  $(\Im W,R)$ are relative to different coordinate systems, Eulerian vs. holomorphic.

To prove the proposition we first compare the regularity  of $\Im W$ with the regularity  of $\eta$, since (either of) these functions determine the conformal map. Let $\{c_\lambda\}$, $\{d_\lambda\}$ be minimal frequency envelopes for  $(\eta,\nabla \phi_{|y = \eta})$, respectively $(\Im W,R)$ in $X$, so that 
we have 
\[
\| d\|_{\ell^1} \leq \epsilon_1 .
\]
Then we will show that for each $\lambda \geq 1/h$ we have the equivalence 
\begin{equation}\label{equiv-fe}
c_\lambda \approx d_\lambda.
\end{equation}

Our bootstrap assumption insures that $\Re W_\alpha$ is pointwise small,  which implies that the change of coordinates $x = \alpha+ \Re W(\alpha)$
is biLipschitz, so we easily have the norm equivalence
\begin{equation}\label{H01-equiv}
\| f\|_{L^2_\alpha} \approx \|f\|_{L^2_x}, \qquad \| f\|_{\dot H^1_\alpha}
\approx \|f\|_{\dot H^1_x}.
\end{equation}
The $L^2$ bound allows us to easily compare the $L^2$ norms of $\eta$
and $\Im W$, which accounts for the case $\lambda = 1/h$, namely
\[
\| \Im W_{< 1/h} \|_{L^2} \lesssim h^\frac32 c_{1/h},
\qquad \| \eta \|_{L^2} \lesssim h^\frac32 d_{1/h}.
\]
For higher frequencies, it remains to compare minimal frequency envelopes for their derivatives $\eta_x$ and $\Im W_\alpha$ in $ H^{\frac12}_h$, which are also comparable to $ c_\lambda$, respectively $d_\lambda$. Here we also need bounds for 
\[
\Re W_\alpha = - \Tilh^{-1} \partial_\alpha \Im W.
\]
But it is easily seen that $d_{\lambda}$ is also an envelope for $\Re W_\alpha$ in $H^{\frac12}_h$.

To begin with, we note that by interpolation, the bound \eqref{H01-equiv} insures the equivalence of all intermediate 
$l^p(H^s_h)$ norms and envelopes for all 
$1 \leq p \leq \infty$ and $0 < s < 1$,
with uniform frequency envelope bounds. We will use this property
for the norm $\ell^1 H^\frac12_h$, in order to harmlessly switch the 
function $\eta_x$ to holomorphic coordinates. Hence it remains 
to compare the $H^{\frac12}_h$ frequency 
envelopes for the functions $\eta_x$ and $W_\alpha$ both measured in the 
holomorphic coordinates. This is convenient since by chain rule we have the relation
\[
\eta_x = \frac{ \Im W_\alpha}{ 1+ \Re W_\alpha}.
\]

To deal with the nonlinear expression we use the algebra property 
of $\ell^1 H^\frac12_h$, expressed in a frequency envelope fashion.
For convenience, we state this as

\begin{lemma}\label{l:algebra}
a) The space $\ell^1 H^\frac12_h$ is an algebra\footnote{This property suffices in the present paper since $W_\alpha$ is small in $L^\infty$.
However, even if $W_\alpha$ were large, then  
bounds as in the lemma would still be valid. However, proving that would require corresponding Moser estimates in $\ell^1 H^\frac12_h$. For that we refer the reader to the similar analysis in \cite{HIT}.}. Furthermore, if 
$u,v \in \ell^1 H^\frac12_h$ have frequency envelopes $c^u_\lambda$, $c^v_\lambda$ then an envelope for $uv$ is given by 
\[
c^{uv}_{\lambda} = c^u_\lambda \|c^v\|_{\ell^1} + c^v_\lambda \|c^u\|_{\ell^1} .
\]

b) Let $u \in L^2$ and $v \in \ell^1 H^\frac12_h$ have frequency envelopes $c^u_\lambda$, $c^v_\lambda$ then an envelope for $uv$ 
in $L^2$ is given by 
\[
c^{uv}(\lambda) = c^u_\lambda \|c^v\|_{\ell^1} .
\]
\end{lemma}

The proof of the lemma is relatively simple and is omitted.  

The smallness of $\epsilon_1$ in our bootstrap assumption allows
us to use the lemma in order to estimate the difference
\[
\eta_x - \Im W_\alpha = - \frac{ \Im W_\alpha \cdot \Re W_\alpha}{ 1+ \Re W_\alpha}.
\]
Precisely, a frequency envelope for $\eta_x - \Im W_\alpha$
will be given by $\epsilon_1 d_\lambda$. Then, by the triangle inequality 
for minimal frequency envelopes, we must have
\[
| c_\lambda - d_\lambda | \lesssim \epsilon_1 d_\lambda.
\]
Since $\epsilon_1 \ll 1$, this implies that $c_\lambda \approx d_\lambda$. This concludes the proof of \eqref{equiv-fe} restricted 
to the $\eta$ and $\Im W$ components. 

Next we consider the equivalence of the frequency envelopes
for $\nabla \phi_{|y= \eta}$ respectively $R = (\phi_x + i \phi_y)_{|y= \eta}$ in $H^1_h$. These are one and the same function, 
and the only  difficulty is that the $H^1_h$ norms and frequency 
envelopes are measured in different frames, Eulerian vs. holomorphic.
The $L^2$ part of the $H^1_h$ norm is easily dealt with using 
\eqref{H01-equiv}, so it remains to compare the frequency envelopes
for their derivatives in $L^2$.

As before, we compute using the chain rule
\[
g:= \partial_x( \phi_x + i \phi_y)_{|y= \eta} = \frac{R_\alpha}{1+\Re W_\alpha}.
\]
Using part (b) of the last lemma, it is easily 
seen that in holomorphic coordinates the function $g$ has a 
minimal frequency  envelope comparable to that of $R_\alpha$. Thus 
it only remains to see that the function $g$ has equivalent 
$L^2$ minimal frequency envelopes in Eulerian and holomorphic coordinates.

This follows if we show the following off-diagonal decay:
\begin{equation}\label{L2-frame-decay}
\| P^E_\lambda P_\mu g\|_{L^2} \lesssim \left\{ \frac{\lambda}{\mu}, \frac{\mu}{\lambda} \right\},
\end{equation}
where $P^E_\lambda $ and $P_\mu$ are Littlewood-Paley projectors in the 
Eulerian, respectively holomorphic frame.

To prove \eqref{L2-frame-decay} we consider two cases:
\medskip

a) $\lambda \geq \mu$. Then we write
\[
\| P^E_\lambda P_\mu g\|_{L^2} \lesssim 
\lambda^{-1} \| \partial_x P_\mu g\|_{L^2} \lesssim 
\lambda^{-1} \| \partial_\alpha P_\mu g\|_{L^2} \lesssim \mu /\lambda.
\]

\medskip

b) $\lambda \leq \mu$. Then we use duality to interchange the two projections, and then argue exactly in the same way.

The proof of the Proposition~\ref{p:control-equiv} is complete.
\end{proof}

As a consequence of the last proposition we can further 
extend the range of our frequency envelope estimates:

\begin{remark}
The previous proposition and its proof show that $\left\lbrace c_\lambda \right\rbrace$ is also a frequency envelope for 

\begin{itemize}
\item $(\Im W, R)$ in $X_0$.

\item $W_\alpha$ in $H^{\frac12}_h$ and $L^\infty$.

\item $Y$ in $H^{\frac12}_h$.
\end{itemize}
Here the last property is a direct consequence of Lemma~\ref{l:algebra}.
\end{remark}

\subsection{ Vertical strips in Eulerian vs holomorphic coordinates.} 

In our main result, we define local energy functionals using vertical strips
in Eulerian coordinates. On the other hand, for the multilinear
analysis in our error estimates in the last two sections, we would
like to use vertical strips in holomorphic coordinates. Of course
these two types of vertical strips do not perfectly match. To switch from one to the other we need to estimate the horizontal drift 
between the two strips in depth. 

As the conformal map is biLipschitz, it suffices to compare the centers of the two strips. It is more convenient to do this in the reverse order,
and compare the Eulerian image of the holomorphic vertical section with 
the Eulerian vertical section:

\begin{proposition} \label{p:switch-strips}
Let $(x_0, \eta(x_0))  = Z(\alpha_0,0)$, respectively $(\alpha_0,0)$ be the coordinates of a point on the free surface in Eulerian, respectively  holomorphic coordinates. 
Assume that \eqref{uniform} holds, and let $\{c_\lambda\}$ be the control frequency envelope in  Definition~\ref{def-fe}. Then we have the uniform bounds:
\begin{equation}
|\Re Z( \alpha_0, \beta) - x_0 + \beta \Im W_{\alpha}(\alpha_0,\beta)| \lesssim  c_\lambda, \qquad |\beta| \approx \lambda^{-1}.
\end{equation}
\end{proposition}

As a corollary, we see that the distance between the two 
strip centers grows at most linearly:

\begin{corollary}
Under the same assumptions as in the above proposition we have
\begin{equation}
|\Re Z( \alpha_0, \beta) - x_0| \lesssim \epsilon_0 |\beta|.
\end{equation}
\end{corollary}

\begin{proof}
We consider the expression
\[
D = \Re Z( \alpha, \beta) - x_0 + \beta \Im W_{\alpha}(\alpha_0,\beta) = \Re W(\alpha,\beta) - \Re W(\alpha,0) + \beta \Im W_{\alpha}(\alpha_0,\beta).
\]
We can express this in terms of $\Im W$ on the top as  follows:
\[
\begin{split}
D = & \ (P_N(D,\beta) -1) \Re W(\alpha,0) + \beta \partial_\alpha P_D (D,\beta) \Im W(\alpha,0)
\\
=  & \ \left(\Tilh^{-1}(P_N(D,\beta) -1) - i \beta D P_D(D,\beta)\right)\Im W(\alpha,0).
\end{split}
\]
The symbol for the multiplier
\[
M(D,\beta) = \Tilh^{-1}(P_N(D,\beta) -1)- i \beta D P_D(D,\beta)
\]
is
\[
\begin{split}
m(\xi,\beta) = & \ \frac{i \cosh \left((\beta+h)\xi\right) -\cosh (h\xi)}{\sinh (h\xi)} - \frac{ i \beta \xi \sinh \left((\beta+h)\xi\right)}{\sinh (h\xi)} \\
= & \ \frac{2\sinh \left( \beta \xi/2\right) \sinh \left((h+\beta/2)\xi\right) - \beta \xi \sinh \left((\beta+h)\xi\right)
}{\sinh(h\xi)}.
\end{split}
\]
This is easily seen to be smooth and satisfy the bound
\[
|m(\xi,\beta)| \lesssim \min\{ 1,|\beta \xi|^2\}.
\]
Given this symbol bound, the conclusion of the proposition follows
by applying Bernstein's inequality for each dyadic frequency, and then
summing up.
\end{proof}

\subsection{ The horizontal gauge invariance}

Here we briefly discuss the gauge freedom due to the 
fact that $\Re F$ is a-priori only uniquely determined up to constants. 
In the infinite depth case this gauge freedom is removed by making the assumption $F \in L^2$. In the finite depth case (see \cite{H-GIT}) instead this is more arbitrarily removed by setting $F(\alpha = -\infty) = 0$. 

In the present paper no choice is necessary for our main result,
as well as for most of the proof. However, in the choice of the 
normal form momentum density correction in Section ~\ref{s:nonlinear}
it is convenient to make such a choice, which is discussed next.
This choice is used in the very last step in Section~\ref{s:F}.

Assume first that we have a finite depth.
We start with a point $x_0 \in \R$ where our local energy estimate 
is centered.   Then we resolve the gauge invariance  with respect to horizontal translations by setting $\alpha(x_0) = x_0$,
which corresponds to setting $ \Re W(x_0)= 0$. In dynamical terms, this implies that  the real part of $F$ is uniquely determined by 
\[
0 = \Re W_t(x_0) = \Re ( F(1+W_\alpha))(x_0),
\]
which yields
\[
\Re F(x_0) = \Im F(x,0) \frac{\Im W_\alpha(x_0)}{1+ \Re W_\alpha(x_0)}.
\]

 In the infinite depth case, the canonical choice for $F$ is the one vanishing at infinity. This corresponds to a moving location in the $\alpha$ variable. We can still rectify this following the finite depth model, at the expense of introducing a constant component in both $\Re W$ and in $F$. We will follow this convention in the paper, in order to insure that our infinite depth computation is an exact limit of the finite depth case.

\section{Local energy decay for linear gravity waves}\label{s:linear}

\subsection{Linearized equations in Eulerian coordinates}

In Eulerian coordinates the linearized equations around the zero solution are 
\begin{equation}\label{ww-lin0}
\left\{ 
\begin{array}{l}
\partial_t \eta = \mathcal{D}_N \psi
\cr
\partial_t \psi = -g\eta ,
\end{array}
\right.
\end{equation}
where $\mathcal{D}_N$ is the Dirichlet to Neumann map associated to depth $h>0$, given by 
\[
D_N \psi = \partial_y \phi_{|top} 
%=  \partial_x q_{|top} 
= \Tilh \partial_x \psi,
\]
where recall that $\Tilh$ is the Tilbert operator given by \eqref{def:Tilbert} and 
$\phi$ is the harmonic extension of $\psi$ in the flat strip $S=\{(x,y)\in \xR^2\ :\ -h<y<0\}$, so that
\begin{equation*}
\left\{
\begin{aligned}
&\Delta \phi=0 \quad \text{in }S,\\
&\phi\arrowvert_{y=0}=\psi,\\
&\partial_y\phi\arrowvert_{y=-h}=0.
\end{aligned}
\right.
\end{equation*}

For such $\eta$ and $\psi$ we define the (conserved) energy as 
\[
E_{lin}(\eta,\psi) := \frac g2 \| \eta \|_{L^2}^2 + \frac 1 2\langle \Tilh\partial_x \psi,\psi \rangle.
\]
We can express the energy in a more symmetric fashion by using the 
harmonic extension $\phi$ of $\psi$ in the strip $S$ with Neumann boundary condition 
on the bottom. Then
\[
E_{lin} (\eta,\psi) = \frac g2\| \eta \|_{L^2}^2 + \frac 1 2\| \nabla \phi\|_{L^2(S)}^2.
\]
We also introduce higher energies
\[
E_{lin}^s (\eta,\psi) :=  \frac{g^{1-2s}}{2}\| (\Tilh^{-1} \partial_x)^s \eta \|_{L^2}^2 + \frac{g^{-2s}}{2}\|(\Tilh^{-1} \partial_x)^s  \nabla \phi\|_{L^2(S)}^2.
\]
These are homogeneous norms in the infinite depth case, but the
homogeneity is broken in the finite depth case.

The local energy for the linearized equation is given by 
\[
\|(\eta,\psi)\|_{LE}^2 = \|\eta\|_{LE^0}^2 + \| \nabla \phi\|_{LE^{-\frac12}}^2,
\]
where 
\[
\|\eta\|_{LE^0} := \sup_{x_0 \in \R} \|\eta\|_{LE^0_{x_0}}, \qquad
\|\eta\|_{LE^0_{x_0}}^2 = \int_0^T\int \chi(x-x_0) \eta^2 \, dx \, dt
\]
while
\[
\| \nabla \phi\|_{LE^{-\frac12}} := \sup_{x_0 \in \R}\| \nabla \phi\|_{LE^{-\frac12}_{x_0}}, \qquad \| \nabla \phi\|_{LE^{-\frac12}_{x_0}}^2 = \int_0^T\iint_S \chi(x-x_0) |\nabla \phi|^2 \, dx dy dt.
\] 
With these notations, the local energy decay estimate for the 
linearized equation is as follows:

\begin{theorem}
There exists a constant $C$ such that, for all $h\in [1,+\infty)$ and all $T\in (0,+\infty)$, solutions $(\eta,\psi)$ to the above system \eqref{ww-lin0} satisfy the local energy bound
\begin{equation}
\Vert (\eta,\psi)\Vert_{LE} \leq C \left(\Vert (\eta,\psi) (0)\Vert_{ E_{lin}^\frac14 }+ \Vert (\eta,\psi) (T)\Vert_{ E_{lin}^\frac14 }\right).
\end{equation}
\end{theorem}

The rest of the section is devoted to the proof of the theorem. By scaling we can and will assume without any loss of generality that $h \gg 1$. Precisely, in the following proof $h$ will play the role of an 
(inverse) semiclassical parameter.

The proof is based on  Morawetz' identities starting from the momentum conservation, and more precisely from the linear counterparts of the momentum densities $I_2$ and $I_3$ in section \ref{s:Cl}.
We define the momentum as 
\[
\mathcal{M} = \int \eta \psi_x \, dx
\]
with $I_2(x) = \eta \psi_x$ as the first momentum density.

For this proof, given a function $f=f(t,x,y)$ with $(x,y)\in \overline{S}$, we set 
$$
\evaluate{f}(t,x):=f(t,x,0).
$$
Now, using the equations for $\eta,\psi$, given a bounded  increasing function $m$, we compute
\[
\partial_t \int m(x) I_2 (t,x) \, dx 
= \int m \evaluate{\phi_y}  \psi_x \, dx - \int gm \eta \eta_x \, dx. 
\]
The second term in the right-hand side gives
\[
\frac g2 \int m_x \eta^2\, dx.
\]
The first term can be written as
\begin{align*}
\int m \evaluate{\phi_y}  \evaluate{\phi_x}
\, dx &= \iint m \partial_y(\phi_{y} \phi_{x})\, dy dx\\
&=\iint m(\phi_{yy}\phi_x+\phi_{y} \phi_{xy}) \, dy dx\\
&=\iint m(-\phi_{xx}\phi_x+\phi_{y} \phi_{xy}) \, dy dx\\
&= \frac12 \iint m_x (\phi_x^2 - \phi_y^2) \, dy dx.
\end{align*}
Thus we conclude that 
\begin{equation}\label{est1}
\partial_t \int m I_2\, dx = \frac g2 \int m_x \eta^2 \, dx + \frac12 \iint m_x (\phi_x^2 - \phi_y^2) \, dy dx.
\end{equation}
The first term on the right is a component of the local energy, whereas the second is nonnegative when $m_x$ is replaced by $1$ 
(see Lemma~\ref{L3.3} in the appendix, applied with $w=1$ and $\eta=0$).

We now continue by using a second momentum density $I_3$, which in addition to the functions $\eta$, $\psi$ and $\phi$, 
depends on the functions $\theta$ and $q$ introduced in the previous sections (see \eqref{defi:qconj} and \eqref{defi:thetainitial}):
\begin{itemize}
\item $\theta$ is the harmonic extension of $\eta$ with Dirichlet boundary condition on the bottom;
\item $q$ is the harmonic conjugate of $\phi$ with Dirichlet boundary condition on the bottom.
\end{itemize}
With these notations, one has
\[
\mathcal{M} = \int I_3\, dx\quad \text{with} \quad I_3=\int_{-h}^0 \nabla \theta \cdot\nabla q \, dy.
\]
Although it is natural to define $I_3$ in terms of $(\theta,q)$, for the computations it is convenient to express $I_3$ in terms of $(\theta,\phi)$. 
It follows immediately from the equations $q_x=-\phi_y$ and $q_y=\phi_x$ that
\[
I_3(t,x)=\int_{-h}^0 \big(\theta_y \phi_x-\theta_x\phi_y)\, dy.
\]
Notice that
$$
\partial_t \theta=\phi_y
$$ 
(this is the simplified version of \eqref{thetat} for the linearized equation). 
%On the other hand, using the operator %$\Tilstrip$ (see %Notations~\ref{nota:harmonic}) one has
%%$$
%\partial_t q=g \Tilstrip \Theta
%$$
%(as can be check by considering the %boundary conditions of these two harmonic functions). Consequently, 
%On the top $\phi$ and $q$ are related by 
%\[
%q = - \Til \phi .
%\]
%The evolution equations for $\theta$ and $q$ are given by
%\[
%\left\{ 
%\begin{array}{l}
%\partial_t \theta = - \partial_x q
%\cr
%\partial_t q = g \Til \theta .
%\end{array}
%\right.
%\]
As a result, we get for any weight $m$,
\begin{equation}\label{est2}
\frac{d}{dt}  \int m I_3\, dx  =  - \iint m \big(\phi_{yy}\phi_x-\phi_{yx}\phi_y\big)\, dydx 
 + \iint m \big(\theta_y\partial_t \phi_x-\theta_x \partial_t \phi_y \big)\, dydx.
 % \cdot \nabla \Tilstrip \theta \, dydx.
\end{equation}
Since $\phi_{yy}=-\phi_{xx}$, integrating by parts, 
the first term gives the expression 
\[
- \iint m \big(\phi_{yy}\phi_x-\phi_{yx}\phi_y\big)\, dydx = \frac12 \iint m_x  |\nabla \phi|^2 \, dydx,
\]
which is the second part of the local energy. Our second observation is that the second term depends only on $m$ and $\eta$. To see this, we use the operator $H_D$ (respectively $H_N$) introduced in the previous section, which maps a function $f=f(x)$ to its harmonic extension in the strip $S$ with Dirchlet (respectively Neumann) boundary condition on the bottom. 
Then, by definition, one has 
$\theta=H_D(\eta)$. 
On the other hand, 
since $\partial_t\phi\arrowvert_{y=0}=-g\eta$, it follows that $\partial_t\phi=-gH_N(\eta)$. 
Consequently, one has
$$
\iint m \big(\theta_y\partial_t \phi_x-\theta_x \partial_t \phi_y \big)\, dydx
=gQ_m(\eta),
$$
where 
\begin{equation}\label{defi:Qmeta}
Q_m(\eta):=\iint m \big(H_N(\eta)_yH_D(\eta)_x-H_N(\eta)_xH_D(\eta)_y\big)\, dydx.
\end{equation}
Thus, we conclude that
\begin{equation}\label{est3}
\frac{d}{dt}  \int m I_3\, dx  =  \frac12 \iint m_x  |\nabla \phi|^2 \, dy dx  + g \iint m Q_m(\eta) \, dy dx.
\end{equation}

Notice that in the infinite depth case, one has $H_N(\eta)=H_D(\eta)$ so 
$Q_m(\eta)=0$, which greatly simplifies the proof of the theorem.
To prove a result that holds uniformly in the finite depth case, the idea here is now to try to combine the two local energies in a more balanced way. Given a parameter $\sigma\in [0,1]$, we define
\[
\mI_m^\sigma(t) =  \sigma \int m I_2
\, dx + (1-\sigma) \int m I_3\, dx .
\]
Then we have the following:

\begin{proposition}\label{P:52}
Let $h \gg 1$. Then

a) For each $\sigma \in [0,1]$ we have
\begin{equation}\label{mI(0)}
|\mI_m^\sigma(t) | \lesssim \Vert (\eta,\psi) \Vert_{E_{lin}^\frac14}^2.
\end{equation}

b) There exist $\sigma < \frac12 $ close to $\frac12$ and $c < 1$ independent of $h$ so that  \begin{equation}\label{dmI}
\mI_m^\sigma(T) - \mI_m^\sigma(0) \geq  \Vert(\eta,\psi)\Vert_{LE_{x_0=0}}^2 -  c\Vert(\eta,\psi)\Vert_{LE}^2
\end{equation}
holds for all solutions $(\eta,\psi)$ of the equation  \eqref{ww-lin0}.

\end{proposition}

The conclusion of the theorem follows  by  taking supremum over all translates  of \eqref{dmI}.  The remainder of the section is devoted to the proof of the proposition.

We begin with part (a). We need to consider the the two momentum densities $I_2$ and $I_3$. The contribution of $I_2$ has the form
\[
\int m\eta\psi_x\, dx.
\]
We estimate this as follows
\[
\left| \int m\eta\psi_x\, dx \right| \lesssim \Vert m\eta \Vert_{H_h^{\frac{1}{4}}} \Vert \psi_x\Vert_{H_h^{-\frac{1}{4}}},
\]
and conclude using the fact that $m$ is a bounded multiplication operator in $H_h^{\frac14}$,
\begin{equation}
\label{w11}
\Vert m\eta \Vert_{H_h^{\frac{1}{4}}}\lesssim \Vert m_x\Vert_{L^{1}} \Vert \eta \Vert_{H_h^{\frac{1}{4}}}.
\end{equation}

Now we consider the contribution of $I_3$. To do so, we integrate by parts to arrive at
\begin{align*}
I_3&=\iint_{S} m(\theta_y\phi_x-\theta_x\phi_y)\, dydx=
\iint_S m \big(\partial_y(\theta \phi_x)-\partial_x(\theta \phi_y)\big)\, dydx\\
&=\int m I_2\, dx-\iint_S m_x\theta q_x\, dydx.
\end{align*}
It remains to estimate the second part for which we will use the $x$-localized $L^2$ bounds for harmonic extensions in Proposition \ref{p:extension}. 
This yields 
$$
\blA  y^{-\frac{1}{4}}\theta\brA_{L^\infty_x(L^2_y)}\leq \blA  y^{-\frac{1}{4}}\theta\brA_{L^2_y(L^\infty_x)}
\lesssim \blA \eta\brA_{H^\frac 14 _h},
$$
and similarly
$$
\blA y^{+\frac{1}{4}}\phi_y\brA_{L^\infty_x(L^2_y)}
\lesssim \blA \phi_y\arrowvert_{y=0}\brA_{H^{-\frac 14} _h}.
$$
Since $m_x$ is a positive function with integral $1$, we 
conclude that
\begin{align*}
\left| \iint_{S} m_x\theta \phi_y\, dydx \right| 
&\leq \Big(\int m_x\, dx\Big) \blA  y^{-\frac{1}{4}}\theta\brA_{L^\infty_x(L^2_y)}
\blA y^{+\frac{1}{4}}\phi_y\brA_{L^\infty_x(L^2_y)}\\
&\lesssim \Vert \eta \Vert_{H_h^{\frac{1}{4}}}
\Vert \phi_y\Vert_{H_h^{-\frac{1}{4}}}.
\end{align*}
Since $\phi_y\arrowvert_{y=0}=\mathcal{D}\psi$, this gives the wanted estimate \eqref{mI(0)}.

We now prove part (b). We have 
\begin{equation}\label{splitting}
\mI^{\sigma}_m(T) - \mI^{\sigma}_m(0) = LE_\phi+LE_\eta,
\end{equation}
where
\[
LE_\phi =  \int _{0}^T  \iint \Big(\frac{1-\sigma}2 m_x  |\nabla \phi|^2  + \frac{\sigma}2 
m_x (\phi_x^2 - \phi_y^2)\Big)\, dydxdt,
\]
and 
\[
LE_\eta = \frac{\sigma g}2 \int_0^T \int m_x \eta^2\, dxdt + (1-\sigma)g\int_0^T  Q_m(\eta) \dt,
\]
where $Q_m(\eta)$ is defined by \eqref{defi:Qmeta}. 
We first observe that the second term in $LE_\phi$ is clearly positive if $\sigma < \frac12$. 
So, to conclude the proof, 
it is sufficient to prove that the 
$LE_\eta$ component controls 
the potential energy. This in turn is 
straightforward in the infinite depth case, since then, $Q_m(\eta)=0$. 
Hence from here on we focus on the finite depth case where the
challenge is in part to gain the uniformity as $h \to \infty$. 

So the goal is to prove that for some $\sigma\in (0,1/2)$, the expression $LE_\eta$  is positive definite, either directly or after taking a supremum over all translations of $m$. For that we need to write it in terms of $\eta$ and $m_x$.

\begin{notation}
Given a complex-valued 
function $b=b(\xi_1,\xi_2)$, we define 
the bilinear Fourier multiplier 
$B$ with symbol $b$ by
\[
B(f,g)(x):=\frac{1}{2\pi}
\iint_{\xR^2}e^{ix(\xi_1+\xi_2)}b(\xi_1,\xi_2)\hat f(\xi_1)\hat g(\xi_2)\, d\xi_1\, d\xi_2.
\]
\end{notation}
\begin{lemma}
The bilinear form $Q_m$ admits the representation
\[
Q_m(\eta)
=\int m_x B^h(\eta,\eta)\, dx,
\]
where $B^h(\eta,\eta)$ is a bilinear Fourier multiplier with symbol
\[
 b^h(\xi,\zeta) =
\frac{ \xi \zeta } { \sinh 2 \xi h \sinh 2 \zeta h}  \frac{
\cosh 2h\xi -\cosh 2h \zeta}{(\xi+\zeta)(\xi-\zeta)}.
\]
\end{lemma}
% \begin{remark}
% It is easily seen that $\symbolb$ is a smooth symbol. Furthermore $h$ appears as a scaling parameter
% \[
% \symbolb(\xi, \zeta)=B_1^1(h\xi, h\zeta).
% \]
% \end{remark}
\begin{proof}Recall that
\[
\widehat{H_D(\eta)}(\xi,y) =  \frac{\sinh \xi (y+h)}{\sinh \xi h}\hat \eta(\xi) , \quad 
\text{resp.}\quad 
\widehat{H_N(\eta)}(\xi,y)
=\frac{\cosh((h+y)\xi)}{\cosh(h\xi)}\hat\eta(\xi).
\]
Consequently, 
\[
\begin{aligned}
\big( &H_N(\eta)_yH_D(\eta)_x-H_N(\eta)_xH_D(\eta)_y\big)
(t,x,y)\\
&=\frac{1}{2\pi}\iint_{\xR^2}e^{ix(\zeta+\xi)}b_0^h(y,\xi,\zeta)\hat\eta(t,\xi)\hat\eta(t,\zeta)\, d\xi\dalpha,
\end{aligned}
\]
where 
\[
\begin{split}
b_0^h(y,\xi,\zeta)= & \ i 
 \xi \zeta  \left(  \frac{\sinh \xi (y+h)}{\sinh \xi h} \frac{\sinh \zeta (y+h)}{\cosh \zeta h} 
-   \frac{\cosh \xi (y+h)}{\sinh \xi h} \frac{\cosh \zeta (y+h)}{\cosh \zeta h} 
    \right)
\\
= & \ 
\frac{- i \xi \zeta } { \sinh \xi h \cosh \zeta h}  \cosh ((y+h)(\xi -\zeta)).
\end{split}
\]
Integrate in $y$ to get 
\[
\int^{0}_{-h} b_0^h(y,\xi,\zeta) \, dy =
 \frac{ -i\xi \zeta } { \sinh \xi h \cosh \zeta h}  \frac{\sinh h(\xi -\zeta)}{\xi-\zeta}.
\]
Notice that for any bilinear Fourier multiplier $B$ with symbol $b$, one has
\[
B(f,f)=B^{sym}(f,f) \quad \text{with}\quad
b^{sym}(\xi_1,\xi_2)=\mez (b(\xi_1,\xi_2)+b(\xi_2,\xi_1)).
$$
By so doing, we obtain that 
$$
\int \big(H_N(\eta)_yH_D(\eta)_x-H_N(\eta)_xH_D(\eta)_y\big)(t,x,y)\, dx 
=B^h_1(\eta,\eta),
\]
where $B^h_1$ is the bilinear Fourier multiplier with symbol 
\[
b^h_1(\xi,\zeta)=\frac{-2 i \xi \zeta } { \sinh 2\xi  \sinh 2\zeta h}  \frac{
\sinh h(\xi+\zeta) \sinh h(\xi -\zeta)}{2(\xi-\zeta)}.
\]
Integrating by parts we obtain 
\[
\int m(x) B^h_1(\eta,\eta) \, dx = \int m_x(x)  B^h(\eta,\eta) \, dx,
\]
where the symbol of $B^h$ is given by 
\[
b^h(\xi,\zeta) = \frac{i}{\xi+\zeta} b^h_1(\xi,\zeta) = \frac{ 2\xi \zeta } 
{ \sinh 2\xi h  \sinh 2\zeta h}  \frac{
\sinh h(\xi+\zeta) \sinh h(\xi -\zeta)}{2(\xi+\zeta)(\xi-\zeta)},
\]
which gives the desired result.
%or equivalently
%\[
% B^h_1(\xi,\zeta) =
%\frac{ \xi \zeta } { \sinh 2 \xi h \sinh 2 \zeta %h}  \frac{
%\cosh 2h\xi -\cosh 2h \zeta}{(\xi+\zeta)%(\xi-\zeta)}.%
%\]
\end{proof}

To conclude the proof of \eqref{dmI}, in light of \eqref{splitting} and the previous lemma, it remains only to prove the following result.

\begin{proposition}\label{p:B1-bd}
For the bilinear form $B^h$ above there exists $c < \frac12$ so that we have
\begin{equation}\label{B1-bd}
\int_0^T\int m_x B^h(\eta,\eta) \, dxdt \geq - c \sup_{x_0 \in \R}\int_0^T \int m_x(x-x_0) \eta^2 \, dxdt.
\end{equation}
\end{proposition}
This concludes the proof of the Proposition~\ref{P:52}. 
It now remains to prove this proposition. 
We remark that we have written this proposition as a separate result in order to be able to apply it 
directly also for the nonlinear problem. 

Our first task is to understand the properties of the symbols $B^h$
and of their kernels $K^h$. The first observation concerning the symbols $b^h$ is that they are all obtained by scaling from a single symbol 
\[
b(\xi,\zeta) = \frac{ 2\xi \zeta } 
{ \sinh 2\xi  \cosh 2\zeta}  \frac{
\sinh (\xi+\zeta) \sinh (\xi -\zeta)}{(\xi+\zeta)(\xi-\zeta)},
\]
as follows,
\[
b^h(\xi,\zeta) = b(h\xi,h\zeta).
\]
Then the kernels $K^h$ are related to the kernel $K$ of $B$ by
\[
K_h(x_1,x_2) = h^{-2} K(h^{-1} x_1,h^{-1} x_2).
\]

Concerning the symbol $b$, one easily sees that it has the following
properties:

\begin{itemize}
\item It is real, even and symmetric.

\item It is uniformly smooth.

\item It decays exponentially away from the axes $\xi = 0$, $\zeta = 0$,
\[
|b(\xi,\zeta)| \leq \frac{1}{1+|\xi| +|\zeta|} e^{-c \min\{|\xi|,|\zeta|\}}.
\]

\item Near $\xi = 0$ it has the expansion
\[
b(\xi,\zeta) = \frac{1}{|\zeta|} \frac{2\xi}{\sinh 2\xi} + O(|\zeta|^{-3}), \qquad |\zeta| \to \infty,
\]
and symmetrically near $\zeta = 0$.
\end{itemize}

Next, we consider the kernel $K$ of $B$, which is the inverse Fourier transform of the symbol $b(\xi,\zeta)$:
\[
K(x_1,x_2)=\frac{1}{(2\pi)^2}\iint e^{ix_1 \xi+i x_2 \zeta}b(\xi,\zeta)\, d\xi\dalpha.
\]
From the above properties of $b$ we the corresponding properties of $K$, which for later reference are collected in the following lemma:

\begin{lemma}\label{l:K}
The kernel $K$ has the following properties:
\begin{enumerate}
\item $K$ is real, even in each variable and symmetric.

\item $K$ is smooth and rapidly decreasing away from the axes $x_1=0$, 
$x_2 = 0$.
%i.e. in the region $\{ \min\{|x_1|,|x_2|\} \gtrsim 1\}$.

\item Near  the axes $x_1=0$, 
$x_2 = 0$ we can expand 
\[
K(x_1,x_2) = - \ln |x_1| \sech^2 x_2  - \ln|x_2| \sech^2 x_1 + 
K^{lip}(x_1,x_2)
\]
where $K^{lip}$ is $C^1$ and decays rapidly, together with its derivatives.
\end{enumerate}
\end{lemma}

We now use these properties to carry out a preliminary step 
in the proof of the Proposition. This is based on the observation
that $B^h$ is primarily localized at frequency $1/h$, which should 
allow us to discard the high frequencies of $\eta$ from $B^h(\eta,\eta)$. Here to fix the meaning of ``high frequencies" we need to choose
a frequency threshold $\lambda_0$ so that $1/h \ll \lambda_0 \ll 1$. 
Then we seek to replace $\eta$ with $\eta_{\leq \lambda_0} = P_{\leq \lambda_0} \eta$. 

Here rather than choosing a sharp frequency localization
operator $P_{\leq \lambda_0}$, we instead choose a localization operator
with a nonnegative kernel; the price to pay for this is to allow harmless rapidly decreasing tails at higher frequency. Then we claim that

\begin{lemma}
If $1/h \ll \lambda_0 \ll 1$ then 
\[
\int_0^T\int m_x B^h(\eta,\eta) \, dxdt = 
\int_0^T\int m_x B^h(\eta_{\leq \lambda_0},\eta_{\leq \lambda_0}) \, dxdt + O(\frac{1}{\lambda_0 h}) \| \eta\|_{LE^0}^2.
\]
\end{lemma}
\begin{proof}
Indeed, consider two dyadic frequencies $h^{-1} \leq \mu
\leq \lambda \lesssim 1$. We will estimate the contribution of
$B(\eta_\lambda,\eta_\mu)$ in terms of the local energy of $\eta$.
For $|\xi| \approx \lambda$ and $|\zeta| \approx \mu$ we have 
\[
|b^h(\xi,\eta)| \lesssim \frac{1}{1+ h \lambda} e^{-c h \mu}
\]
with matching regularity on the same dyadic scale.
Then we have 
\[
\left|   \int_0^T\int m_x B^h(\eta_{\mu},\eta_{\lambda}) \, dxdt  \right| \lesssim \frac{1}{1+ h \lambda} e^{-c h \mu} \|\eta_\mu\|_{LE^0}
\| \eta_\lambda \|_{LE^0} \lesssim \frac{1}{1+ h \lambda} e^{-c h \mu} \|\eta\|_{LE^0}^2.
\]
Then the conclusion of the lemma follows after summation over $\mu > 1/h$, $\lambda > \lambda_0$.

\end{proof}

The last Lemma allows us to localize $\eta$ to low frequencies
on the left in \eqref{B1-bd}. We now investigate the effect
of such a change on the right in \eqref{B1-bd}. The idea here 
is that averaging $\eta$ over a large scale allows us to replace 
the local $L^2$ norm in $x$ by the $L^\infty$ norm. Precisely, we 
have 

\begin{lemma}
For $\lambda_0 \leq 1$ we have 
\[
 \| \eta_{\leq \lambda_0}\|_{L^\infty_x L^2_t}^2 \leq (1 + C \lambda_0) \|\eta\|_{LE^0}^2 .
\]
\end{lemma}

\begin{proof}
Here we take advantage of the fact that the kernel of $P_{\leq \lambda_0}$ is nonnegative and has integral $1$.
Then by the triangle inequality we have 
\[
\| \eta_{\leq \lambda_0}\|_{LE^0} \leq \| \eta\|_{LE^0}.
\]
On the other hand differentiating yields another $\lambda_0$ factor,
\[
\| \partial_{x} \eta_{\leq \lambda_0}\|_{LE^0} \lesssim \lambda_0
\| \eta\|_{LE^0}.
\]
Then by the fundamental theorem of calculus and by the Cauchy-Schwarz's inequality, we compute
\[
\left| \int_{0}^T \eta_{\leq \lambda_0}^2(x,t)\, dt - \int_{0}^T\int m_x(x) \eta_{\leq \lambda_0}^2(x,t)\, dx dt\right| \lesssim \| \eta_{\leq \lambda_0}\|_{LE^0} \|\partial_x  \eta_{\leq \lambda_0}\|_{LE^0} ,
 \]
which implies that
\[
\begin{split}
\int_{0}^T \eta_{\leq \lambda_0}^2(x,t) dt \leq & \ \|\eta_{\leq \lambda_0}\|_{LE^0_0}^2
+ \| \eta_{\leq \lambda_0} \|_{LE^0} \|\partial_x  \eta_{\leq \lambda_0}\|_{LE^0} 
\\
\leq & \ (1+C \lambda_0) \|\eta\|_{LE^0}^2
\end{split}
\]
as needed.
\end{proof}

As a consequence of the last two lemmas, by choosing $1/h \ll \lambda_0 \ll 1$ and using the fact that $\int m_x\, dx = 1$, we can replace the bound \eqref{B1-bd} with 
\[
\int_0^T \int B^h(\eta_{<\lambda_0},\eta_{<\lambda_0})(0)\dt  
\geq -c g\|\eta_{<\lambda_0} \|_{L^\infty_x L^2_t}^2, \qquad 0 < c < \frac12 .
\]
Now we discard the frequency localization; then $h$ becomes a scaling parameter and we can freely set it to $1$. Hence, we have reduced
Proposition~\ref{p:B1-bd} to the following:

\begin{proposition}\label{p:positive}
 The following bound holds: 
\[
\int_0^T \int B(\eta,\eta) \, dx dt  \geq -c \|\eta \|_{L^\infty_x L^2_t}^2, \qquad 0 < c < \frac12 .
\]
\end{proposition}

We first observe that $B(0,0) = \dfrac12$. This implies that 
\[
\int K(x_1,x_2) \, dx_1dx_2 = \frac12.
\]
The key step in the proof of the proposition is the following

\begin{lemma}\label{lemma:Kpositive}
The kernel $K$ is positive.
\end{lemma} 
Before proving this result, let us explain how to conclude the proof of Proposition~\ref{p:positive} with this lemma. 
Firstly, notice that if $K$ is nonnegative, then
\[
\int |K(x_1,x_2)| \, dx_1 dx_2 = \frac12,
\]
and then it is obvious that the proposition holds with $c = \frac12$.
But if $K$ is actually positive, there is a little trick to get a small extra gain. 
Precisely, we can write 
\[
K(x_1,x_2) = K_1(x_1,x_2) + L(x_1) L(x_2),
\]
where $L$ is nonnegative and $K_1$ is still positive. Then the contribution of the $L$ term is nonnegative, while $K_1$ has 
integral $c < \frac12$. Then the conclusion of the proposition follows
for this $c$. We now have to prove the lemma.

\begin{proof}[Proof of Lemma~\ref{lemma:Kpositive}]
  By the symmetries of $K$, it is sufficient to consider the case
  $0 < x \leq y$ (shaded region in the picture).  To compute $K$ we
  view the symbol $b$ as a product of
\[
C_1 = \coth 2\xi \csch 2\zeta -  \csch 2\xi \coth 2\zeta,
\]
and 
\[
D_1 = \frac{\zeta \xi}{\zeta^2 -\xi^2}.
\]
The Fourier transforms of $\coth \xi$ and $\csch \xi$ are $F = \coth x$ respectively  $G = \tanh$, so the 
Fourier transform of $C_1$ is (up to positive constants)
\[
F(x) G(y) - G(x) F(y).
\] 
On the other hand for the Fourier transform of $D_1$ we use the backward 
fundamental solution 
for the wave equation, and then differentiate it in $x$ and $y$. We get
\[
\partial_x \delta_{y+|x| = 0},
\]
which  is supported on  a $\pi/2$ degree angle downward from $0$.
Taking the convolution of the two we get
\[
K(x_0,y_0) = \int_{ y-y_0 = |x-x_0|}  G'(x)F(y) - F'(x) G(y)\, dx ,
\]
where the region of integration $ \circled{1} \cup \circled{2}$ is the upward $\pi/2$ degree angle from $(x_0,y_0)$. (see picture). Here $F$ is singular at $x = 0$, so the second term is 
interpreted in the principal value sense.

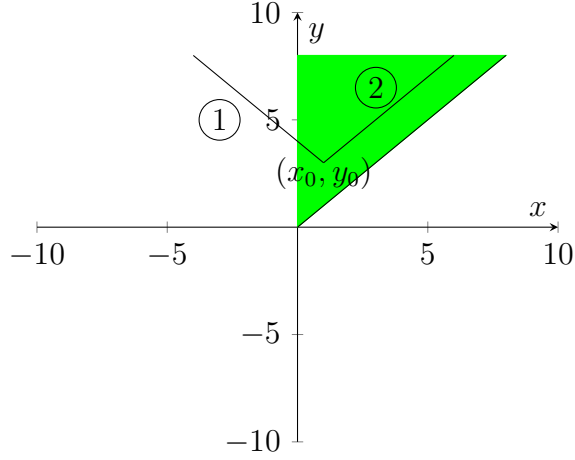
\begin{figure}
  \centering

  \begin{tikzpicture}
    \begin{axis}[ xmin=-10, xmax=10, ymin=-10, ymax=10, axis x
      line = middle, axis y line = middle,  xlabel = $x$, ylabel =$y$, ytick = {}, xtick = {} ]

\filldraw[green, fill=green]  (axis cs:0,0) -- (axis cs:8,8) --  (axis cs:0,8) --  (axis cs:0,0);
 \draw (axis cs:1,3) -- (axis cs:-4,8);
 \node at (axis cs:-3,5)  {\(\circled{1}\)};
 \draw (axis cs:1,3) -- (axis cs:6,8);
 \node at (axis cs:3,6.5)  {\(\circled{2}\)};
 \node at (axis cs:1,2.5) {\((x_0,y_0)\)};
 \draw (axis cs:0,0) -- (axis cs:8,8);
    \end{axis}
  \end{tikzpicture}
  \caption{Integration regions for $K$}
\end{figure}

Based on previous computations, we know that $K$  blows up logarithmically  on the axes and decays exponentially on the diagonals. Then  the positivity of $K$ would be a consequence of the bounds 
\begin{equation}\label{int2}
(\partial_y-\partial_x ) K > 0, \qquad 0 < x_0 \leq y_0,
\end{equation}
within the shaded area of the picture,
respectively 
\begin{equation}\label{int1}
(\partial_y+ \partial_x ) K < 0, \qquad 0 < x_0 = y_0.
\end{equation}

Indeed, we can compute
\[
(\partial_y-\partial_x ) K(x_0,y_0) = \int_{\circled{2}}  F'(y) G'(x) - G'(y)F'(x)\, dx ,
\]
respectively
\[
(\partial_y-\partial_x ) K(x_0,y_0) = \int_{\circled{1}}  F'(y) G'(x) - G'(y)F'(x)\, dx .
\]
Here the first integrand is nonsingular, but the second is again interpreted in the principal value  sense at $x = 0$.

We remark that $G' > 0$, $F' < 0$ and 
\[
\frac{F'(x)}{G'(x)} = - \coth^2(x),
\]
which immediately shows that the above integral over $\circled{2}$ is positive and thus \eqref{int2} holds. Then it remains to establish \eqref{int1} over  the positive half-line $x_0=y_0$.
While the integrand over $\circled{1}$ is also positive pointwise, it has the  distributional $x^{-2}$ type singularity at $x = 0$, which we expect makes the  outcome negative !

To summarize, we need to prove that  the  following integral is negative,
\[
I(x_0) = \int_{-\infty}^{x_0}  F'(y) G'(x) - G'(y)
F'(x)\, dx,  \qquad y = 2x_0 - x.
\]
We separate the analysis into three cases:
\bigskip

\textbf{i) Large $x_0$, $x_0 > 5$.} There $y > 5$, so it is natural
to expand in powers of $e^{-y}$. Since 
$F'(y),G'(y) \approx e^{-2y}$, the leading term in the integrand is $e^{-4 y_0}$ (here we take $x_0 = y_0$).

For $F'$ and $G'$ we have the asymptotic expressions at infinity
\[
F'(x) = - \frac{1}{\sinh^2x} = - \frac{4 e^{-2x}}{(1 - e^{-2x})^2} \approx - 4 e^{-2x} - 8 e^{-4x} ,
\]
\[
G'(x) =  \frac{1}{\cosh^2x} =  \frac{4 e^{-2x}}{(1 + e^{-2x})^2} \approx 4 e^{-2x} - 8 e^{-4x} .
\]
Then for our integral we have the expansion up to $e^{-8y_0}$ terms
\[
\begin{split}
I \approx & \  \int_{-\infty}^{y_0}   \frac{( 4 e^{-2(2y_0-x)} - 8 e^{-4(2y_0-x)})}{\sinh^2 x} 
- \frac{( 4 e^{-2(2y_0-x)} + 8 e^{-4(2y_0-x)})}{\cos^2 x}\, dx
\\
= & \ - 4e^{-4y_0} \int_{-\infty}^{y_0} e^{2x} (\frac{1}{\cosh^2 x} - \frac{1}{\sinh^2 x})\, dx 
- 8 e^{-8y_0}  \int_{-\infty}^{y_0} e^{4x}( \frac{1}{\cosh^2 x} + \frac{1}{\sinh^2 x})\, dx.
\end{split}
\]
By direct computation the first integral gives zero when taken all the way to $+\infty$. Thus, we get
\[
I = 4e^{-4y_0} \int_{y_0}^\infty e^{2x} (\frac{1}{\cosh^2 x} - \frac{1}{\sinh^2 x}) \, dx 
- 8 e^{-8y_0}  \int_{-\infty}^{y_0} e^{4x}( \frac{1}{\cosh^2 x} + \frac{1}{\sinh^2 x}) \, dx + O(e^{-8y_0}).
\]
Now in both integrals the leading contribution comes from $x = y_0$, and has size $e^{-6y_0}$. To compute it we write
\[
\begin{split}
I = &  - 4e^{-4y_0} \int_{y_0}^\infty e^{2x} \frac{1}{\cosh^2 x \sinh^2 x} \, dx 
- 16 e^{-8y_0}  \int_{-\infty}^{y_0} 4 e^{2x}\, dx + O(e^{-8y_0})
\\
= &  - 4e^{-4y_0} \int_{y_0}^\infty 16 e^{-2x} \frac{1}{\cosh^2 x \sinh^2 x}\,  dx 
- 16 e^{-8y_0}  \int_{-\infty}^{y_0} 4 e^{2x}\,dx+ O(e^{-8y_0})
\\
= &  - 32 e^{-6y_0}  - 32 e^{-6y_0}+ O(e^{-8y_0}) = - 64 e^{-6y_0} + O(e^{-8y_0}).
\end{split}
\]

\medskip 

\textbf{i) Small $x_0$, $x_0 < 0.1$.}
In this range we have 
\[
\begin{split}
I = & \  \int_{-1}^{x_0} -  \frac{1}{\sinh^2(2x_0 -x)} + \frac{1}{\sinh^2 x} \, dx  + O(1)
\\
= & \ - 2 \int_{x_0}^1 \frac{1}{\sinh^2 x} \, dx + O(1) = - 2 \coth x_0 + O(1),
\end{split}
\] 
as desired. Here in the first line the expression $\dfrac{1}{\sinh^2 x}$
is interpreted as the distributional derivative $\partial_x (p.v. \cosech x)$.

\medskip 

\textbf{i) Medium $x_0$, $0.1 < x_0 < 5$.}
For the intermediate range we do not have an algebraic proof, but a direct MATLAB computation easily confirms the result. 
\end{proof}

\section{Local energy decay for gravity waves} \label{s:nonlinear}

In this section we  prove our main result in Theorem~\ref{ThmG1}.
We begin by emulating the computation in the previous section for the linear case. We define the functional 
\[
\mI_m^\sigma(t) = \int m(x)( \sigma I_2(x,t) + (1-\sigma) I_3(x,t))\, dx .
\]
Using the density-flux pairs for the momentum, we have 
\[
\partial_t \mI_m^\sigma(t) =  \int m_x ( \sigma S_2(x,t) + (1-\sigma) S_3(x,t)) \, dx.
\]
Hence, in order to prove the theorem we need to establish the following bounds:

\begin{description}
\item[(i)] Fixed time bounds,
\begin{equation}\label{I2-bound}
\left|   \int m(x) I_2 \, dx  \right| \lesssim \| \eta\|_{H^{\frac14}_h} \| \psi_x\|_{H^{-\frac14}_h},
\end{equation}
\begin{equation}\label{I3-bound}
\left|   \int m(x) I_3 \, dx \right| \lesssim \| \eta\|_{H^{\frac14}_h} \| \psi_x\|_{ H^{-\frac14}_h}.
\end{equation}

\item[(ii)] Time integrated bound; for some $\sigma \in (0,1)$  and $c < 1$, we have
\begin{equation} \label{S23-bound}
\int_{0}^T \int m_x ( \sigma S_2(t) + (1-\sigma) S_3(t))\, dx dt \gtrsim LE_0(\eta,\psi) - c LE(\eta,\psi).
\end{equation}
\end{description}

\subsection{ Fixed time bounds} 

The  bound for the contribution of $I_2$ is identical to the one in the linear model. 
For the contribution of $I_3$ there is a slight difference, which is due to the fact that the domain 
of integration is no longer a strip. Hence in order to apply Proposition~\ref{p:extension} we need to switch to 
holomorphic coordinates, and to use Proposition~\ref{p:switch-strips}  in order to relate vertical strips 
in holomorphic vs. euclidean coordinates.

\subsection{ Time integrated bounds} 

As before, here we take $\sigma < \frac12$, but close to
$\frac12$. Using the expressions in Lemmas ~\ref{l:I2},~\ref{l:I3} 
as well as the relations \eqref{phit} and \eqref{thetat}
we write the integral in \eqref{S23-bound} as a combination of two
leading order terms plus error terms
\[
\int_{0}^T \int m_x ( \sigma S_2(t) + (1-\sigma) S_3(t))\, dx dt = LE_{\psi} + g LE_\eta + Err_1 + g Err_2 + Err_3,
\]
where   
\[
LE_\psi := \frac12 \int_0^T \iint m_x [ \sigma(\phi_x^2 - \phi_y^2) + (1-\sigma) |\nabla \phi|^2] \, dx dy dt
\]
\[
LE_\eta := \int_0^T \frac{\sigma}{2}  \int m_x  \eta^2 \, dx - (1-\sigma) \iint m_x \theta_y ( \theta-\HN(\eta))\, dx dy dt  ,
\]
and finally 
\[
\begin{aligned}
&Err_1 :=  \sigma \int_0^T \int  m_x \eta \mN(\eta) \psi \,  dx dt ,\\
&Err_2 := \frac{1-\sigma}{2} \int_0^T \iint m_x \theta_y \HN(|\nabla \phi|^2) \, dx dy dt ,\\
&Err_3 :=  \frac{1-\sigma}{2} \int_0^T \iint m_x \phi_y \HD(\nabla \theta \nabla \phi)\,  dx dy dt .
\end{aligned}
\]

\smallskip

Our strategy in what follows will be to peel off a leading quadratic part, which we interpret using our bounds for the linear equation.
The remaining cubic and higher order expressions will be viewed 
as error terms. All but one of the the cubic error terms will be  estimated perturbatively. 

Finally, the last error term turns out to be unbounded both due to low and to high  high frequencies. For this term we instead apply a partial normal form correction, which replaces it with bounded terms, both time integrated,  and at the endpoints of the time intervals. The latter correspond to a nonlinear normal form modification of the momentum density.
  
For many of the nonlinear estimates it is useful  to switch to holomorphic coordinates.
That greatly facilitates multilinear analysis. There is a price to pay for that, as our $m_x$ cutoff is vertical in the Eulerian frame,
 but not in the holomorphic frame.

For the remainder of this section we reduce the nonlinear estimate to the linear estimates 
in Section~\ref{s:linear}, plus a number of error terms, which need to be estimated perturbatively.
The last two sections are devoted to the proof of the error estimates. In Section~\ref{s:switch}
we show that the Eulerian local energy norms admit equivalent counterparts in the holomorphic setting,
and use this equivalence and multilinear analysis to estimate some of the error terms.
Finally, in Section~\ref{s:F} we deal with the more difficult error terms which involve the function $F$,
and arise out of the normal form analysis.

%%%%%%%%%%%%%%%%%%%%%%%%%%%%%%%%%%%%%%%%%%%%%%%%%%%%%%%%%%%%%%%%%%%%%%%%%%
%%%%%%%%%%%%%%%%%%%%%%%%%%%%%%%%%%%%%%%%%%%%%%%%%%%%%%%%%%%%%%%%%%%%%%%%%%
%%%%%%%%%%%%%%%%%%%%%%%%%%%%%%%%%%%%%%%%%%%%%%%%%%%%%%%%%%%%%%%%%%%%%%%%%%
%%%%%%%%%%%%%%%%%%%%%%%%%%%%%%%%%%%%%%%%%%%%%%%%%%%%%%%%%%%%%%%%%%%%%%%%%5
\subsection{ The \texorpdfstring{$LE_\theta$}{}  term}
%%%%%%%%%%%%%%%%%%%%%%%%%%%%%%%%%%%%%%%%%%%%%%%%%%%%%%%%%%%%%%%%%%%%%%%%%%
%%%%%%%%%%%%%%%%%%%%%%%%%%%%%%%%%%%%%%%%%%%%%%%%%%%%%%%%%%%%%%%%%%%%%%%%%%
%%%%%%%%%%%%%%%%%%%%%%%%%%%%%%%%%%%%%%%%%%%%%%%%%%%%%%%%%%%%%%%%%%%%%%%%%%
%%%%%%%%%%%%%%%%%%%%%%%%%%%%%%%%%%%%%%%%%%%%%%%%%%%%%%%%%%%%%%%%%%%%%%%%%5

Here we need to compare the contribution of $\HN(\eta)$,
\[
I_1 = \int_0^T \iint m_x \theta_y ( \HN(\eta) - \theta)\, dx dy =  \iint m_x \theta ( \theta - \HN(\eta))_y\, dx dy dt,
\]
with the expression
\[
 \int_0^T \int \frac12 m_x \eta^2\, dx dt
\]
from the first term in $LE_\theta$.

We remark that $\HN(\eta)$ and $\theta$ solve the same equation and have the same boundary 
condition on the top, but different boundary conditions on the bottom (Dirichlet, respectively Neumann).
Thus they cancel in the infinite depth case, but not in the finite depth case.

To estimate this  we move to conformal coordinates $ z = \alpha +
i\beta$.  This does not change the equations for $\HN(\eta)$ and
$\theta$.  Precisely, if $\alpha_0$ is the image of $x_0$ in the
conformal setting, then we seek to compare the integral $I_1$ with its
conformal counterpart
\[
I_1^{hol} =  \int_0^T \iint m_\alpha \, (\alpha- \alpha_0)  \theta ( \theta - \HN(\eta))_\beta \, d\alpha d\beta dt.
\]
We will view the difference between the two integrals as an error term,
\[
Err_4 = I_1 - I_1^{hol}
\]
to be estimated later.

The expression $I_1^{hol}$ can be rewritten as
\[
\begin{split}
I_1^{hol} = & \ \int_0^T \iint m_\alpha \, \theta  ( \theta - \HN(\eta))_\beta  \, d\alpha d\beta dt 
\\ = & \   \int_0^T  \iint  m \, \theta_{\alpha} 
( \HN(\eta) - \theta)_\beta  + m \, \theta  ( \HN(\eta) - \theta)_{\alpha \beta}     \, d\alpha d\beta dt
\\
= & \  \int_0^T  \iint  m \,\theta_{\alpha}  ( \HN(\eta) - \theta)_\beta  -  m \, \theta_\beta  ( \HN(\eta) - \theta)_\alpha \,     d\alpha d\beta
\\
= & \ \int_0^T  \iint  m \,(\theta_{\alpha} \HN(\eta)_{\beta}  -  \theta_\beta   \HN(\eta)_{\alpha}) \,    d\alpha d\beta dt .
\end{split}
\]
Recalling that $\theta = \HD(\eta)$,  the above integral becomes
\[
I_1^{hol} = \int_0^T \iint m 
\HD(\eta)_{\alpha} \HN(\eta)_{\beta}  -  \HD(\eta)_\beta   \HN(\eta)_{\alpha})
\,   d\alpha d\beta dt,
\]
which is identical to the corresponding expression obtained in the analysis of the linearized problem in Section~\ref{s:linear}.
Hence, as there, it can be further represented as 
\[
I_1^{hol} = \int_0^T \int m_\alpha(\alpha-\alpha_0) 
 \, B^h(\eta,\eta)  \, d\alpha dt .
\]

On the other hand,  as a consequence of  the bound $|W_\alpha| \lesssim \epsilon$  we obtain the relation
\[
\int_0^T \int \frac12 m_x \eta^2 \, dx = \int_0^T \int \frac12 m_\alpha \eta^2 \, d\alpha dt + O(\epsilon)  \| \eta\|_{LE}^2.
\]
Combining the two terms, we have established that
\[
LE_\theta - Err_4 =  \int_0^T \frac{\sigma}{2}  \int m_\alpha  \eta^2\,  d\alpha - (1-\sigma) \int_0^T \int m_\alpha(\alpha-\alpha_0) 
B^h_2(\eta,\eta)\,   d\alpha   dt  + O(\epsilon)  \| \eta\|_{LE}^2.
\]
We conclude the argument here by showing that for $\sigma > \frac12$
close to $\frac12$ we have the bound
\begin{equation}\label{LE-theta}
LE_\theta - Err_4  \gtrsim \frac12 \| \eta\|^2_{LE_{x_0}} - c \|\eta\|_{LE}^2,
\end{equation}
where $c < \frac12$ is a universal constant. This in turn 
is a consequence of

\begin{proposition}\label{p:B1-bd-nl}
For the bilinear form $B^h_2$ above there exists $c < \frac12$ so that we have
\begin{equation}\label{B1-bd-nl}
\int_0^T\int m_\alpha(\alpha - \alpha_0) B^h_2(\eta,\eta) \, d \alpha dt \geq - c g \sup_{x_1 \in \R}\int_0^T \int m_\alpha(\alpha-\alpha_1) \eta^2 \, d\alpha dt.
\end{equation}
\end{proposition}

This is a direct counterpart of Proposition~\ref{p:B1-bd} from the 
linear analysis. The only difference is that on the right, $\alpha_1$ 
is not constant in time but instead we have that $\alpha_1 = \alpha_1(t,x_1)$. Because of this we cannot directly cite 
Proposition~\ref{p:B1-bd} here. However, it will be easy to reduce 
the above proposition to Proposition~\ref{p:B1-bd}.

\begin{proof}
To reduce to Proposition~\ref{p:B1-bd}  we simply change coordinates
back into Eulerian coordinates. The Jacobian is $1+ \Re W_\alpha = 1+O(\epsilon)$ so it only yields negligible $O(\epsilon)$ errors.
The same applies for the changes in the argument of $m$,
\[
m_\alpha(\alpha-\alpha_0) = m_x(x-x_0) + O(\epsilon).
\]
It remains to consider the change in the operator $B^h$. We consider 
this at the level of the kernel $K^h$ of $B^h$. Referring back 
to Section~\ref{s:linear}, the kernel of $K^h$ in the holomorphic coordinates is
\[
K^h(\alpha_1,\alpha_2;\alpha) = K\Big(\frac{\alpha_1 -\alpha}h,
\frac{\alpha_2 -\alpha}h\Big).
\]
After the change of coordinates this becomes
\[
\tilde K_h (x_1,x_2;x) := K\Big(\frac{\alpha(x_1,t) -\alpha(x,t)}h,
\frac{\alpha(x_2,t) -\alpha(x,t)}h\Big).
\]
We would like to replace this with $K^{h}(x_1,x_2,x)$ at the expense 
of $O(\epsilon)$ errors.  For this we use the relations
\[
\alpha(x_i,t) -\alpha(x,t) = (x_i -x)(1+O(\epsilon).
\]
Then we compute using the properties
of $K$ in Lemma~\ref{l:K}:
\[
|\tilde K_h (x_1,x_2;x) -K^{h}_2(x_1,x_2,x)| \lesssim \epsilon h^{-2}( 1 + h^{-1} (|x-x_1| + |x-x_2|)^{-N}).
\]
This easily gives $O(\epsilon)$ errors, and finally allows us to reduce 
 the proposition to Proposition~\ref{p:B1-bd}.
\end{proof}

%%%%%%%%%%%%%%%%%%%%%%%%%%%%%%%%%%%%%%%%%%%%%%%%%%%%%%%%%%%%%%%%%%%%%%%%%%%%%%
%%%%%%%%%%%%%%%%%%%%%%%%%%%%%%%%%%%%%%%%%%%%%%%%%%%%%%%%%%%%%%%%%%%%%%%%%%%%%%
%%%%%%%%%%%%%%%%%%%%%%%%%%%%%%%%%%%%%%%%%%%%%%%%%%%%%%%%%%%%%%%%%%%%%%%%%%%%%%
%%%%%%%%%%%%%%%%%%%%%%%%%%%%%%%%%%%%%%%%%%%%%%%%%%%%%%%%%%%%%%%%%%%%%%%%%%%%%%
\subsection{ The error terms} 
%%%%%%%%%%%%%%%%%%%%%%%%%%%%%%%%%%%%%%%%%%%%%%%%%%%%%%%%%%%%%%%%%%%%%%%%%%%%%%
%%%%%%%%%%%%%%%%%%%%%%%%%%%%%%%%%%%%%%%%%%%%%%%%%%%%%%%%%%%%%%%%%%%%%%%%%%%%%%
%%%%%%%%%%%%%%%%%%%%%%%%%%%%%%%%%%%%%%%%%%%%%%%%%%%%%%%%%%%%%%%%%%%%%%%%%%%%%%
%%%%%%%%%%%%%%%%%%%%%%%%%%%%%%%%%%%%%%%%%%%%%%%%%%%%%%%%%%%%%%%%%%%%%%%%%%%%%%

At this point we have four error terms to deal with, $Err_1$, $Err_2$, $Err_3$ and $Err_4$.
Three of them will be directly estimated in a perturbative fashion:

\begin{proposition}\label{p:e124}
We have the following estimates:
\begin{equation}\label{e124}
|Err_1| + |Err_2| + |Err_4| \lesssim \epsilon \|(\eta,\psi)\|^2_{LE}.
\end{equation}
\end{proposition}

This proposition is proved in the following section.

The difficult term is $Err_3$, which turns out to be unbounded both because of low frequency contributions 
and high frequency contributions. We will  address this difficulty in two steps. The first is to switch to 
the holomorphic coordinates counterpart of $Err_3$. The second is to apply a nonlinear normal form type correction to the momentum density. 

For the first step, the holomorphic counterpart of $Err_3$ is
\[
Err_3^{hol} :=   \int_0^T \iint m_\alpha (\alpha - \alpha_0)  \phi_y \HD(\nabla \theta \nabla \phi)\, d\alpha d\beta dt.
\]
On the top we have $\phi_y = \Im R$, while for $\nabla \theta \nabla \phi$ we compute its value as 
\[
\nabla \theta \nabla \phi = J^{-1} (\nabla_h \theta \nabla_h \phi) = J^{-1} \Im (\bar W_\alpha Q_\alpha) = 
\Im ( \frac{\bar W_\alpha}{1+\bar W_\alpha} R).
\]
Therefore we obtain
\[
Err_3^{hol} =   \int_0^T \iint m_\alpha(\alpha - \alpha_0)  \Im R \HD\left( \frac{\bar W_\alpha}{1+\bar W_\alpha} R
\right) \,  d\alpha d\beta dt .
\]
The transition between $Err_3$ and $Err_3^{hol}$ is harmless:

 \begin{proposition}\label{p:e3-diff}
We have the following estimate:
\begin{equation}
\label{e3:diff}
|Err_3 - Err_3^{hol}| \lesssim \epsilon \|(\eta,\psi)\|_{LE}.
\end{equation}
\end{proposition}

Next we turn our attention to the remaining unbounded error term $Err_3^{hol}$. Here we will borrow 
an idea from normal forms, and rectify this error via a normal form type correction. Since we are trying
to address both low and high frequencies, our correction will be genuinely nonlinear as opposed 
to the traditional cubic one, which would only address the low frequencies.

Our correction is based on the following computation, which uses the equations \eqref{FullSystem-re}:
\[
\begin{split}
\frac{d}{dt} (\Im W  \Re W_\alpha) = 
& \ \partial_\alpha (\Im W \Re W_t) +   \Im (W_t  \bar W_\alpha)
\\
= & \ \partial_\alpha (\Im W \Re W_t)  - \Im ( F(1+W_\alpha)  \bar W_\alpha)
\\
= & \ \partial_\alpha (\Im W \Re W_t)   -  \Im F |W_\alpha|^2
- \Im( F \bar W_\alpha)
\\ 
= & \ \partial_\alpha (\Im W \Re W_t)   -  \Im F (|W_\alpha|^2 + 2 \Re W_\alpha)
+ \Im( F W_\alpha)
\\
= & \ \partial_\alpha (\Im W \Re W_t)   -  \Im Q_\alpha J^{-1} (|W_\alpha|^2 + 2 \Re W_\alpha)
+ \Im( F W_\alpha)
\\
= & \ \partial_\alpha (\Im W \Re W_t)   -  \Im \left( R \frac{\bar W_\alpha}{1+\bar W_\alpha }\right) - \Im (R W_\alpha)
+ \Im( F W_\alpha)
\\
= & \ \partial_\alpha (\Im W \Re W_t)   -  \Im \left( R \frac{\bar W_\alpha}{1+\bar W_\alpha }\right) + \Im ((F-R) W_\alpha).
\end{split}
 \]
This allows us to express $\displaystyle{\Im ( \dfrac{\bar W_\alpha}{1+\bar W_\alpha} R})$ on the top as
\[
2 \Im ( \frac{\bar W_\alpha}{1+\bar W_\alpha} R) = 
- \frac{d}{dt} (\Im W  \Re W_\alpha) + \partial_\alpha (\Im W \Re W_t)   +  \Im( (F-R) W_\alpha).
\]
The first expression on the right will correspond to  our (partial) normal 
form correction to the Morawetz 's identity. The second has an $\alpha$ derivative, and thus better low frequency decay. Finally, the third is the imaginary part of a 
holomorphic function, so it has a trivial holomorphic extension.

Correspondingly, we can write $Err_3^{hol}$ in the form
\begin{equation}\label{err3-dec}
\begin{split}
2Err_3^{hol} =  & \  \left. \iint m_\alpha  \Im R \HD\left( \Im W \Re W_\alpha\right)   d\alpha d\beta \right|_0^T
+ Err_5 + Err_6 + Err_7,
\end{split}
\end{equation}
where 
\[
\begin{aligned}
&Err_5:=   \int_0^T \iint m_\alpha  \Im R_t \,\HD\left( \Im W \Re W_\alpha
\right) \,  d\alpha d\beta dt,\\
&Err_6:=   \int_0^T \iint m_\alpha  \Im R \, \Im( (F-R) W_\alpha)\,
 d\alpha d\beta dt,\\
&Err_7:=   \int_0^T \iint m_\alpha  \Im R\, \partial_\alpha \,\HD\left( \Im W \Re W_t)
\right)\,   d\alpha d\beta dt.
\end{aligned}
\]
The first term in \eqref{err3-dec} can be estimated directly using Proposition~\ref{p:extension},
\[
\begin{split}
 \left| \iint m_\alpha  \Im R\, \HD\left( \Im W \Re W_\alpha\right)  \, d\alpha d\beta \right|
\lesssim  & \ \| R\|_{H^{-\frac14}} \| \Im W \Re W_\alpha\|_{H^\frac14} 
\\ \lesssim & \ \| R\|_{H^{-\frac14}} \| \Im W \|_{H^\frac14} 
\lesssim E^\frac14,
\end{split}
\]
since $\Re W_{\alpha} \in l^1H^{\frac{1}{2}}_h$, due to the multiplicative estimate
\[
\| fg\|_{H^{\frac{1}{4}}_h} \lesssim \| f\|_{H^{\frac{1}{4}}_h} \|g\|_{ l^1H^{\frac{1}{2}}_h}.
\]

Then it remains to estimate the error terms:

\begin{proposition}\label{p:e567}
We have the following estimates:
\begin{equation}
|Err_5| + |Err_6| + |Err_7| \lesssim \epsilon \|(\eta,\psi)\|^2_{LE}.
\end{equation}
\end{proposition}

All of these errors involve the expression $F$, since in the fluid domain we have 
\[
W_t = F(1+W_\alpha),
\]
for $W$, respectively
\[
\begin{split}
R_t = &  \ \frac{1}{1+W_\alpha}( Q_{\alpha t} - R W_{\alpha t}) 
\\
 = &  \ \frac{1}{1+W_\alpha} ((- FQ_\alpha)_\alpha + R (F(1+W_\alpha))_{\alpha }   + g \Til W_\alpha + P[|R|^2]_\alpha) \\
 = &  \ - F R_\alpha +\frac{1}{1+W_\alpha} (- g \Til W_\alpha + P[|R|^2]_\alpha)
\end{split}
\]
for $R$. Corresponding to the last relation, we split
\[
Err_5 = Err_5^1 + Err_5^2 + Err_5^3 .
\]

Thus, we have proved Theorem~\ref{ThmG1} modulo the results in Propositions~\ref{p:e124}, 
\ref{p:e3-diff} and ~\ref{p:e567}.

%%%%%%%%%%%%%%%%%%%%%%%%%%%%%%%%%%%%%%%%%%%%%%%%%%%%%%%%%%%%%%%%%%%%%%%%%%
%%%%%%%%%%%%%%%%%%%%%%%%%%%%%%%%%%%%%%%%%%%%%%%%%%%%%%%%%%%%%%%%%%%%%%%%%%
%%%%%%%%%%%%%%%%%%%%%%%%%%%%%%%%%%%%%%%%%%%%%%%%%%%%%%%%%%%%%%%%%%%%%%%%%%
%%%%%%%%%%%%%%%%%%%%%%%%%%%%%%%%%%%%%%%%%%%%%%%%%%%%%%%%%%%%%%%%%%%%%%%%%5
%%%%%%%%%%%%%%%%%%%%%%%%%%%%%%%%%%%%%%%%%%%%%%%%%%%%%%%%%%%%%%%%%%%%%%%%%%
%%%%%%%%%%%%%%%%%%%%%%%%%%%%%%%%%%%%%%%%%%%%%%%%%%%%%%%%%%%%%%%%%%%%%%%%%%
%%%%%%%%%%%%%%%%%%%%%%%%%%%%%%%%%%%%%%%%%%%%%%%%%%%%%%%%%%%%%%%%%%%%%%%%%%
%%%%%%%%%%%%%%%%%%%%%%%%%%%%%%%%%%%%%%%%%%%%%%%%%%%%%%%%%%%%%%%%%%%%%%%%%5
\section{Local energy bounds in holomorphic coordinates}
\label{s:switch}
%%%%%%%%%%%%%%%%%%%%%%%%%%%%%%%%%%%%%%%%%%%%%%%%%%%%%%%%%%%%%%%%%%%%%%%%%%
%%%%%%%%%%%%%%%%%%%%%%%%%%%%%%%%%%%%%%%%%%%%%%%%%%%%%%%%%%%%%%%%%%%%%%%%%%
%%%%%%%%%%%%%%%%%%%%%%%%%%%%%%%%%%%%%%%%%%%%%%%%%%%%%%%%%%%%%%%%%%%%%%%%%%
%%%%%%%%%%%%%%%%%%%%%%%%%%%%%%%%%%%%%%%%%%%%%%%%%%%%%%%%%%%%%%%%%%%%%%%%%5
%%%%%%%%%%%%%%%%%%%%%%%%%%%%%%%%%%%%%%%%%%%%%%%%%%%%%%%%%%%%%%%%%%%%%%%%%%
%%%%%%%%%%%%%%%%%%%%%%%%%%%%%%%%%%%%%%%%%%%%%%%%%%%%%%%%%%%%%%%%%%%%%%%%%%
%%%%%%%%%%%%%%%%%%%%%%%%%%%%%%%%%%%%%%%%%%%%%%%%%%%%%%%%%%%%%%%%%%%%%%%%%%
%%%%%%%%%%%%%%%%%%%%%%%%%%%%%%%%%%%%%%%%%%%%%%%%%%%%%%%%%%%%%%%%%%%%%%%%%5

As a first step in the proof of the error estimates needed for our
main theorem, in this section we seek to understand how to transfer
the local energy bounds to the holomorphic setting. Then we will also
consider some bilinear expressions, and use them to estimate the
simpler error terms.

\subsection{Notations}

Our starting point here is represented by the local energy norms
in the Eulerian setting, which, are equivalently defined as 
\[
\|(\eta,\psi)\|_{LE} = \|\eta\|_{LE^0} + \| \nabla \phi\|_{LE^{-\frac12}},
\]
where 
\[
\|\eta\|_{LE^0} := \sup_{x_0 \in \R} \| \eta\|_{L^2(\HS(x_0))},
\qquad
\| \nabla \phi\|_{LE^{-\frac12}} := \sup_{x_0 \in \R} \| \nabla \phi\|_{L^2(\HVS(x_0))}.
\]
Here $\HS(x_0)$, respectively $\HVS(x_0)$ represent the Eulerian strips
\[
\HS(x_0) := \{ [0,T] \times [x_0-1,x_0+1] \}, \qquad
\HVS(x_0) := \{ [0,T] \times [x_0-1,x_0+1] \times [-h,0]\}.
\]

Our first objective will be to prove that these norms are equivalent to their 
counterparts in the holomorphic setting. In holomorphic coordinates the 
functions $\eta$ and $\nabla \phi$ are represented by $\Im W$ and $R$. 
Thus we will seek to replace the above local energy norm with
\[
\|(W,R)\|_{LE} := \|\Im W\|_{LE^0} + \| R \|_{LE^{-\frac12}},
\]
where 
\[
\|\Im W\|_{LE^0} := \sup_{x_0 \in \R} \| \Im W \|_{L^2(\HS_h(x_0))},
\qquad
\| R \|_{LE^{-\frac12}} := 
\sup_{x_0 \in \R} \| R \|_{L^2(\HVS_h(x_0))}.
\]
Here $\HS_h(x_0)$, respectively $\HVS_h(x_0)$, represent the holomorphic strips
\[
\HS_h(x_0) := \{(t,\alpha): t \in [0,T], \ \alpha \in [\alpha_0-1,\alpha_0+1]\}, \qquad
\HVS_h(x_0) := \HS(x_0) \times [-h,0],
\]
where $\alpha_0 = \alpha_0(t,x_0)$ represents the holomorphic coordinate of $x_0$, which in general will depend on $t$. 

We remark that while the strips $\HS_h(x_0)$ on the top roughly correspond 
to the image of $\HS(x_0)$ in holomorphic coordinates, this is no longer the case for the strips $\HVS_h(x_0)$ relative to $\HVS(x_0)$. While these 
are well matched on the top, in depth there may be a horizontal drift, which has been estimated in Proposition~\ref{p:switch-strips}.

The first main outcome of this section will be the equivalence
\begin{proposition}\label{p:equiv}
Assuming the uniform bound \eqref{uniform}, we have the equivalence:
\begin{equation}\label{equiv}
\|(\eta,\psi)\|_{LE} \approx \|(W,R)\|_{LE}.
\end{equation}
\end{proposition}

Here the correspondance between the $LE^0$ norms of $\eta$ and $\Im W$ is straightforward
due to the bilipschitz property of the conformal map. However, the correspondence between the $LE^{-\frac12}$ norms of $\nabla \phi$ and $R$
is less obvious, and is proved in Proposition \ref{p:R-LE} below.

One difference between the norms for $\Im W$ and for $R$ is that they are 
expressed in terms of the size of the function on the top, respectively in depth. For the purpose of multilinear estimates later on we will need access
to both types of norms for both $\Im W$ and for $R$. Since the local energy norms are defined using the unit spatial scale, in order to describe the 
behavior of functions in these spaces we will differentiate between high frequencies and low frequencies. We begin with functions on the top:

\medskip

a) \textbf{ High frequency characterization on top.} Here we will use local norms
on the top, for which we will use the abbreviated notation
\[
\| u \|_{L^2_t H^s_{loc}} := \sup_{x_0 \in \R } \| u\|_{L^2_t H^s_{\alpha}([\alpha_0-1, \alpha_0+1])},
\]
where again $\alpha_0 = \alpha_0(x_0,t)$.

\medskip

b) \textbf{ Low frequency characterization on top.} Here we will use local norms
on the top to describe the frequency $\lambda$ or $\leq \lambda$ part of functions, where $\lambda < 1$ is a dyadic frequency. By the 
uncertainty principle such bounds should be uniform on the $\lambda^{-1}$ 
spatial scale. Then it is natural to use the following norms:
\[
\| u \|_{L^2_t L^\infty_{loc}(B_\lambda)} := \sup_{x_0 \in \R } \| u\|_{L^2_t L^\infty_{\alpha}(B_\lambda(x_0))},
\]
where
\[
B_\lambda(x_0) := \{ \alpha \in \R: \ |\alpha-\alpha_0| \lesssim \lambda^{-1} \}.
\]
We remark that the local norms in $a)$ correspond exactly to the $B_{\lambda}(x_0)$ norms with $\lambda =1$.

\bigskip
Next we consider functions in the strip which are harmonic extensions 
of functions on the top.

a1) \textbf{ High frequency characterization in strip.} Here we will use local norms on regions with depth at most $1$, for which we will use the abbreviated notation
\[
\| u \|_{L^2_t X_{loc}(A_1)} := \sup_{x_0 \in \R } \| u\|_{L^2_t X(A_1(x_0))},
\]
where $X$ will represent various Sobolev norms and
\[
A_1(x_0) := \{(\alpha,\beta): |\beta| \lesssim 1,\ |\alpha-\alpha_0| \lesssim 1 \}.
\]

\medskip

b1) \textbf{ Low frequency characterization in strip.} Here a 
frequency $\lambda < 1$ is associated with depths $|\beta| \approx \lambda^{-1}$. Thus, we define the regions
\[
A_{\lambda}(x_0) = \{ (\alpha,\beta): |\beta| \approx \lambda^{-1},\  |\alpha-\alpha_0| \lesssim \lambda^{-1} \}, \qquad \lambda < 1,
\]
and in these regions we use the uniform norms,
\[
\| u \|_{L^2_t L^\infty_{loc}(A_\lambda)} := \sup_{x_0 \in \R }
\| u\|_{L^2_t L^\infty_{\alpha,\beta}(A_\lambda(x_0))}.
\]

We will also denote 
\[
\begin{aligned}
&\mathbf{B}^1(x_0) := \{ (\alpha,\beta); \ |\alpha - \alpha_0| \leq 1, \ \beta \in [-1,0] \} ,\\
&\mathbf{B}^\lambda (x_0) := \{ (\alpha,\beta); \ |\alpha - \alpha_0| \leq \lambda^{-1}, \ \beta \in [-\lambda^{-1},0] \}, \mbox{ for } \lambda < 1. 
\end{aligned}
\]

To simplify the notations in the following analysis,
we will also denote
\begin{equation}\label{constants}
\| (\eta,\psi)\|_{LE} := M, \qquad \|(\eta,\psi)\|_{X} := \epsilon \leq \epsilon_0 \ll 1.
\end{equation}

Given the equivalence of the $X$ norms in Proposition~\ref{uniform-hol},
as well as the equivalence of the $LE$ norms in the next subsection,
these bounds also transfer to the holomorphic setting as follows:

\begin{equation}\label{constants-hol}
\| (\Im W,R)\|_{LE} \lesssim M, \qquad \|(\Im W,R)\|_{X} \lesssim \epsilon \ll 1.
\end{equation}
Furthermore, we recall that the frequency envelopes $\{ c_\lambda \}$ for $(\eta,\psi)$ in $X$ also transfer to  $(\Im W,R)$ in $X$.

%%%%%%%%%%%%%%%%%%%%%%%%%%%%%%%%%%%%%%%%%%%%%%%%%%%%%%%%%%%%%%%%%%%%%%%%%%
%%%%%%%%%%%%%%%%%%%%%%%%%%%%%%%%%%%%%%%%%%%%%%%%%%%%%%%%%%%%%%%%%%%%%%%%%%
%%%%%%%%%%%%%%%%%%%%%%%%%%%%%%%%%%%%%%%%%%%%%%%%%%%%%%%%%%%%%%%%%%%%%%%%%%
%%%%%%%%%%%%%%%%%%%%%%%%%%%%%%%%%%%%%%%%%%%%%%%%%%%%%%%%%%%%%%%%%%%%%%%%%5
\subsection{Multipliers and Bernstein's inequality in uniform norms}
%%%%%%%%%%%%%%%%%%%%%%%%%%%%%%%%%%%%%%%%%%%%%%%%%%%%%%%%%%%%%%%%%%%%%%%%%%
%%%%%%%%%%%%%%%%%%%%%%%%%%%%%%%%%%%%%%%%%%%%%%%%%%%%%%%%%%%%%%%%%%%%%%%%%%
%%%%%%%%%%%%%%%%%%%%%%%%%%%%%%%%%%%%%%%%%%%%%%%%%%%%%%%%%%%%%%%%%%%%%%%%%%
%%%%%%%%%%%%%%%%%%%%%%%%%%%%%%%%%%%%%%%%%%%%%%%%%%%%%%%%%%%%%%%%%%%%%%%%%5

Here we aim to understand how multipliers act on the uniform spaces
defined above. 

We will work with a multiplier $M_{\lambda_2}(D)$ associated to a dyadic frequency $\lambda_2$. In order to be able to use the bounds in several circumstances,  we make a weak assumption on their (Lipschitz) 
symbols $m_{\lambda_2}(\xi)$:

\begin{equation}\label{m-symbol}
\begin{aligned}
|m_{\lambda_2}(\xi)| \lesssim  \ (1+\lambda_2^{-1}|\xi|)^{-3}, \ 
\mbox{  and  }\ 
| \partial_\xi^{k+1} m_{\lambda_2}(\xi) | \lesssim c_k  \ |\xi|^{-k}(1+\lambda_2^{-1}|\xi|)^{-4}. 
\end{aligned}
\end{equation}
Examples of such symbols include
\begin{itemize}
\item Littlewood-Paley localization operators $P_{\lambda_2}$, $P_{\leq \lambda_2}$.
\item The multipliers $p_D(\beta, D)$ and $p_N(\beta, D)$ in subsection~\ref{ss:laplace}
with $|\beta| \approx \lambda_2^{-1}$.
\end{itemize}

We will separately consider high frequencies, where 
we work with the spaces  $L^2_t L^p_{loc}$, 
and low frequencies, where we work with the spaces  
$L^2_t L^p_{loc}(B_\lambda)$ associated with a dyadic 
frequency $1/h \leq \lambda \leq 1$.

\bigskip

\textbf{A. High frequencies.}
Here we consider a dyadic high frequency $1/h \leq \lambda_2 \leq 1$,
and seek to understand how multipliers $M_{\lambda_2}(D)$
associated to frequency $\lambda_2$ act on the spaces $L^2_t L^p_{loc}$.

\begin{lemma}\label{l:bern-loc-hi}
Let $1/h \leq \lambda_1,\lambda_2 \leq 1$ and $1 \leq p \leq q \leq \infty$.  Then
\begin{equation}
\| M_{\lambda_2}(D) \|_{L^2_t L^p_{loc} \to 
L^2_t L^q_{loc}} \lesssim \lambda_2^{\frac1p-\frac1q}.
\end{equation}
\end{lemma}

\bigskip

\textbf{B. Low frequencies.}
Here we consider two dyadic low frequencies $1/h \leq \lambda_1, \lambda_2 \leq 1$,
and seek to understand how multipliers $M_{\lambda_2}(D)$
associated to frequency $\lambda_2$ act on the spaces $L^2_t L^p_{loc}(B_{\lambda_1})$. For such multipliers we have:

\begin{lemma}\label{l:bern-loc}
Let $1/h \leq \lambda_1,\lambda_2 \leq 1$ and $1 \leq p \leq q \leq \infty$.

a) Assume that $\lambda_1 \leq \lambda_2$. Then
\begin{equation}
\| M_{\lambda_2}(D) \|_{L^2_t L^p_{loc}(B_{\lambda_1}) \to 
L^2_t L^q_{loc}(B_{\lambda_1})} \lesssim \lambda_2^{\frac1p-\frac1q}.
\end{equation}

b) Assume that $ \lambda_2 \leq \lambda_1$. Then
\begin{equation}
\| M_{\lambda_2}(D) \|_{L^2_t L^p_{loc}(B_{\lambda_1}) \to 
L^2_t L^q_{loc}(B_{\lambda_2})} \lesssim \lambda_1^{\frac1p} \lambda_2^{-\frac1q}.
\end{equation}
\end{lemma}
We remark that part (a) is nothing but the classical Bernstein's
inequality in disguise, as the multiplier $M_{\lambda_2}
$ does not mix
$\lambda_1^{-1}$ intervals. Part (b) is the more interesting one,
where the $\lambda_1^{-1}$ intervals are mixed.
\begin{proof}[Proof of Lemmas~\ref{l:bern-loc-hi},\ref{l:bern-loc}]
We first note that Lemma~\ref{l:bern-loc-hi} can be viewed a a particular
case of Lemma~\ref{l:bern-loc} (a) with $\lambda_1 = 1$. So in what follows we will only prove Lemma~\ref{l:bern-loc}.

A direct consequence of the symbol bounds \eqref{m-symbol}
is the fact that the kernel $K_{\lambda_2}$ of $M_{\lambda_2}(D)$ satisfies 
the bound
\begin{equation}\label{m-kernel}
|K_{\lambda_2}(\alpha)| \lesssim \frac{\lambda_2}{1+ \lambda_2^2 \alpha^2}.
\end{equation}
We will show that \eqref{m-kernel} yields the conclusion of the Lemma.

a)  We fix $x_0 \in \R$ and seek to estimate 
\[
\| M_{\lambda_2}(D) u\|_{L^2_t L^q_{loc}(B_{\lambda_1}(x_0))}.
\]
For that we cover $S \times [0,T]$ with width $\lambda_1^{-1}$ strips,
\[
S \times [0,T]= \bigcup_{j \in \mathbb{Z}} S_{\lambda_1}(x_0 + j \lambda_1^{-1}) .
\]
For $(t,\alpha) \in S_{\lambda_1}(x_0)$ we write 
\[
|M_{\lambda_2}(D) u (t,\alpha)| \lesssim |u| \ast \frac{\lambda_2}{1+ \lambda_2^2 \alpha^2}
\lesssim \sum_j (  1_{S_{\lambda_1}(x_0 + j \lambda_1^{-1})}|u|) \ast \frac{\lambda_2}{1+ \lambda_2^2 \alpha^2}.
\]
Now we consider two cases. If $|j| \leq 2$ then we simply use Young's inequality. This no longer suffices for all $j$ because of the need 
for summation in $j$. However, for such $j$ we can use the kernel
decay instead. If $(t,\alpha_1) \in  S_{\lambda_1}(x_0 + j \lambda_1^{-1}) $
then 
\[
|\alpha - \alpha_1| \approx j \lambda_1^{-1}.
\]
Therefore 
\[
\frac{\lambda_2}{1+ \lambda^2_2 (\alpha-\alpha_1)^2} 
\approx \lambda_2^{-1} j^{-2} \lambda_1^{2}.
\]
Using Young's inequality yields 
\[
\| (  1_{S_{\lambda_1}(x_0 + j \lambda_1^{-1})}|u|) \ast \frac{\lambda_2}{1+ \lambda_2^2 \alpha^2} \|_{L^2_t L^q(B_{\lambda_1}(x_0))} \lesssim
\lambda_2^{\frac1p-\frac1q} \frac{ \lambda_1^2}{\lambda_2^2 j^2}
\| 1_{S_{\lambda_1}(x_0 + j \lambda_1^{-1})}u  \|_{L^2_t L^p(B_{\lambda_1}(x_0+j \lambda_1^{-1})}. 
\]
Now the $j$ summation is straightforward.

b) It suffices to consider the case $q = \infty$, 
and then use H\"older's inequality. Here we seek to estimate 
\[
\| M_{\lambda_2}(D) u\|_{L^2_t L^q_{loc}(B_{\lambda_2}(x_0))}.
\]
We use the same covering as above, and for $(t,\alpha) \in S_{\lambda_1}(x_0)$ we write 
\[
|M_{\lambda_2}(D) u (t,\alpha)| \lesssim |u| \ast \frac{\lambda_2}{1+ \lambda_2^2 \alpha^2}
\lesssim \sum_j (  1_{S_{\lambda_1}(x_0 + j \lambda^{-1})}|u|) \ast \frac{\lambda_2}{1+ \lambda_2^2 \alpha^2}.
\]
This time H\"older's inequality yields
\[
\| (  1_{S_{\lambda_1}(x_0 + j \lambda_1^{-1})}|u|) \ast \frac{\lambda_2}{1+ \lambda_2^2 \alpha^2}\|_{L^2_t L^\infty(B_{\lambda_2}(x_0))}
\lesssim  \frac{\lambda_1^{\frac{1}p-1} \lambda_2}{1+ \lambda_2^2 \lambda_1^{-2} j^2}
\| 1_{S_{\lambda_1}(x_0 + j \lambda_1^{-1})}u  \|_{L^2_t L^p(B_{\lambda_1}(x_0+j \lambda_1^{-1})},
\]
and the result follows again after $j$ summation.
\end{proof}

%%%%%%%%%%%%%%%%%%%%%%%%%%%%%%%%%%%%%%%%%%%%%%%%%%%%%%%%%%%%%%%%%%%%%%%%%%
%%%%%%%%%%%%%%%%%%%%%%%%%%%%%%%%%%%%%%%%%%%%%%%%%%%%%%%%%%%%%%%%%%%%%%%%%%
%%%%%%%%%%%%%%%%%%%%%%%%%%%%%%%%%%%%%%%%%%%%%%%%%%%%%%%%%%%%%%%%%%%%%%%%%%
%%%%%%%%%%%%%%%%%%%%%%%%%%%%%%%%%%%%%%%%%%%%%%%%%%%%%%%%%%%%%%%%%%%%%%%%%5
\subsection{Switching strips}
%%%%%%%%%%%%%%%%%%%%%%%%%%%%%%%%%%%%%%%%%%%%%%%%%%%%%%%%%%%%%%%%%%%%%%%%%%
%%%%%%%%%%%%%%%%%%%%%%%%%%%%%%%%%%%%%%%%%%%%%%%%%%%%%%%%%%%%%%%%%%%%%%%%%%
%%%%%%%%%%%%%%%%%%%%%%%%%%%%%%%%%%%%%%%%%%%%%%%%%%%%%%%%%%%%%%%%%%%%%%%%%%
%%%%%%%%%%%%%%%%%%%%%%%%%%%%%%%%%%%%%%%%%%%%%%%%%%%%%%%%%%%%%%%%%%%%%%%%%5

At several points in our analysis we need to switch local energy 
type integrals from the Euclidean to the holomorphic setting.
Here we compute this transition systematically, establishing bounds
that will be repeatedly used in the sequel.

The set-up is as follows. We consider some smooth function $\Psi$
in the fluid domain, which can be viewed either in the Eulerian or the
holomorphic coordinates. For such a function, we seek to compare the following two integrals:
\[
\mI_{E} := \int_0^T \iint_{\Omega (t)} m'(x-x_0) \Psi(x,y)\,dy dx dt,
\]
respectively
\[
\mI_{H} := \int_0^T \iint_{\Omega (t)} m'(\alpha-\alpha_0) \Psi(\alpha,\beta)\,  d\beta d\alpha dt.
\]
To fix the notations, the Eulerian strip is centered at $x= x_0$, which on the top corresponds to  $\alpha = \alpha_0$. However, in depth the line $x = x_0$ corresponds to a curve $\alpha = \alpha_0(t,\beta)$.
We will need to account for this difference. Our result is as follows:

\begin{proposition}\label{p:switch-int}
We have
\begin{equation}
\begin{split}
|\mI_E - \mI_H| \lesssim & \ \int_0^T \iint_{A_1(x_0)}
%\tilde m'(\alpha-\alpha_0)
(|W_\alpha|+|\Re W(\alpha,\beta) - \Re W(\alpha_0,0)|)\,
|\Psi(\alpha,\beta)|\,
d\beta d\alpha  dt
\\ & \ +
\int_0^T \int_{\beta \in [-h,-1]} ( c_\beta +\sup_{|\alpha - \alpha_0|} 
|W_\alpha|)
\sup_{|\alpha - \alpha_0| < |\beta|} ( |\Psi| + |\beta| |\Psi_\alpha|)\,
 d\beta dt,
\end{split}
\end{equation}
%where $\tilde{m}'$ is a bump function with a slightly larger support than $m'$. 
\end{proposition}

\begin{proof}
We switch $\mI_E$ to holomorphic coordinates by changing variables.
This yields
\[
\mI_E = \int_0^T \iint_{\Omega (t)} J m'(x-x_0) \Psi (\alpha,\beta) \, d\beta d\alpha dt.
\]
Since $|J-1| \lesssim |W_\alpha|$, we can harmlessly replace $J$ by $1$,
and then we are left with the difference
\[
 \int_0^T \iint_{\Omega (t)} (m'(x-x_0) - m'(\alpha-\alpha_0)) \Psi(\alpha,\beta)\, d\beta d\alpha dt.
\]
Here we have 
\[
x-x_0 = \alpha-\alpha_0 + \Re W(\alpha,\beta) - \Re W(\alpha_0,0).
\]
The function $m'$ is supported in the unit interval, and $W$ has an $\epsilon$ small Lipschitz constant ($\epsilon$ is the control norm defined in the Introduction). Then, within the support of $m'(x-x_0)$ we must 
have 
\begin{equation} \label{near}
|\alpha -\alpha_0| \lesssim \epsilon |\beta| + 1.
\end{equation}
We now divide the analysis in two cases depending on the size of $\beta$.

\bigskip

\textbf{ a) Small depth, $-1 < \beta < 0$.}
Here we simply use the Lipschitz property of $m'$ to get
\[
|m'(x-x_0) - m'(\alpha-\alpha_0)| \lesssim |\Re W(\alpha,\beta) - \Re W(\alpha_0,0)|.
\]

\bigskip

\textbf{ b) Large depth, $-h < \beta < -1$.}
Here we continuously switch between the two bumps 
$m'(x-x_0)$ and $m'(\alpha - \alpha_0)$.
Denoting 
\[
d(k,\alpha) := \alpha-\alpha_0 + k(\Re W(\alpha,\beta) - \Re W(\alpha_0,0)),
\]
we consider the family of bump functions $m'(d(k,\alpha))$ with $k \in [0,1]$.
Within the support of these bump functions we still have
$|\alpha - \alpha_0| \leq |\beta|$, therefore, using also
Proposition~\ref{p:switch-strips}
\begin{equation}\label{dka}
| \Re W(\alpha,\beta) - \Re W(\alpha_0,0)| \lesssim |\beta|(
c_\beta + \sup_{|\alpha - \alpha_0| \leq |\beta|} |W_\alpha|).
\end{equation}

Using the functions $m'(d(k,\alpha))$ we have
\[
\begin{split}
m'(x-x_0) -m'(\alpha - \alpha_0) = & \ \int_0^1 
\frac{d}{dk} m'(d(k,\alpha))\, dk
\\
= & \ \int_0^1 
 m''(d(k,\alpha)) 
(\Re W(\alpha,\beta) - \Re W(\alpha_0,0))\, dk
\\
= & \ \int_0^1 \partial_\alpha  m'(d(k,\alpha)) 
\frac{\Re W(\alpha,\beta) - \Re W(\alpha_0,0)}{1+\Re W_\alpha}\, dk .
\end{split}
\]
Hence, integrating by parts we get
\[
\begin{split}
D(\beta,t):= & \ \int (m'(x-x_0) - m'(\alpha-\alpha_0)) \Psi(\alpha,\beta)\, d\alpha
\\
= & \ - \int_{0}^1 \int  m'(d(k,\alpha)) \partial_\alpha \left[
\frac{\Re W(\alpha,\beta) - \Re W(\alpha_0,0)}{1+\Re W_\alpha} \Psi(\alpha,\beta)\right]\, d\alpha  dk.
\end{split}
\]
Here $m'$ is a bump function with unit integral, so taking absolute values we get
\[
|D(\beta,t)| \lesssim \sup_{|\alpha - \alpha_0| < |\beta|} 
|W_\alpha| |\Psi|  + | \Re W(\alpha,\beta) - \Re W(\alpha_0,0)| ( |\Psi_\alpha| + |W_{\alpha\alpha}| |\Psi|).
\]
In this context we have 
\[
|W_{\alpha\alpha}| \lesssim \epsilon |\beta|^{-1},
\]
so the conclusion follows from \eqref{dka}.

\end{proof}

%%%%%%%%%%%%%%%%%%%%%%%%%%%%%%%%%%%%%%%%%%%%%%%%%%%%%%%%%%%%%%%%%%%%%%%%%%
%%%%%%%%%%%%%%%%%%%%%%%%%%%%%%%%%%%%%%%%%%%%%%%%%%%%%%%%%%%%%%%%%%%%%%%%%%
%%%%%%%%%%%%%%%%%%%%%%%%%%%%%%%%%%%%%%%%%%%%%%%%%%%%%%%%%%%%%%%%%%%%%%%%%%
%%%%%%%%%%%%%%%%%%%%%%%%%%%%%%%%%%%%%%%%%%%%%%%%%%%%%%%%%%%%%%%%%%%%%%%%%%
%%%%%%%%%%%%%%%%%%%%%%%%%%%%%%%%%%%%%%%%%%%%%%%%%%%%%%%%%%%%%%%%%%%%%%%%%%
%%%%%%%%%%%%%%%%%%%%%%%%%%%%%%%%%%%%%%%%%%%%%%%%%%%%%%%%%%%%%%%%%%%%%%%%%%
\subsection{Bounds for \texorpdfstring{$\eta = \Im W$}{}.}
%%%%%%%%%%%%%%%%%%%%%%%%%%%%%%%%%%%%%%%%%%%%%%%%%%%%%%%%%%%%%%%%%%%%%%%%%%
%%%%%%%%%%%%%%%%%%%%%%%%%%%%%%%%%%%%%%%%%%%%%%%%%%%%%%%%%%%%%%%%%%%%%%%%%%
%%%%%%%%%%%%%%%%%%%%%%%%%%%%%%%%%%%%%%%%%%%%%%%%%%%%%%%%%%%%%%%%%%%%%%%%%%
%%%%%%%%%%%%%%%%%%%%%%%%%%%%%%%%%%%%%%%%%%%%%%%%%%%%%%%%%%%%%%%%%%%%%%%%%5

Here we have the straightforward equivalence
\begin{equation}\label{W-eta-le}
\| \eta \|_{LE^0} \approx \| \Im W\|_{LE^0}
\end{equation}
as $\eta$ and $\Im W$ are one and the same function up to a biLipschitz 
change of coordinates. Our first aim will be to understand the bounds for the low frequencies of $\Im W$ on the top:

\begin{lemma}\label{l:W-LE-top}
For each dyadic frequency $\lambda < 1$ we have
\begin{equation}\label{W-LE-top}
\|\Im W_{\leq \lambda} \|_{L^2_t L^\infty_{loc}(B_\lambda)} \lesssim \| \Im W\|_{LE^0}. 
\end{equation}
\end{lemma}
\begin{proof}
Since $LE^0 = L^2_t L^2_{loc}(B_1)$, this bound is a direct application 
of Lemma~\ref{l:bern-loc} (b).
\end{proof}

On the other hand, for nonlinear estimates, we also need bounds in depth,
precisely over the regions $A_{\lambda}(x_0)$.  There we have

\begin{lemma}\label{l:theta-LE-loc}
For each dyadic frequency $\lambda < 1$ we have
%\begin{equation}
%\int \sup_{A_{\lambda}(\alpha_0) } \left\{ | \Im W|^2 + \lambda^{-2} |W_\alpha|^2\right\} %\, dt \lesssim \| \Im W\|_{LE^0}^2. 
%\end{equation}
\begin{equation}\label{theta-LE-loc}
\|\Im W \|_{L^2_t L^\infty_{loc}(A_\lambda)} +\lambda^{-1}\|W_{\alpha} \|_{L^2_t L^\infty_{loc}(A_\lambda)} \lesssim \| \Im W\|_{LE^0}. 
\end{equation}
\end{lemma}
\begin{proof}  

We start by recalling that $\Im W$ is harmonic in the strip with Dirichlet boundary condition on the bottom. Then  $\Im W (\alpha, \beta)$ is given by 
\[
\Im W(\beta) =  P_{D}( \beta , D) \Im W(0),
\]
where the symbol $p_D(\xi,\beta)$ of the multiplier $P_D(\beta)$ is 
\[
p_{D}(\xi, \beta)= \frac{\sinh \left( \xi (\beta+h)\right)}{\sinh (\xi h)}.
\]
For $|\beta| \approx \lambda^{-1}$ these symbols satisfy 
uniformly the condition \eqref{m-symbol} with $\lambda= \lambda_1$.
Then the kernel bound \eqref{m-kernel} also holds uniformly, and the 
conclusion of Lemma~\ref{l:bern-loc} applies also uniformly. This yields
the bound for the first term on the left. The bound for the second term on the left is similar, by applying the same argument to the operators
$\lambda^{-1} \partial_\alpha P_D(\beta)$ and $\lambda^{-1} \partial_\beta P_D(\beta)$ uniformly in $|\beta| \approx \lambda^{-1}$.

Alternatively, we note that one can obtain the bound for $W_\alpha$ 
or equivalently for $\nabla \Im W$ by elliptic regularity.
 We have already obtained estimates for $\Im W$ in the  region $A_{\lambda}(x_0)$, which has size $\lambda^{-1}$,  and so using the elliptic regularity we can estimate the derivatives of a harmonic function in a domain in terms of the solution on a larger domain: 
\[
\Vert \nabla_{\alpha, \beta} \Im W (\alpha ,\beta )\Vert_{L^{\infty}(A_{\lambda}(x_0))} \lesssim \lambda \Vert \theta (\alpha , \beta ) \Vert_{L^{\infty}(c A_{\lambda}(x_0))}, \quad c>1. 
\]

\end{proof}

Also connected to $\theta = \Im W$, we need to estimate the difference
$\theta - H_N(\eta)$. Here we are comparing two harmonic functions with 
same Dirichlet data on the top, but with homogeneous Dirichlet vs. Neumann boundary condition the bottom. The regions over which we compare the difference  are of size $h$:
\[
\mathbf{B}_{h}(x_0) := \{ (\alpha, \, \beta )\,  :  \,  \beta \in [-h, 0], \  |\alpha-\alpha_0| \lesssim h \}.
\] 
We have
\begin{lemma}\label{l:theta-D-N}
For the difference $\theta - H_N(\eta)$ have
\begin{equation}
\|h^j \nabla^j ( \theta - H_N(\eta))\|_{L^2_t L^\infty_{loc} (\mathbf{B}_h) }
% \int_0^T \sup_{(\alpha, \beta)\in \mathbf{B}_h (x_0)} \left\{ | \theta - H_N(\eta)|^2 + h^{-2} |\nabla(\theta - H_N(\eta))|^2 \right\} \, dt 
\lesssim \| \Im W\|_{LE^0}, \qquad j = 0,1,2.
\end{equation}
\end{lemma}
\begin{proof}
We first compute
\[
(\theta - \HN(\eta))(\beta) = C(\beta,D) \eta,
\]
where
\[
C(\beta,\xi) := \frac{\sinh \left( (h+\beta) \xi\right) }{\sinh \left(h \xi\right)} - \frac{\cosh\left( (h+\beta) \xi\right)}{\cosh (h \xi)} =
\frac{ 2\sinh (\beta \xi)}{\sinh \left( 2 h \xi\right)} 
\]
has size $1$ for $|\xi| < h^{-1}$ and decays exponentially for larger $\xi$. Thus these kernels satisfy uniformly the condition \eqref{m-symbol} with $\lambda_2 = 1/h$. Hence the bound for $\Im W - \HN \left(\Im W\right)$
follows by Lemma~\ref{l:bern-loc} (b) with $\lambda_1 = 1$ and $\lambda_2 = 1/h$.

We now turn our attention to the $j= 1$ case, namely the map 
\[
\eta \to \nabla (\theta -  \HN(\eta))(\alpha, \beta).
\]
Differentiating the previous symbol in either $\alpha$ or $\beta$ yields another factor of $\xi$, namely leads to the symbols
\[
\frac{\xi \cdot \sinh(\xi \beta)}{\sinh (2h\xi)}, \qquad \frac{\xi \cdot \cosh(\xi \beta)}{\sinh (2h\xi)}.
\]
Both are bump functions on the $h^{-1}$ scale, 
but now their size  is improved to $h^{-1}$. 
Thus both operators equal $h^{-1}$ times an averaging operator on the $h$ scale. Hence Lemma~\ref{l:bern-loc}(b) again 
applies, but yields another $h^{-1}$ factor. The same argument applies as well for the second order derivatives of $\theta-\HN(\eta)$.

\end{proof}

Now we are already able to estimate the easiest of the error terms:

\begin{proof}[\bf Proof of the $Err_4$ bound]
We estimate the difference between the two integrals $I_1$ and $I_1^{hol}$ using Proposition~\ref{p:switch-int} with $\Psi = 
\theta(\theta - H_N(\eta))_\beta$. We also need to account for the 
difference 
\[
\theta(\theta - H_N(\eta))_y - \theta(\theta - H_N(\eta))_\beta ,
\]
which, by chain rule, is readily estimated by 
\[
|\theta(\theta - H_N(\eta))_y - \theta(\theta - H_N(\eta))_\beta |
\lesssim |\theta||\nabla_{\alpha, \beta} (\theta - H_N(\eta))| |W_\alpha|.
\]
Combining this with Proposition~\ref{p:switch-int} and using $|W_\alpha| < \epsilon$, we obtain
\[
\begin{split}
|I_1 - I_1^{hol}| \lesssim & \ \epsilon \int_0^T \iint_{A_1(\alpha_0)} |\theta||\nabla (\theta - H_N(\eta))|\,   d\beta d\alpha dt
\\ & \ 
+ 
\epsilon \int_0^T \int_{-h}^{-1} \sup_{|\alpha - \alpha_0| < |\beta|} 
|\theta||\nabla (\theta - H_N(\eta))| + |\beta| |\partial_\alpha( \theta
(\theta-\HN(\eta))_\beta)|\,  d\beta dt.
\end{split}
\]
It remains to bound the two integrals by $\| \eta \|_{LE}^2$, both of which are straightforward in view of Lemma~\ref{l:theta-LE-loc} and 
Lemma~\ref{l:theta-D-N}.
\end{proof}

%%%%%%%%%%%%%%%%%%%%%%%%%%%%%%%%%%%%%%%%%%%%%%%%%%%%%%%%%%%%%%%%%%%%%%%%%%
%%%%%%%%%%%%%%%%%%%%%%%%%%%%%%%%%%%%%%%%%%%%%%%%%%%%%%%%%%%%%%%%%%%%%%%%%%
%%%%%%%%%%%%%%%%%%%%%%%%%%%%%%%%%%%%%%%%%%%%%%%%%%%%%%%%%%%%%%%%%%%%%%%%%%
%%%%%%%%%%%%%%%%%%%%%%%%%%%%%%%%%%%%%%%%%%%%%%%%%%%%%%%%%%%%%%%%%%%%%%%%%5
\subsection{ Estimates for \texorpdfstring{$Y$}{}}
%%%%%%%%%%%%%%%%%%%%%%%%%%%%%%%%%%%%%%%%%%%%%%%%%%%%%%%%%%%%%%%%%%%%%%%%%%
%%%%%%%%%%%%%%%%%%%%%%%%%%%%%%%%%%%%%%%%%%%%%%%%%%%%%%%%%%%%%%%%%%%%%%%%%%
%%%%%%%%%%%%%%%%%%%%%%%%%%%%%%%%%%%%%%%%%%%%%%%%%%%%%%%%%%%%%%%%%%%%%%%%%%
%%%%%%%%%%%%%%%%%%%%%%%%%%%%%%%%%%%%%%%%%%%%%%%%%%%%%%%%%%%%%%%%%%%%%%%%%5

Here we prove a local energy bound for the auxiliary holomorphic function
\[
Y = \frac{W_\alpha}{1+W_\alpha}.
\]

\begin{lemma}\label{l:Y}
a) For $\lambda > 1$ we have 
\begin{equation}\label{Yhi}
\| Y_\lambda \|_{L^2_t L^2_{loc}} \lesssim \lambda M.
\end{equation}

b) For $\lambda \leq 1$ we have 
\begin{equation}\label{Ylo}
\| Y_\lambda \|_{L^2_t L_{\alpha}^\infty(B_\lambda(x_0))} \lesssim \lambda M.
\end{equation}
\end{lemma}
We note that both estimates follow directly from 
Lemma~\ref{l:theta-LE-loc} if $Y$ is replaced by $W_\alpha$. However
to switch to $Y$ one would seem to need some Moser type inequalities,
which unfortunately do not work in negative Sobolev spaces.
The key observation is that in both of these estimates it is critical that $W_\alpha$ is holomorphic, and $Y$ is an analytic function 
of $W_\alpha$.

\begin{proof}
We will bound $Y$ on the top using bounds for its holomorphic extension. Based on the bounds for $W$
in \eqref{W-eta-le} and \eqref{theta-LE-loc}, this satisfies estimates as follows:
\medskip

A. If $-1 < \beta < 0$ then on each unit strip $S$ we have
\begin{equation}\label{Y-high}
\| \partial_\alpha^j Y(\cdot, \beta)\|_{L^2_t L^2_{\alpha}(B_1(x_0)} \lesssim |\beta|^{-1-j} M .
\end{equation}
 
\medskip

B. If $-h < \beta < -1$ then on each $|\beta|$ strip  we have
\begin{equation}\label{Y-low}
\| \partial_\alpha^j Y(\cdot, -\beta)\|_{L^2_t L^\infty_\alpha(B_\beta(x_0))} \lesssim |\beta|^{-1-j} M .
\end{equation}

We use the following representation of $Y$ on the top,
\[
\begin{split}
Y(\alpha,0) = & \ Y(\alpha,-h) + i \int_{-h}^0 Y_\alpha(\alpha,\beta) \, d\beta 
\\
= & \ Y(\alpha,-h) + i h Y_{\alpha}(\alpha,-h) +   \int_{-h}^0 \beta Y_{\alpha\alpha}(\alpha,\beta)\, d\beta .
\end{split}
\]
The function $Y(\alpha,-h)$ is at frequency $1/h$, and obeys the bounds \eqref{Y-low}
therefore the first two terms above  easily satisfy the bounds in the lemma.

It remains to consider the integral term, where we treat the integrand differently depending on $\beta$
and on $\lambda$.

 \medskip

\textbf{ Case I: $\lambda > 1$.}  Here we are only interested in unit strips, and use $L^2$ bounds.
Depending on $\beta$, we differentiate as follows:
 
\medskip

\textbf{ Case I.a:  Small $\beta$, $- \lambda^{-1} < \beta < 0$.} 
There we use \eqref{Y-high} to estimate 
\[
\|P_\lambda \partial_\alpha^2 Y(\cdot,\beta)\|_{L^2_t L^2_{loc}} \lesssim 
\lambda^2 \| Y(\cdot,\beta)\|_{L^2_t L^2_{loc}} \lesssim |\beta|^{-1} \lambda^2 M ,
\]
which suffices for the $\beta$ integration.

\medskip

\textbf{ Case I.b:  Large $\beta$, $- h < \beta <  \lambda^{-1}$.} 
There depending on the size of $\beta$ we use either \eqref{Y-high} or 
\eqref{Y-low} to estimate 
\[
\|P_\lambda \partial_\alpha^2 Y(\cdot,\beta)\|_{L^2_t L^2_{loc}} \lesssim  |\beta|^{-3}  M,
\]
which again suffices for the $\beta$ integration.

\bigskip

\textbf{ Case II: $\lambda > 1$.}  Here we are only interested in  strips of width
$\lambda^{-1}$, and use $L^\infty$ bounds. Depending on $\beta$, we differentiate as follows:

\medskip

\textbf{ Case II.a:  Very small $\beta$, $- 1 < \beta < 0$.} 
Here we cover the $\lambda^{-1}$ strip with unit strips, use H\"older's inequality, then Bernstein's inequality to get
\[
\|P_\lambda \partial_\alpha^2 Y(\cdot,\beta)\|_{L^2_t L^\infty_{loc}(B_\lambda)} \lesssim 
\lambda^\frac52  \| Y(\cdot,\beta)\|_{L^2_t L^2_{loc}(B_\lambda)} \lesssim
\lambda^2  \| Y(\cdot,\beta)\|_{L^2_t L^2_{loc}} \lesssim
 |\beta|^{-1} \lambda^2 M,
\]
which is enough.

\medskip

\textbf{ Case II.b:  Small $\beta$, $- \lambda^{-1}  < \beta < -1$.} 
Here we cover the $\lambda^{-1}$ strip with $|\beta|^{-1}$  strips, use H\"older inequality, then Bernstein's inequality to get
\[
\begin{split}
\|P_\lambda \partial_\alpha^2 Y(\cdot,\beta)\|_{L^2_t L^\infty_{loc}(B_\lambda)} \lesssim & \ 
\lambda^\frac52  \| Y(\cdot,\beta)\|_{L^2_t L^2_{loc}(B_\lambda)} \lesssim
\lambda^2 \beta^{-\frac12} \| Y(\cdot,\beta)\|_{L^2_t L^2_{loc}(B_\beta)} 
\\ \lesssim & \ 
\lambda^2  \| Y(\cdot,\beta)\|_{L^2_t L^\infty_{loc}(B_\beta)} \lesssim
 |\beta|^{-1} \lambda^2 M,
\end{split}
\]
which is enough.

\medskip 

\textbf{ Case II.c:  Large $\beta$, $- h < \beta <  \lambda^{-1}$.} 
There we use \eqref{Y-low} to estimate 
\[
\|P_\lambda \partial_\alpha^2 Y(\cdot,\beta)\|_{L^2_t L^2_{loc}(B_\lambda)} \lesssim  |\beta|^{-3}  M,
\]
which again suffices for the $\beta$ integration.

\end{proof}

%%%%%%%%%%%%%%%%%%%%%%%%%%%%%%%%%%%%%%%%%%%%%%%%%%%%%%%%%%%%%%%%%%%%%%%%%%
%%%%%%%%%%%%%%%%%%%%%%%%%%%%%%%%%%%%%%%%%%%%%%%%%%%%%%%%%%%%%%%%%%%%%%%%%%
%%%%%%%%%%%%%%%%%%%%%%%%%%%%%%%%%%%%%%%%%%%%%%%%%%%%%%%%%%%%%%%%%%%%%%%%%%
%%%%%%%%%%%%%%%%%%%%%%%%%%%%%%%%%%%%%%%%%%%%%%%%%%%%%%%%%%%%%%%%%%%%%%%%%5
\subsection{Bounds for \texorpdfstring{$\partial \phi = R$}{}}
%%%%%%%%%%%%%%%%%%%%%%%%%%%%%%%%%%%%%%%%%%%%%%%%%%%%%%%%%%%%%%%%%%%%%%%%%%
%%%%%%%%%%%%%%%%%%%%%%%%%%%%%%%%%%%%%%%%%%%%%%%%%%%%%%%%%%%%%%%%%%%%%%%%%%
%%%%%%%%%%%%%%%%%%%%%%%%%%%%%%%%%%%%%%%%%%%%%%%%%%%%%%%%%%%%%%%%%%%%%%%%%%
%%%%%%%%%%%%%%%%%%%%%%%%%%%%%%%%%%%%%%%%%%%%%%%%%%%%%%%%%%%%%%%%%%%%%%%%%5

This is not as easy as for $\eta = \Im W$, because the strips in the 
Eulerian and holomorphic setting do not agree, and can in effect be 
quite different.  Nevertheless, we will still prove

\begin{proposition} \label{p:R-LE}
Assume \eqref{uniform} holds. Then we have 
\begin{equation}\label{R-LE-equiv}
\| \psi\|_{LE^{-\frac12}} \approx \|R\|_{LE^{-\frac12}} \qquad (\text
{mod } \epsilon \|\Im W\|_{LE^0}).
\end{equation}
\end{proposition}
Here the equivalence should be interpreted as the double inequality
\[
\| \psi\|_{LE^{-\frac12}} \lesssim \|R\|_{LE^{-\frac12}} 
+\epsilon \|\Im W\|_{LE^0}, 
\qquad 
\|R\|_{LE^{-\frac12}}\lesssim 
\| \psi\|_{LE^{-\frac12}}  +\epsilon \|\Im W\|_{LE^0}.
\]

\begin{proof}
We recall that $|\nabla \phi|^2 = |R|^2$, so all we need is to transfer the $L^2$ local bound from unit strips in the  Eulerian setting to unit strips in the holomorphic setting. 

 To switch from one strip to another we will critically use the bound in
 Proposition~\ref{p:switch-int}, which uses the fact that in depth
 the distance between the two strips is smaller than $\epsilon|\beta|$.
  Because of this, we start with a preliminary result, which is more easily proved:
 
 \begin{lemma}\label{l:R-LE}
 For each dyadic $\lambda < 1$ we have
 \begin{equation}\label{R-LE}
 \Vert R \Vert _{L^2_t L^{\infty}_{loc}(A_{\lambda})} +\lambda^{-2}\Vert \nabla R \Vert _{L^2_t L^{\infty}_{loc}(A_{\lambda})} \lesssim \lambda
 \| R\|_{LE^{-\frac12}}^2 .
 \end{equation}
 \end{lemma}
 
 \begin{proof}
From the definition of the  local energy functional associated to $R$ we know that we have $L^2$ control over $R$ inside every vertical strip of width $1$. However, initially we do not have any information on the top or on the bottom of the strip. As a consequence we first prove the desired bound in a region $A_\lambda(x_0)$  that avoids the case 
 $\vert \beta\vert  \approx  h$. In order to use the control we have on $R$ we split the region $2A_\lambda(x_0)$ in strips of width $1$ and then add the $\lambda ^{-1}$ bounds on strips to obtain
 \[
 \Vert R\Vert^2_{L^2 (2A_\lambda(x_0))}\leq \lambda^{-1}\Vert R\Vert_{LE^{-\frac{1}{2}}}.
 \]
 Then the bound in the lemma follows by elliptic regularity. 
 
 Finally, if $\vert \beta\vert  \approx  h $ we use the homogeneous boundary conditions Dirichlet or Neumann for  $\Re R$ and $\Im R$  to  separately  mirror them in a symmetric domain below the bottom via reflection principle, and then proceed as above.

 \end{proof}
  We now return to the proof of the Proposition~\ref{p:R-LE}. For  this we need to compare the integrals 
 \[
 \mI_E = \int_0^T \iint m_x(x-x_0) |\nabla \phi|^2\, dx dy dt,
 \qquad 
 \mI_H = \int_0^T \iint m_x(\alpha-\alpha_0) |R|^2\, dx dy dt,
 \]
 and show that 
 \[
 |\mI_E - \mI_H| \lesssim \epsilon(\| R \|_{LE^{-\frac12}}^2 +  \|W\|_{LE^0}^2).
 \]
 
 Since $|\nabla \phi|^2 = R^2$, we can apply directly Proposition~\ref{p:switch-int}. This yields
\[
\begin{split}
|\mI_E - \mI_H| \lesssim & \ \epsilon \int_0^T \iint_{A_1(x_0)} |R|^2 \, d\beta d\alpha dt 
\\ & \ + 
\int_0^T \int_{-h}^1 ( c_\beta + \sup_{|\alpha - \alpha_0| < \beta}
|W_\alpha|) \sup_{|\alpha - \alpha_0| < \beta}
|R|^2 + | R R_\alpha|\, d\beta dt.
\end{split}
\]
 The first integral is directly estimated by $\|R\|_{LE^{-\frac12}}^2$.
 For the contribution of $c_\beta$ we use the dyadic summability 
 of $c_\beta$ along with Lemma~\ref{l:R-LE}. Hence we are left with
 \[
 \int_0^T \int_{-h}^1 \sup_{|\alpha - \alpha_0| < \beta}
|W_\alpha| \sup_{|\alpha - \alpha_0| < \beta}
|R|^2 + | R R_\alpha| \, d\beta dt .
 \]

To bound this last integral we switch roles
and use the local energy norm for $W_\alpha$ via Lemma~\ref{l:W-LE-top}
and for $R$ via Lemma~\ref{l:R-LE}, while for $R_\alpha$ we use 
the control norm and Bernstein's inequality to get the bound $|R_\alpha| \lesssim |\beta|^{\frac12} c_\beta$. This yields the fixed $\beta$ bound
\[
c_\beta |\beta|^{-1} \| R\|_{LE^{-\frac12}} \|\Im W\|_{LE^0}.
\]
Finally we integrate with respect to $\beta \in [-h,-1]$ to obtain 
\[
\epsilon \| R\|_{LE^{-\frac12}} \|\Im W\|_{LE^0}.
\]
\end{proof}

The local energy norm for $R$ measures the function inside the entire strip. 
However, we also need to estimate it on the top:

\begin{lemma}\label{l:R-LE-top}
For $R \in LE^{-\frac12}$ we have the following high frequency bound on the top: 
\begin{equation}\label{R-LE-top-high}
\| R\|_{L^2_t H^{-\frac12}_{loc}} \lesssim \| R \|_{ LE^{-\frac12} },
\end{equation}
respectively the low frequency bound
\begin{equation}\label{R-LE-top-low}
\|  R_\lambda \|_{L^2_t L^\infty(B_\lambda)} \lesssim 
\lambda^\frac12 \| R\|_{LE^{-\frac12}},
\qquad \lambda \leq 1.
\end{equation}
\end{lemma}
\begin{proof}
The first part follows from the trace theorem, as $R \in L^2_t L^2_{loc}(A_1)$
is a harmonic function. The second part is more delicate, but we can use 
the same argument as in Lemma~\ref{l:Y}. Precisely, we write
\[
R(\alpha,0) = R(\alpha,-h) + \int_{-h}^0 i R_\alpha (\alpha,\beta)\, d\beta.
\]
For the first term we can use Lemma~\ref{l:R-LE}. For the second term we 
split the integral into 
\[
R_1 = \int_{-h}^{-\lambda^{-1}} i R_\alpha (\alpha,\beta)\, d\beta, \qquad
R_2 = \int_{-\lambda^{-1}}^0 i R_\alpha (\alpha,\beta)\, d\beta .
\]

For $R_1$ we use the gradient bound in Lemma~\ref{l:R-LE},
to compute
\[
\|R_1\|_{L^2_t L^\infty_{loc}(B_\lambda)} \lesssim \int_{-h}^{-\lambda^{-1}} 
|\beta|^{-\frac32} \, d\beta \lesssim \lambda^{\frac12} M ,
\]
and the spectral projector $P_\lambda$ is harmless.

For $R_2$ on the other hand we use the spectral projector for Bernstein's inequality in Lemma~\ref{l:bern-loc}, and then to eliminate the derivative
\[
\begin{split}
\| P_\lambda R_2 \|_{L^2_t L^\infty_{loc}(B_\lambda)} \lesssim & \ 
\lambda^{\frac{1}{2}} \| P_\lambda R_2 \|_{L^2_t L^2_{loc}(B_\lambda)} 
\\
\lesssim & \
\lambda^{\frac32} \left\| \int_{-\lambda^{-1}}^0  R(\alpha,\beta)\,  d\beta \right\|_{L^2_t L^2_{\alpha}(B_\lambda)}
\\
\lesssim & \
\lambda \left\|  R  \right\|_{L^2_t L^2_{\alpha}(\mathbf{B}_\lambda)},
\end{split}
\]
where at the last step  we have used H\"older's inequality in $\beta$.
To estimate $R$ over a square $\mathbf B_\lambda(x_0)$ of width $\lambda^{-1}$,
we cover the square with $\lambda^{-1}$ strips $\HVS(x_0+j)$ 
with $|j| \lesssim \lambda^{-1}$, and then use Holder's inequality  again to get
\[
\| P_\lambda R_2 \|_{L^2_t L^\infty_{loc}(B_\lambda)} \lesssim 
\lambda^\frac12 \| R \|_{LE^{-\frac12}}.
\]
\end{proof}

%%%%%%%%%%%%%%%%%%%%%%%%%%%%%%%%%%%%%%%%%%%%%%%%%%%%%%%%%%%%%%%%%%%%%%%%%%
%%%%%%%%%%%%%%%%%%%%%%%%%%%%%%%%%%%%%%%%%%%%%%%%%%%%%%%%%%%%%%%%%%%%%%%%%%
%%%%%%%%%%%%%%%%%%%%%%%%%%%%%%%%%%%%%%%%%%%%%%%%%%%%%%%%%%%%%%%%%%%%%%%%%%
%%%%%%%%%%%%%%%%%%%%%%%%%%%%%%%%%%%%%%%%%%%%%%%%%%%%%%%%%%%%%%%%%%%%%%%%%5
\subsection{Bilinear estimates for \texorpdfstring{$|\nabla \phi|^2 = |R|^2$}{}
 and its harmonic extension}
%%%%%%%%%%%%%%%%%%%%%%%%%%%%%%%%%%%%%%%%%%%%%%%%%%%%%%%%%%%%%%%%%%%%%%%%%%
%%%%%%%%%%%%%%%%%%%%%%%%%%%%%%%%%%%%%%%%%%%%%%%%%%%%%%%%%%%%%%%%%%%%%%%%%%
%%%%%%%%%%%%%%%%%%%%%%%%%%%%%%%%%%%%%%%%%%%%%%%%%%%%%%%%%%%%%%%%%%%%%%%%%%
%%%%%%%%%%%%%%%%%%%%%%%%%%%%%%%%%%%%%%%%%%%%%%%%%%%%%%%%%%%%%%%%%%%%%%%%%5

Here we will prove the following bound:

\begin{lemma}\label{l:R2}
a) The function $|\nabla \phi|^2 = |R|^2$ restricted to the top satisfies the following 
estimate:
\begin{equation}\label{R2-LE}
\| |R|^2 \|_{LE^0} \lesssim \epsilon M .
\end{equation}

b) Its low frequency part satisfies
\begin{equation}\label{R2-LE-low}
\|P_\lambda |R|^2\|_{{L^2_t L^\infty_{loc}(B_\lambda (x_0))}} \lesssim c_\lambda M .
\end{equation}

c) In addition, for each $\lambda < 1$ there is a decomposition
\[
 P_{< \lambda} |R|^2 = G_{\lambda}^1 + G_{\lambda}^2,
\]
where
\begin{equation}\label{R2-G1}
\sup_{\lambda} \| G_{\lambda}^1 \|_{L^1 L^\infty(B_\lambda(x_0))} \lesssim \lambda  M^2,
\end{equation}
while
\begin{equation}\label{R2-G2}
 \| G_{\lambda}^2 \|_{L^2 L^\infty(B_\lambda(x_0))} \lesssim c_\lambda M.
\end{equation}
\end{lemma}

\begin{proof}
a) We restate this as a bound for $R$,
\begin{equation}
\label{inqR}
\| |R|^2 \|_{LE^0} \lesssim \| R \|_{LE^{-\frac12}}  \| R \|_{\ell^1 L^\infty _t H^1_h},
\end{equation}
where the $\ell^1$ summability is measured using the control frequency envelope $\{c_\lambda\}$. For this we use a Littlewood-Paley decomposition
\[
R\bar{R}=\sum_{\lambda} (R_{<\lambda}\bar{R}_{\lambda}
R_{\lambda}\bar{R}_{<\lambda} )+ \sum_{\lambda} R_{\lambda}\bar{R}_{\lambda},
\]
and  analyze each component separately. We discuss two cases: first when $\lambda \geq 1$ and the second is when $\lambda <1$. For now we discuss the first case, i.e., $\lambda \geq 1$. To bound $R$  we will use either the control norm $X$, or the local energy norm $LE^{-\frac{1}{2}}$.
Correspondingly, we have the following bounds for the dyadic pieces
\[
\Vert R_{< \lambda}\Vert_{L^2_t L^2_{loc}} \lesssim \lambda^{\frac{1}{2}}\Vert R_{< \lambda}\Vert_{L^2_t H^{-\frac{1}{2}}_{h,loc}} \lesssim
\lambda^{\frac{1}{2}} M,
\]
respectively
\[
\quad \Vert R_{\lambda}\Vert_{L^\infty L^2} \leq \lambda^{-1}\Vert R_{\lambda}\Vert_{L^\infty H_h^1} \lesssim \lambda^{-1} 
c_\lambda.
\]
We begin with the low-high frequency term where we compute
using Bernstein's inequality in Lemma~\ref{l:bern-loc-hi}
\[
\Vert R_{\lambda}\bar{R}_{< \lambda}\Vert_{L^2_t L^2_{loc}} \lesssim 
\lambda^\frac12 \Vert R_{\lambda}\bar{R}_{<\lambda}\Vert_{L^2_t L^1_{loc}} \lesssim \lambda^\frac12 \Vert R_{\lambda} \Vert_{L^\infty_t L^2}
\Vert \bar{R}_{<\lambda}\Vert_{L^2_t L^2_{loc}} \lesssim c_\lambda M.
\]
Here we can sum up with respect to dyadic $\lambda$ as needed.

For $\sum_{\lambda} R_{\lambda}\bar{R}_{\lambda}$ we perform a similar analysis, and consider the product's output at frequency $\nu$, where $\nu \lesssim \lambda $. Here, $\nu$ can be $\geq1 $ or $<1$. We assume first that $\nu \geq 1$, and return to the other case later in the proof.  From  Bernstein's inequality in Lemma~\ref{l:bern-loc}
\[
\Vert  P_{\nu}\left( R_{\lambda}\bar{R}_{\lambda}\right)\Vert_{L^2_t L^2_{loc}} \lesssim  \nu^{\frac{1}{2}}\Vert P_{\nu}\left( R_{\lambda}\bar{R}_{\lambda}\right)\Vert_{L^2_t L^1_{loc}},
\]
and further, by Cauchy's inequality,  we get 
\[
\begin{aligned}
\Vert  P_{\nu}\left( R_{\lambda}\bar{R}_{\lambda}\right) \Vert_{L^2_t L^2_{loc}} \lesssim  \nu^{\frac{1}{2}}\Vert  R_{\lambda}\Vert_{L^2_t L^2_{loc}}\Vert \bar{R}_{\lambda}\Vert_{L^2_t L^2}\lesssim \left( \frac{\nu}{\lambda}\right)^{\frac{1}{2}} \| R_{\lambda} \|_{L^2_t H^{-\frac{1}{2}}_{h,loc}}  \|\bar{ R}_{\lambda} \|_{L^\infty_t H^1_h}
\lesssim \left( \frac{\nu}{\lambda}\right)^{\frac{1}{2}} c_\lambda M.
\end{aligned}
\]
The $\lambda$ summation is again straightforward.

Therefore  \eqref{inqR} holds for the high frequency $( \geq 1)$
part of the output. The remaining case in \eqref{inqR} corresponds to low frequency output and will follow from the proof of part (b) below.

b) The goal here is to prove the following estimate 
\begin{equation}
\label{bound<1}
\sum_{\lambda < 1} \|P_\lambda |R|^2\|_{{L^2_t L^\infty_{loc}(B_\lambda)}} \lesssim \| R \|_{LE^{-\frac12}}  \| R \|_{\ell^1 L^\infty _tH^1_h},
\end{equation}
which in particular suffices to finish the proof of part $a)$ of the proposition. Again we use the control frequency envelope $\{c_\lambda\}$
to measure the $\ell^1$ summation in the second factor on the right, and will show that
\begin{equation}
\label{bound<1a}
 \|P_\lambda |R|^2\|_{{L^2_t L^\infty_{loc}(B_\lambda)}} \lesssim c_\lambda M.
 \end{equation}

We need to consider the expressions $P_\lambda \left( R_{\nu} \bar{R}_{\mu}\right)$, where by Littlewood-Paley  trichotomy, we have several cases to discuss:
\begin{flushleft}
\textbf{i.) Case $\nu\approx \mu $,  $\mu >\lambda$ and $\mu >1$.} 
\end{flushleft}
In this case both input frequencies are comparable and larger than $1$ but the output frequency is $\lambda <1$. We use Bernstein's inequality and H\"older's inequality in both space and time to obtain
\[
\Vert P_{\lambda} \left(R_{\mu} \bar{R}_{\mu}\right)\Vert_{L^2_t L^{\infty}_{loc}(B_\lambda)}\lesssim \lambda  \Vert R_{\mu} \bar{R}_{\mu} \Vert_{L^2_t L^1_{loc}(B_\lambda)}
\lesssim \lambda  \Vert R_{\mu} \Vert_{L^{2}_t L^2_{loc}(B_\lambda) }\Vert \bar{R}_{\mu} \Vert_{L^\infty_t L^2}.
\]
Since the input frequencies are  higher than $1$, we estimate the first factor using Lemma~\eqref{l:R2} adapted for the dyadic pieces, together with the fact that in an interval of size $\lambda^{-1}$ we have about $\lambda^{-1}$ size $1$ subintervals. For the second factor we use the control envelope $c_\lambda$. This yields 
\[
\begin{aligned}
\lambda  \Vert R_{\mu} \Vert_{L^{2}_t L^2_{loc}(B_\lambda) }\Vert \bar{R}_{\mu} \Vert_{L^\infty_t L^2(B_\lambda)}
&\lesssim \lambda \mu^{\frac12} \lambda^{-\frac{1}{2}}\Vert R_{\mu}\Vert_{L^2_t H^{-\frac12}_{loc}} \Vert R_{\mu}\Vert_{L^{\infty}_t L^2_{\alpha}}\\
&\lesssim \lambda^{\frac12} \mu^{-\frac12}\Vert R\Vert_{LE^{-\frac12}}\Vert R_{\mu}\Vert_{L^{\infty}_t H^1_{\alpha}} 
\\
& \lesssim \lambda^{\frac12} \mu^{-\frac12} c_\mu M.
\end{aligned}
\]
Now the $\mu$ summation is straightforward due to the off-diagonal decay.

\medskip

\begin{flushleft}
\textbf{ii.) Case $\nu\approx \mu $,  $\mu >\lambda$ and $\mu <1$.} 
\end{flushleft}
This case is a harder one because we deal with different scale localizations. More explicitly the input frequencies are on the scale $\mu^{-1}$ which is less  than the output frequency which lives on the scale $\lambda^{-1}$. Thus, we first use Bernstein's inequality in Lemma~\ref{l:bern-loc}, followed by  H\"older's inequality in both space and time:
\begin{equation}
\label{eq1}
\Vert P_{\lambda} \left(R_{\mu} \bar{R}_{\mu}\right)\Vert_{L^2_t L^{\infty}_{loc}(B_\lambda)}\lesssim \lambda  \Vert R_{\mu} \bar{R}_{\mu} \Vert_{L^2_t L^1_{loc}(B_\lambda)}
\lesssim \lambda  \Vert R_{\mu} \Vert_{L^{\infty}_t L^2_{loc}(B_\lambda) }\Vert \bar{R}_{\mu} \Vert_{L^2_t L^2_{loc}(B_\lambda)},
\end{equation}
and then we use the control envelope $c_\mu$ to arrive to
\begin{equation}
\label{eq2}
\begin{split}
 \lambda  \Vert R_{\mu} \Vert_{L^{\infty}_t L^2_{loc}(B_\lambda) }\Vert \bar{R}_{\mu} \Vert_{L^2_t L^2_{loc}(B_\lambda)}\lesssim & \   \lambda \mu^{-1}  \Vert R_{\mu} \Vert_{L^{\infty}_t H^1_{h, loc}(B_\lambda) }\Vert \bar{R}_{\mu} \Vert_{L^2_t L^2_{loc}(B_\lambda)} 
 \\ \lesssim & \  \lambda \mu^{-1}  c_\mu\Vert \bar{R}_{\mu} \Vert_{L^2_t L^2_{loc}(B_\lambda)} .
  \end{split}
 \end{equation}

In the second term on the right we switch from $\lambda^{-1}$ width strips to $\mu^{-1}$ wide strips using Holder's inequality, followed by H\"older's inequality
again and then Lemma~\eqref{l:R-LE} to obtain
\begin{equation}
\label{eq3}
\Vert \bar{R}_{\mu} \Vert^2_{L^2_t L^2_{loc}(B_\lambda)} 
\lesssim \left(\frac{\mu}{\lambda}\right)^\frac12\Vert \bar{R}_{\mu} \Vert^2_{L^2_t L^2_{loc}(B_\mu)} \lesssim \frac{\mu}{\lambda^\frac12}\Vert \bar{R}_{\mu} \Vert^2_{L^2_t L^\infty_{loc}(B_\mu)}   
\lesssim \left(\frac{\mu}{\lambda}\right)^\frac12 M.
\end{equation}
Using this in \eqref{eq2} we have proved that
\begin{equation}\label{eq4}
\Vert P_{\lambda} \left(R_{\mu} \bar{R}_{\mu}\right)\Vert_{L^2_t L^{\infty}_{loc}(B_\lambda)} \lesssim \lambda^{\frac12} \mu^{-\frac12} c_\mu M.
\end{equation}
The $\mu$ summation is again straightforward.

\medskip

\begin{flushleft}
\textbf{iii.) Case $\nu <\lambda  $ and $\mu \approx \lambda$.}\end{flushleft}  Here  we observe that we can drop the projection $P_{\lambda}$, and then we can use Lemma~\ref{l:R-LE} for the first factor and Bernstein's inequality for the second one
\begin{equation}\label{eq5}
\begin{aligned}
\Vert R_{\nu}\bar{R}_{\mu}\Vert_{L^2_tL^{\infty}_{loc}(B_{\lambda})}&\lesssim \Vert R_{\nu}\Vert_{L^2_tL^{\infty}_{loc}(B_{\lambda})} \Vert R_{\mu}\Vert_{L^{\infty}_t L^{2}_{\alpha}}\lesssim \nu^{\frac{1}{2}} \mu^{-\frac{1}{2}} c_\mu M.
\end{aligned}
\end{equation}
We do have  off-diagonal decay  since $\nu <\mu $, and summing over such $\nu$ yields a bound of $c_\lambda M$ as desired. 

c) We observe that we only need to place low-low interactions
in $G^1$ and high-high interactions in $G^2$. In this context by low-low we mean that both  input frequencies are smaller than $\lambda$, and then their output is also smaller than $\lambda$, and by high-high interaction we refer to larger than $\lambda$ input frequencies that give rise to a smaller than $\lambda$ output frequency. 

 We begin with the  input frequencies $\mu$ and $\nu$ both smaller than $\lambda$, and by H\"older's inequality in time we get that
\[
\Vert R_{\mu}R_{\nu}\Vert_{L^1_tL^{\infty}_{loc}(B_{\lambda})} \lesssim \Vert R_{\mu}\Vert_{L^2_tL^{\infty}_{loc}(B_{\lambda})}\Vert R_{\nu}\Vert_{L^2_tL^{\infty}_{loc}(B_{\lambda})}.
\]
Since both $\mu$ and $\nu$ are smaller than $\lambda$ we can apply Lemma ~\ref{l:R-LE} and get
\[
 \Vert R_{\mu}\Vert_{L^2_tL^{\infty}_{loc}(B_{\lambda})}\Vert R_{\nu}\Vert_{L^2_tL^{\infty}_{loc}(B_{\lambda})}\lesssim \mu^{\frac{1}{2}} \nu^{\frac{1}{2}}\Vert  R\Vert^2_{LE^{-\frac{1}{2}}}.
\]
Summing over both $\mu , \nu < \lambda$ we get that indeed
\[
\sum_{\lambda}\Vert R_{\mu}R_{\nu}\Vert_{L^1_tL^{\infty}_{loc}(B_{\lambda})} \lesssim \lambda \Vert  R\Vert^2_{LE^{-\frac{1}{2}}},
\]
which finishes the proof of \eqref{R2-G1}.

For the high-high case the analysis in part (i) and (ii) applies together with the summation over $ \lambda$ and $\mu$. The bound for $G_2$ follows.

\end{proof}

Using the $|R|^2$ bound, we are able to estimate two more of the error terms:

\begin{proof}[\bf The estimate for $Err_1$ in Proposition~\ref{p:e124}]
We recall that
\[
Err_1  := \int_0^T \int  \sigma m_x \eta \mN(\eta) \psi \, dx dt .
\]
Since $\mN(\eta) \psi = |\nabla \phi|^2$ on the top, 
this is a direct consequence of Lemma~\ref{l:R2} (a).

\bigskip

\end{proof}

\begin{proof}[\bf Proof of the $Err_2$ estimate]

We recall that the expression for $Err_2$ is given  by
\[
Err_2 := \int_0^T \iint m_x(x-x_0) \theta_y \HN(|\nabla \phi|^2)\, dx dy dt.
\]
We first recast it in holomorphic coordinates,
\[
Err_2 = \int_0^T \iint m_x(x-x_0) (\Re W_\alpha+|W_\alpha|^2)  \HN(|R|^2)\, d\alpha d\beta dt.
\]
To estimate it  we will combine the bounds in Lemma~\ref{l:theta-LE-loc} with those in Lemma~\ref{l:R2}. We exploit these bound in two steps. 
\medskip

\emph{ 1. High frequency bounds.}
Here we consider the contributions where at least one of the $W_\alpha$ and $\HN(|R|^2)$ factors is at high frequency 
$( \geq 1)$. In this case the corresponding harmonic extension decays exponentially  in $\beta$  
on the unit scale, therefore the bound for the corresponding part of $Err_2$ is localized both in $\alpha$ and 
in $\beta$ on the unit scale. On this scale, by elliptic regularity, we have local bounds
\[
W_\alpha \in L^2_t (L^2_\beta H^{-\frac12}_{\alpha})_{loc}, \qquad 
\HN(|R|^2)  \in L^2_t (L^2_\beta H^{\frac12}_{\alpha})_{loc}
\]
in terms of the $LE^0$ norms for $\theta$ and $| R|^2$ on the top. These are dual spaces. Furthermore,
the remaining $W_\alpha$ factors are harmless since from the $X$ bound we have
\[
W_\alpha \in L^\infty_t L^\infty_\beta (\ell^1 H^\frac12_{\alpha}).
\]
 
\medskip

\emph{2. Low frequency bounds.} Here we use the decomposition in part (c) of the last lemma,
where $\lambda$ is matched to the depth $|\beta| \approx \lambda^{-1}$.
 
For $G^1_\lambda$ we combine \eqref{R2-G1}  with the trivial $L^\infty$ bound for $W_\alpha$ derived
fom the $X$ norm, where the latter comes with $\ell^1$ summability.

For  $G^1_\lambda$ instead we combine \eqref{R2-G1}  with the bound  \eqref{l:theta-LE-loc} for $W_\alpha$.
\end{proof}

\subsection{Bilinear estimates for 
\texorpdfstring{$  \Im W \cdot \Re W_\alpha $}{} and its 
harmonic extension}

This expression appears in the normal form correction part of the 
proof of our nonlinear Morawetz inequality. Here we will prove the following bound:

\begin{lemma}\label{l:W2}
a) The function $\Im W \cdot \Re W_\alpha$ restricted to the top satisfies the following 
high frequency estimate:
\begin{equation}\label{WW=high}
\| \Im W \cdot \Re W_\alpha \|_{LE^0} \lesssim \epsilon \| \Im W\|_{LE^0}. 
\end{equation}

b) Its low frequencies  satisfy the additional bound
\begin{equation}\label{WW-low}
 \|  P_\lambda (\Im W \cdot \Re W_\alpha) \|_{L^2_tL^\infty_{loc}(B_\lambda(x_0))} \lesssim 
c_\lambda  \| \Im W\|_{LE^0}. 
\end{equation}

c) In addition, for each $\lambda < 1$ there is a decomposition
\[
 P_{< \lambda}(\Im W \cdot \Re W_\alpha)  = G_{\lambda}^1 + G_{\lambda}^2,
\]
where
\begin{equation}\label{WW-G1}
\sup_{\lambda} \| G_{\lambda}^1 \|_{L^1_t L^\infty_{loc}(B_\lambda(x_0))} \lesssim \lambda  M^2,
\end{equation}
while
\begin{equation}\label{WW-G2}
 \| G_{\lambda}^2 \|_{L^2_t L^\infty_{loc}(B_\lambda(x_0))} \lesssim  c_\lambda  M.
\end{equation}

\end{lemma}

\begin{proof}
 a) Here we use the fact that $W_{\alpha}$ is bounded in $L^{\infty}$ 
\[
\begin{aligned}
\| \Im W \cdot \Re W_\alpha \|_{LE^0} &= \| \Im W \cdot \Re W_\alpha \|_{L^2_t L^2_{loc}(B_{\lambda}(x_0))} \\
& \lesssim \Vert \Im W \Vert _{L^2_t L^2_{loc}(B_{\lambda} (x_0))} \|\Re W_\alpha \|_L^{\infty}\\
 &\lesssim  \epsilon  \| \Im W\|_{LE^0} .
 \end{aligned}
\]
\medskip

b) The proof is  exactly as in Lemma~ \eqref{l:R2} with the corresponding adjustments that come from the fact that  $\Im W$  and $\Re W_{\alpha}$ are differently balanced in comparison with  $R$: one is $1/2$ derivative  less than $R$ and one is $1/2$ derivative above $R$, respectively. 

The only  slight technical difference that arises, is when one considers the case of low-high interactions, with the high frequency on $\Re W_{\alpha}$. In this case, instead of looking separately at the norms 
\[
\| \Im W_{\nu}\cdot \Re W_{\mu,\alpha}\|_{L^2_t L^{\infty}_{loc}(B_{\lambda}(x_0))}, \qquad \nu <\mu \lesssim 1,
\]
and then sum over  $\nu$ with $\nu <\mu$, we 
group terms and analyze directly  
\[
 \| \Im W_{<\mu} \cdot \Re W_{\mu,\alpha}\|_{L^2_t L^{\infty}_{loc}(B_{\lambda}(x_0))}.
\]
By doing so we avoid the potentially troublesome $\nu$ summation.

Thus, we proceed as follows
\[
 \| (\Im W)_{<\mu} \cdot (\Re W_{\alpha})_{\mu}\|_{L^2_tL^{\infty}_{loc}(B_{\mu})}\lesssim  \| (\Im W)_{<\mu} \|_{L^2_tL^{\infty}_{loc}(B_{\mu})}
\|  (\Re W_{\alpha})_{\mu}\|_{L^{\infty}} \lesssim c_\mu M,
\]
where for the first factor we have used Lemma~\ref{l:W-LE-top}, while the dyadic bound for $\Re W_{\alpha}$ follows from Proposition~\ref{p:control-equiv}. This suffices for both parts (b)
and (c) of the lemma.

\end{proof}

\begin{proof}[\bf Proof of the $Err_5^2$ estimate.]
Here we consider the bound for the second term in $Err_5$, namely
\[
Err_5^2 := g \int_0^T \iint \Im \left( \frac{1}{1+W_\alpha} \Til W_\alpha \right)  H_D (\Im W \Re W_\alpha) 
\, d\alpha d\beta dt .
\]
The same proof as for $Err_2$ applies,  using Lemma~\ref{l:theta-LE-loc}, which now 
 is combined with Lemma~\ref{l:W2} instead of Lemma~\ref{l:R2}.
\end{proof}

\bigskip

\begin{proof}[\bf Proof of the $Err_5^3$ estimate]
Here we consider the third term in $Err_5$, namely
\[
Err_5^3 = \int_0^T \iint \frac{1}{1+W_\alpha} \Im P[|R|^2]_\alpha \, H_D (\Im W \cdot\Re W_\alpha)\,  d\alpha d\beta dt.
\]
This is again the same proof as for $Err_2$, using Lemma~\ref{l:R2} and Lemma~\ref{l:W2}.
\end{proof}

%%%%%%%%%%%%%%%%%%%%%%%%%%%%%%%%%%%%%%%%%%%%%%%%%%%%%%%%%%%%%%%%%%%%%%%%%%
%%%%%%%%%%%%%%%%%%%%%%%%%%%%%%%%%%%%%%%%%%%%%%%%%%%%%%%%%%%%%%%%%%%%%%%%%%
%%%%%%%%%%%%%%%%%%%%%%%%%%%%%%%%%%%%%%%%%%%%%%%%%%%%%%%%%%%%%%%%%%%%%%%%%%
%%%%%%%%%%%%%%%%%%%%%%%%%%%%%%%%%%%%%%%%%%%%%%%%%%%%%%%%%%%%%%%%%%%%%%%%%5
\section{Bounds involving \texorpdfstring{$F$}{}}
\label{s:F}
%%%%%%%%%%%%%%%%%%%%%%%%%%%%%%%%%%%%%%%%%%%%%%%%%%%%%%%%%%%%%%%%%%%%%%%%%%
%%%%%%%%%%%%%%%%%%%%%%%%%%%%%%%%%%%%%%%%%%%%%%%%%%%%%%%%%%%%%%%%%%%%%%%%%%
%%%%%%%%%%%%%%%%%%%%%%%%%%%%%%%%%%%%%%%%%%%%%%%%%%%%%%%%%%%%%%%%%%%%%%%%%%
%%%%%%%%%%%%%%%%%%%%%%%%%%%%%%%%%%%%%%%%%%%%%%%%%%%%%%%%%%%%%%%%%%%%%%%%%5

The aim of this section is to prove the error estimates involving
$F$. These are all tied to the normal form correction we use to deal
with the unbounded error term $Err_3$.
We recall that 
\[
F = P\left[ \frac{2i \Im Q_\alpha}J\right] = R + P \left [ 2 i \Im \left( R \bar Y \right)\right]  :=  R+ F^{[2]},
 \]
where we have separated the linear part $F$ and the quadratic and 
higher order part $F^{[2]}$. The imaginary part of $F^{[2]}$ is explicit on the top:
\[
\Im F^{[2]} = \Im \left ( R \bar Y\right).
\]
Thus in the fluid domain we can write
\[
\Im F^{[2]} =  H_D \left( \Im \left ( R \bar Y \right) \right).
\]
In Eulerian coordinates, the expression $H_D \left(\Im \left ( R \bar Y \right)\right)$ arises as the nonlinear component of 
$\theta_t$, see \eqref{thetat}. Indeed, in holomorphic coordinates, we compute on the top
\[
\nabla \phi \nabla \psi = \Im \left( R \frac{\bar W_\alpha}{1+\bar W_\alpha} \right) = \Im \left( F - \frac{Q_\alpha}{1+W_\alpha} \right) = \Im (F -R).
\]

Understanding $\Re F^{[2]}$, on the other hand, is a slightly more delicate matter, since a-priori it is only determined modulo constants. In our setting, the constant in $\Re F$ is determined
by
\begin{equation}\label{ReF-const}
\begin{split}
\Re F^{[2]}(\alpha_0,0) = & \  
\left(\frac{\Im F \Im W_\alpha}{1+\Re W_\alpha} - \Re R\right) (\alpha_0,0)
\\ = & \
 \left(\frac{\Im (R \bar Y) \Im W_\alpha}{J(1+\Re W_\alpha)}-\Re R\right) (\alpha_0,0).
\end{split}
\end{equation}
We will not use the full expression in the sequel, but merely the bound
\begin{equation}\label{const-bd}
|\Re F^{[2]}(\alpha_0,0)| \lesssim |R(\alpha_0,0)|.
\end{equation}

In what follows we will first establish direct bounds for $\Im F^{[2]}$, which has a bilinear structure as described above. The real part will satisfy similar bounds except at very low frequencies $\leq 1/h$.

\subsection{Bilinear estimate for \texorpdfstring{$\Im F^{[2]} =
\HD(\nabla \phi \nabla \psi) $}{}}
For this expression we will prove the following bounds,
which  are needed in order to switch from 
$Err_3$ to $Err_3^{hol}$ and prove Proposition~\ref{p:e3-diff}:

\begin{lemma}\label{l:ImF}
a) The function $F^{[2]}$ restricted to the top satisfies the following high frequency estimate with $\lambda \geq 1$.
\begin{equation} \label{ImF-hi}
\|P_\lambda F^{[2]} \|_{LE}  \lesssim \lambda^\frac12 \epsilon M, \qquad \lambda \geq 1.
\end{equation}

b) It also satisfies the  low frequency bound
\begin{equation}\label{ImF-low}
 \| P_\lambda F^{[2]}  \|_{L^2_t L^\infty_{loc}(B_\lambda)} \lesssim \lambda^\frac12 \epsilon M, \qquad 1/h < \lambda < 1.
\end{equation}

c) Finally, at very low frequencies we have:
\begin{equation}\label{ImF-vlow}
 \| P_{<1/h} \Im F^{[2]}  \|_{L^2_t L^\infty_{loc}(B_{1/h})} \lesssim h^{-\frac12} \epsilon M.
\end{equation}

\end{lemma}

We also list some straightforward consequences of the above Lemma:

\begin{corollary}
The low frequency part of $\Im F^{[2]}$ satisfies on the top
\begin{equation}\label{ImF-hi+}
\|P_{\leq 1} \Im F^{[2]}  \|_{LE}  \lesssim  \epsilon M.
\end{equation}
Its harmonic extension satisfies the  bound
\begin{equation}\label{ImF-low-H}
 \| \HD(\nabla \theta \nabla \phi)  \|_{L^2_t L^\infty(A_\lambda)} \lesssim  \lambda^\frac12 \epsilon M.
\end{equation}
\end{corollary}

The estimates in part a) are not entirely satisfactory because the
$\ell^1$ summation with respect to $\lambda$ is missing for $\lambda >
1$.  Similarly  the $\ell^1$ summation with respect to $\lambda < 1$ is missing in part (b).  To compensate for that, we complement the above
result as follows:

\begin{lemma}\label{l:ImFa}
a) The function $F^{[2]}_h = F^{[2]}_{\geq 1}$ restricted to the top admits the following high frequency
decomposition
\[
 F^{[2]}_h = F^{[2],1}_{h}+ F^{[2],1}_{h},
\]
where the dyadic pieces of $F^{[2],1}_{h}$ satisfy
\begin{equation}\label{ImF-g1}
\| F^{[2],1}_{\lambda} \|_{L^2_t H^{-\frac12}_{loc}}  \lesssim  M c_\lambda,
\end{equation}
while  the dyadic pieces of $F^{[2],2}_{h}$ satisfy
\begin{equation}\label{ImF-g2}
\| F^{[2],2}_{\lambda} \|_{L^\infty_t H^{1}}  \lesssim  \epsilon c_\lambda.
\end{equation}

\end{lemma}
As a consequence of the previous lemma and interpolation (or by a similar direct proof),
we have

\begin{corollary}\label{c:Lp}
The function $F^{[2],2}_{h}$ in the last lemma also satisfies the interpolated bounds
 \begin{equation}\label{ImF-g2+}
\| F^{[2],2}_{h}\|_{L^p_t H^{s}_{loc}}  \lesssim        \epsilon^{2-\frac{2}{p}} M^{\frac{2}{p}},
\end{equation}
where
\[
2 < p < \infty, \qquad \frac{1}p = \frac{1-s}{3}.
\]
\end{corollary}

Similarly, to account for the lack of summability 
in the low frequency bound \eqref{ImF-low}, we have the following:

\begin{lemma}\label{l:ImFb}
We can decompose $F^{[2]}_{l}:= F^{[2]}_{[1/h,1]}$ into 
\[
F^{[2]}_{l} = F_l^{[2],1} + F_l^{[2],2},
\]
where the dyadic pieces of $F_l^{[2],1}$ satisfy the dyadic bounds
\begin{equation}\label{Fl1-le}
 \lambda^{-\frac12} \|F_\lambda^{[2],1}\|_{L^2_t L^\infty_{loc}(B_\lambda)} \lesssim c_\lambda M ,
\end{equation}
while  the dyadic pieces of $F_l^{[2]}$ satisfy the weaker bound
\begin{equation}\label{Fl2-le}
\sup_{\lambda} \lambda^{-\frac12}\|F_\lambda^{[2],1}\|_{L^2_t L^\infty_{loc}(B_\lambda)}  \lesssim \epsilon M,
\end{equation}
as well as the uniform bound 
\begin{equation}\label{Fl2-e}
 \lambda^{\frac12} \|F_\lambda^{[2],2}\|_{L^\infty}  \lesssim c_\lambda \epsilon . 
\end{equation}
\end{lemma} 
% The above analysis  also applies at frequency $\leq h^{-1}$ to $\Im (F-R)$. 

The bounds in Lemma~\ref{l:ImF}  will be used in order to estimate
trilinear terms. For quartic terms on one hand we have more
flexibility, and  Lemmas~\ref{l:ImFa}, \ref{l:ImFb} are more useful.

We now successively prove the above lemmas. 

\begin{proof}[Proof of Lemmas~\ref{l:ImF},~\ref{l:ImFa},~\ref{l:ImFb}]
Here we use the dyadic local energy bounds for $R$ in Lemma~\ref{l:R-LE}
as well as the local energy bounds for $Y$ in Lemma~\ref{l:Y}. On the other hand, in terms of the 
control norm, we have the bounds:
\[
\|R_\lambda\|_{L^2} \lesssim \lambda^{-1} c_\lambda, \qquad \| Y_\lambda\|_{L^2}
\lesssim \lambda^{-\frac12} c_\lambda.
\]

As usual we consider the Littlewood-Paley decomposition of $F^{[2]}_\lambda$, 
\[
F^{[2]}_\lambda = \P [ \Im ( R_\lambda \bar Y_{<\lambda}) +
\Im ( R_{< \lambda} \bar Y_{\lambda})] + \P  \sum_{\mu \geq \lambda}
P_\lambda  \Im ( R_{\mu} \bar Y_{\mu}) .
\]

The first two terms in this decomposition, namely the high-low and the low-high interaction,
are estimated in the same manner as in Lemmas~\ref{l:R2},\ref{l:W2}
to obtain the high frequency  bounds
\[
\| \P \Im ( R_\lambda \bar Y_{<\lambda}) \|_{L^2_t L^2_{loc}} +
\|\P \Im ( R_{< \lambda} \bar Y_{\lambda})]\|_{L^2_t L^2_{loc}}  \lesssim c_\lambda M \lambda^\frac12 ,
\]
respectively the low frequency bounds
\[
\| \P \Im ( R_\lambda \bar Y_{<\lambda}) \|_{L^2_t L^\infty_{loc}(B_\lambda)} +
\|\P \Im ( R_{< \lambda} \bar Y_{\lambda})]\|_{L^2_tL^\infty_{loc}(B_\lambda)}  \lesssim c_\lambda M \lambda^\frac12.
\]
These both suffice for Lemma~\ref{l:ImF}, and show that these contributions can be placed in $F^{[2],1}_{h}$ for Lemma~\ref{l:ImFa}, respectively in $F^{[2],1}_l$ for Lemma~\ref{l:ImFb}.

Thus it remains to consider the case of high-high interactions, $P_\lambda  \Im ( R_{\mu} \bar Y_{\mu})$.
Here we separate the analysis into low and high frequencies.

\medskip

\textbf{ A. High frequencies $1 \leq \lambda \leq \mu$.} 
Here we estimate again as in Lemmas~\ref{l:R2},\ref{l:W2},
\[
\| P_\lambda  \Im ( R_{\mu} \bar Y_{\mu})\|_{L^2_t L^2_{loc}} 
\lesssim \lambda^\frac12 \| \Im ( R_{\mu} \bar Y_{\mu})\|_{L^2_t L^1_{loc}}
\lesssim \lambda^\frac12 \| R_{\mu}\|_{L^2_t L^2_{loc}} \| \bar Y_{\mu}\|_{L^2_t L^2}
\lesssim \lambda^\frac12 c_\mu M.
\]
This suffices for the $\mu$ summation, which yields the conclusion of Lemma~\ref{l:ImF}(a),
but yields no $\lambda$ summation due to a lack of off-diagonal decay. Because of this,
for Lemma~\ref{l:ImFa} we place this term in $F^{[2],2}_{h}$ and estimate it  by
\[
\| P_\lambda  \Im ( R_{\mu} \bar Y_{\mu})\|_{L^\infty_t L^2} \lesssim 
 \lambda^\frac12 \|   \Im ( R_{\mu} \bar Y_{\mu})\|_{L^\infty_t L^1}
\lesssim  \lambda^\frac12 \|  R_{\mu}\|_{L^\infty L^2} \| Y_{\mu}\|_{L^\infty_t L^2}
 \lesssim  \lambda^\frac12 \mu^{-\frac32} c_\mu^2,
\]
where we have off-diagonal decay,
\[
\sum_{\mu > \lambda}\| P_\lambda  \Im ( R_{\mu} \bar Y_{\mu})\|_{L^\infty_t L^2}  \lesssim  \lambda^{-1} c_\lambda,
\]
as desired.

\medskip

\textbf{ B. Low frequencies $1 \leq \lambda \leq \mu$.} 
Here we should also consider two cases, $\mu \leq 1$ and $\mu > 1$. The latter case is 
similar but simpler, so it is omitted. Assuming $\lambda\leq \mu \leq 1$ we compute 
\[
\begin{split}
\| P_\lambda  \Im ( R_{\mu} \bar Y_{\mu})\|_{L^2_t L^\infty_{loc}(B_\lambda)} 
\lesssim & \ \lambda \| \Im ( R_{\mu} \bar Y_{\mu})\|_{L^2_t L^1_{loc}(B_\lambda)}
\lesssim \lambda \| R_{\mu}\|_{L^2_t L^2_{loc}(B_\lambda)} \| \bar Y_{\mu}\|_{L^\infty_t L^2}
\\ \lesssim & \  \mu^\frac12 \lambda^\frac12 \| R_{\mu}\|_{L^2_t L^2_{loc}(B_\mu)} \| \bar Y_{\mu}\|_{L^\infty_t L^2}
\lesssim  \lambda^\frac12 c_\mu M.
\end{split}
\]

This is again good enough for Lemma~\ref{l:ImF}(c), but there is no $\lambda$ summation.
Hence, for Lemma~\ref{l:ImFb} we place these contributions in $F_\lambda^{[2],2}$, and estimate them 
exactly as in the high frequency case.
\end{proof}

\begin{proof}[\bf The bound for $Err_3 - Err_3^{hol}$: proof of Proposition~\ref{p:e3-diff}]

 The expression $Err_3$ is given by
\[
Err_3 = \int_0^T \iint m_x(x-x_0) \Im R \,  \HD(\nabla \theta \nabla \phi)\, dx dy dt ,
\]
while its holomorphic counterpart is 
\[
Err_3^{hol} = \int_0^T \iint m_\alpha(\alpha-\alpha_0) \Im R \, \HD(\nabla \theta \nabla \phi)\, dx dy dt.
\]
In their difference we obtain errors due to (i) Jacobian terms and (ii) the switch between Eulerian vertical strips and the vertical strips in holomorphic coordinates. We estimate the difference using Proposition~\ref{p:switch-int}, which yields
\[
| Err_3 - Err_3^{hol} | \lesssim D_1 + D_2,
\]
where 
\[
D_1 := \int_0^T \iint_{A_1} (| W_\alpha| + | \Re W(\ - \Re W(\alpha_0,0)|) |R||\HD(\nabla \theta \nabla \phi)|\,  d\alpha d\beta dt.\,
\] 
\[
D_2 := \int_{0}^T \int_{-h}^{-1} 
( c_\beta + \sup_{|\alpha - \alpha_0| < |\beta|} |W_\alpha|)
\sup_{|\alpha - \alpha_0| < |\beta|} |R| |\HD(\nabla \theta \nabla \phi)| + |\beta| | \partial_\alpha [R\HD(\nabla \theta \nabla \phi)]|
\, d\beta dt.
\]

\medskip

{\bf A. The estimate for $D_1$.}
Here  we use the decomposition in Lemma~\ref{l:ImFa}.  The
harmonic extension of (the high frequencies of) $F^{[2],1}_{h}$ belongs to
$L^2_t L^2_{loc}$ by elliptic regularity, which is combined with the similar bound for $R$,
and suffices. To deal with $F^{[2],2}_{h}$ we imbalance a bit the scales using Corollary~\ref{c:Lp}. Working with $s \in (-\frac12,\frac12)$ we
obtain that its harmonic extension satisfies
\[
\| H_N(F^{[2],2}_{h})\|_{L^{p}_t L^q_{loc}} \lesssim M^{\frac{2}{p}} \epsilon^{2-\frac{2}{p}}, \qquad 2 < q < \infty, \qquad \frac{3}{p} - \frac{2}{q} = \frac12.
\]

Consider first the $W_\alpha$ term, which we also imbalance,  interpolating in a similar manner
between the energy and the local energy bound. This yields
\[
\| W_\alpha\|_{L^{p_1}_t L^{q_1}_{loc}} \lesssim M^{\frac{1}{p_1}} \epsilon^{1-\frac{1}{p_1}}, 
\qquad 2 \leq  q_1 \leq  \infty, \qquad \frac{3}{p_1} - \frac{2}{q_1} = 0.
\]
 We choose exponents appropriately so that 
\[
\frac{1}p + \frac{1}{p_1} = \frac{1}{1}+\frac{1}{q_1} = \frac12.
\]
Then we multiply, combining with the $L^2$ local energy bound for $R$ and using H\"older's inequality. 

Next we consider the $|\Re W(\alpha,\beta) - \Re W(\alpha_0,0)|$ term. 
To argue as for the previous difference we simply estimate it by the Fundamental Theorem of Calculus
\[
\| \Re W(\alpha,\beta) - \Re W(\alpha_0,0)\|_{L^{p_1}_t L^{q_1}_{loc}}
\lesssim \| W_\alpha \|_{L^{p_1}_t L^{q_1}_{loc}},
\]
and conclude in the same manner.
\medskip

{\bf B. The estimate for $D_2$.}
Consider a dyadic frequency $\lambda < 1$. Then in the corresponding regions $A_\lambda$ we have by \eqref{ImF-low}
\[
\|  \HD(\nabla \theta \nabla \phi)\|_{L^2_t L^\infty(A_\lambda)} 
\lesssim \lambda^{-1} \| \nabla \HD(\nabla \theta \nabla \phi)\|_{L^2_t L^\infty(A_\lambda)} 
\lesssim \epsilon M \lambda^{\frac12}.
\]
Combined with the bound for $R$ in Lemma~\ref{l:R-LE}
this suffices for the $c_\beta$ term. It remains to consider the 
$W_\alpha$ term.  For that it suffices to match the above bound with a corresponding bound for $R W_\alpha$,
\begin{equation}\label{RWa}
\| \sup_{A_\lambda} |R|  \sup_{A_\lambda} |W_\alpha| \|_{L^2_t}
 \lesssim c_\lambda M \lambda^{\frac12}.
\end{equation}
It remains to prove \eqref{RWa}. Harmlessly neglecting the 
exponentially decaying tails at higher frequencies, we write in $A_\lambda$
\[
R  = \sum_{\mu \leq \lambda} R_\mu, \quad   W_\alpha =\sum_{\nu \leq \lambda}  W_{\nu,\alpha}.
\]
For $\mu < \nu$ we write
\[
\| \sup_{A_\lambda} |R_{\mu}|  \sup_{A_\lambda} |W_{\nu,\alpha}| \|_{L^2_t}
\lesssim \| R_\mu\|_{L^2_t L^\infty_{loc}(B_\lambda)} 
\| W_{\nu,\alpha}\|_{ L^\infty_tL^{\infty}_{\alpha}} \lesssim \mu^{\frac12} c_\nu M,
\]
while for $\nu < \mu$ we have
\[
\| \sup_{A_\lambda} |R_{\mu}|  \sup_{A_\lambda} |W_{\nu,\alpha}| \|_{L^2_t}
 \lesssim \| R_\mu\|_{ L^{\infty}_t L^{\infty}_{\alpha} }
\| W_{\nu,\alpha}\|_{ L^2_t L^\infty_{loc}(B_\lambda)} \lesssim \mu^{-\frac12} \nu  c_\nu M,
\]
and \eqref{RWa} follows in both cases after $\mu,\nu$ summation.

The proof of Proposition~\ref{p:e3-diff} is concluded.

\end{proof}

\subsection{Estimates involving  \texorpdfstring{$F$}{}}

There are three error terms which involve the full $F$, namely $Err_5$, $Err_6$ and $Err_7$. In this section we will estimate these terms. We need to deal with $F$ in the following combinations: 

\begin{enumerate}
\item The harmonic function 
\begin{equation}\label{G1}
G_1 := \Im(F R_\alpha).
\end{equation}

\item The harmonic function 
\begin{equation}\label{G2}
G_2 := \Im(F^{[2]} W_\alpha).
\end{equation}

\item The harmonic extension
\begin{equation}\label{G3}
G_3 := \HD( \Im W \Re F(1+W_\alpha)).
\end{equation}
\end{enumerate}

We will state our main bounds directly in terms of these expressions, rather than in terms of $F$. This is because the bounds for $G_1$, $G_2$ and $G_3$ are better viewed as trilinear bounds, rather than more directly as iterated bilinear bounds. We begin with $G_1$ and $G_2$, where the results
are easier to state:

\begin{proposition}\label{p:nonlin}
a) High frequency bounds. The functions $G_1$, and $G_2$ have the following
regularity in the fluid domain:

\begin{equation}\label{G1-high}
\|G_1\|_{L^2_t L^2_\beta H^{-\frac12}_\alpha(A^1(x_0))} 
\lesssim \epsilon M,
\end{equation}

\begin{equation}\label{G2-high}
\| G_2\|_{L^2_t L^2(A_1(x_0))} \lesssim \epsilon^2 M.
\end{equation}.

b) Low frequency bounds:
\begin{equation}\label{G1-low}
\|G_1\|_{L^2_t L^\infty(A_\lambda(x_0))}  \lesssim  \lambda c_\lambda M ,
\end{equation}
\begin{equation}\label{G2-low}
\||G_2\|_{L^2_t L_{\alpha}^\infty(A_\lambda(x_0))}   \lesssim  \lambda^{\frac12} c_\lambda M .
\end{equation}
\end{proposition}

We postpone the proof of the proposition, in order to complete the proof of the 
$Err_5^1$ and $Err_6$ bounds.

\bigskip

\begin{proof}[\bf Proof of the bound for $Err_5^1$.]
We recall that
\[
Err_5^1:=   \int_0^T \iint m_\alpha  \Im (FR_\alpha) \HD\left( \Im W \Re W_\alpha
\right) \,  d\alpha d\beta dt.
\]
At high frequency this is estimated combining \eqref{G1-high} and \eqref{WW=high}. 
At low frequency instead we combine  \eqref{G1-low} and \eqref{WW-low}. 
\end{proof}

\bigskip

\begin{proof}[\bf Proof of the bound for $Err_6$.]
We recall that
\[
Err_6=   \int_0^T \iint m_\alpha  \Im R  \Im( (F-R) W_\alpha)
\, d\alpha d\beta dt.
\]
This we can estimate using Lemma~\ref{l:R-LE} for $R$, and \eqref{G2-high},   \eqref{G2-low} for the second factor.
\end{proof}

Next we consider the bounds for $G_3$, which are summarized in the following:

\begin{proposition}\label{p:nonlin2}
For each $1/h < \mu < 1$, the function $G_3$ admits a decomposition
\[
G_3 = G_3^{\mu,1} + G_3^{\mu,2},
\]
where the two components satisfy estimates as follows:

a) High frequency bounds.
\begin{equation}\label{G31-high}
 \| \partial_\alpha G_3^{\mu,1}\|_{L^2_t L^2(A_1(\alpha_0))} \lesssim \mu^{-\frac12} c_\mu   M,
\end{equation}
\begin{equation}\label{G32-high}
\sup_\mu \mu^{-\frac12} \| \partial_\alpha G_3^{\mu,2}\|_{L^1 L^2(A_1(\alpha_0))} \lesssim  M^2.
\end{equation}

b) Low frequency bounds: 
\begin{equation}\label{F-G31}
    \mu^\frac12  \| \partial_\alpha G_3^{\mu,1}\|_{L^2_t L^\infty(A_\mu(\alpha_0))}
  \lesssim  c_\mu  M ,
\end{equation}
respectively 
\begin{equation}\label{F-G32}
  \sup_\mu  \mu^{-\frac12}  \| \partial_\alpha G_3^{\mu,2}\|_{L^1_t L^\infty (A_\mu(\alpha_0))}  \lesssim   M^2 .
\end{equation}
\end{proposition}

We now use this Proposition to estimate the remaining error:

\bigskip

\begin{proof}[\bf Proof of the bound for $Err_7$.]
We recall that
\[
Err_7:=   \int_0^T \iint m_\alpha  \Im R \partial_\alpha G_3 \, d\alpha d\beta dt.
\]
We first estimate the bound for $R_{\geq 1}$, using the decomposition
above with $\mu = 1$. The $G_3^{1,1}$ contribution is easy to bound
using the local energy for $R$. The $G_3^{1,2}$ contribution is also
easy to bound using the uniform $H^1_h$ control norm for $R$.

Then we estimate the contribution of $R_{\mu}$, where we use the above
decomposition associated  exactly to the frequency $\mu$. Precisely, 
we match the $G_3^{\mu,1}$ bound with the local energy estimate for $R_\mu$,
while on the other hand we match the  $G_3^{\mu,2}$ bound with the uniform bound for $R$
using the control norm.
\end{proof}

The remainder of the section is devoted to the proof of
Propositions~\ref{p:nonlin},\ref{p:nonlin2}. In estimating the
contributions of $F$ we will separately consider three regimes:

\begin{description}
\item[I. High frequencies] $\geq 1$. Here the real and imaginary part of $F$ satisfy similar estimates.

\item[II. Low frequencies] $\in [h^{-1},1]$. Again the real and
  imaginary part of $F$ satisfy similar estimates. We also include
  here the very low frequencies of $\Im F$ and $R$.

\item[III. Very low frequencies] $\leq h^{-1}$ for $\Re (F-R)$. It is
  here that the difference between $\Re F$ and $\Im F$ comes into
  play, along with the assignment of constants as discussed in the
  beginning of the section.

\end{description}

\bigskip

\textbf{I. The high frequencies of $F$.} For the $R_h$ component of $F_h$ we simply use Lemma~ \ref{l:R-LE}. For the high frequencies $F^{[2]}_h$ of $F^{[2]}$ we instead rely on 
Lemmas~\ref{l:ImF}, \ref{l:ImFa}. Only in a few cases we need to backtrack further
and use the structure of $F^{[2],1}_{h}$ and $F^{[2],2}_{h}$.

\bigskip

\textbf{ I.a. The contribution of $F_h$ to  $G_1$.}  We use
the bilinear Littlewood-Paley expansion
\begin{equation}\label{Fh-Ralpha}
P_\lambda (F_h R_\alpha) = F_{h,\lambda}  R_{<\lambda,\alpha} +  F_{h,<\lambda}  R_{\lambda,\alpha} +
\sum_{\mu \geq \lambda}  P_\lambda(F_\mu  R_{\mu,\alpha}).
\end{equation}
For $F_h$ we use the expansion
\[
F_h = R_h + F_h^{[2],1} + F_h^{[2],2},
\]
where the last two terms are as in Lemma~\ref{l:ImFa}. We successively consider the three terms in \eqref{Fh-Ralpha}. 

For the first term in \eqref{Fh-Ralpha} it is easy to bound the output of the 
$F^{[2],2}_\lambda$ component. Indeed, using the local energy norm for $R$ and 
Lemma~\ref{l:R-LE-top}, we have
\[
\| R_{\alpha}\|_{L^2_t H^{-\frac32}_{loc}} \lesssim M,
\]
 which can be in turn easily combined with \eqref{ImF-g2}. 

The output of the $R_\lambda+ F^{[2],1}_\lambda$ component is more difficult to estimate. We recall that $F^{[2],1}_\lambda$ arises from unbalanced frequency interactions,
so we expand it as 
\[
F^{[2],1}_\lambda = \P [ \Im (R_\lambda (\bar Y_{< \lambda}+2)) + \Im ( R_{<\lambda} \bar Y_{\lambda})] .
\]
Multiplying this by $R_{<\lambda,\alpha}$ we obtain a trilinear form,
for which we need to balance the three input frequencies. There are
two terms to consider, and we only consider the worst one,
\[
\P( Y_{\lambda} \bar R_{< \lambda}) R_{<\lambda,\alpha}.
\]
This is estimated using H\"older's inequality and Bernstein's inequality as follows:
\[
\begin{aligned}
\| \P( Y_{\lambda} \bar R_{< \lambda}) R_{<\lambda,\alpha}\|_{L^2_t H^{-1}_{loc}}&\lesssim \lambda^{-1}\|  Y_{\lambda} \bar R_{< \lambda} R_{<\lambda,\alpha}\|_{L^2_t L^2_{loc}}\\
&\lesssim \lambda^{-1} \Vert Y_{\lambda} \Vert_{L^{\infty}_t L^2}\Vert \bar{R}_{<\lambda}\Vert_{L^2_t L^{\infty}} \Vert R_{<\lambda, \alpha}\Vert_{L^{\infty}_t L^{\infty}} \\
&\lesssim \lambda^{-1} \Vert Y_{\lambda} \Vert_{L^{\infty}_t L^2} \, \lambda^{\frac{1}{2}}\,\Vert \bar{R}_{<\lambda}\Vert_{L^2_t L^{2}} \, \lambda^{\frac{1}{2}}\, \Vert R_{<\lambda, \alpha}\Vert_{L^{\infty}_t L^2} \\
&\lesssim \| Y_\lambda \|_{L^\infty_t \dot H^\frac12} \| R\|_{L^2_t H^{-\frac12}_{loc}} \| R_{\alpha} \|_{L^\infty_t L^2}\\
&\lesssim c_\lambda M \epsilon .
\end{aligned}
\]
 Here the three factors on the right are estimated using the uniform control norm, local energy,
respectively the uniform control  norm.

For the second term in \eqref{Fh-Ralpha} it is easy to bound the output of the $R$ and the 
$F^{[2],1}$ component, using the uniform control norm for $R$.

We are left with the $F^{[2],2}$ component, which arises from balanced interactions of $R$ and $Y$. Thus 
we obtain again a trilinear form. Precisely, we  need to bound in  $L^\infty_t L^2_{loc}$ the expression
\[
\sum_{\mu_1 \lesssim \mu,\lambda} P_{\mu_1} (R_{\mu} \bar Y_{\mu}) R_{\lambda,\alpha} = \sum_{\mu < \lambda} (R_{\mu} \bar Y_{\mu}) R_{\lambda,\alpha}+
\sum_{\mu \geq \lambda} P_{< \lambda} (R_{\mu} \bar Y_{\mu}) R_{\lambda,\alpha}.
\]
Here for the first sum where $\mu < \lambda$ we use H\"older's inequality followed by Bernstein's inequality and arrive at the bound
\begin{equation}\label{RmuYmu-low}
 \left\|  (R_{\mu} \bar Y_{\mu}) R_{\lambda,\alpha}\right\|_{L^2_t H^{-1}_{loc}} \lesssim \mu 
\| R\|_{L^2_t H^{-\frac12}_{loc}}
 \|Y\|_{L^\infty_t L^{\infty}} \|R_{\lambda,\alpha} \|_{L^\infty L^2} \lesssim \mu c_\lambda \epsilon M,
\end{equation}
which suffices after $\mu$ summation. Here we used the local energy for $R$ and  the uniform control norm  for the remaining factors.

On the other hand for the second sum which corresponds to the
range $\mu \geq \lambda$  we estimate
\begin{equation}\label{RmuYmu-high}
 \left\|   (R_{\mu} \bar Y_{\mu}) R_{\lambda,\alpha}  \right\|_{L^2 _t H^{-1}_{loc}} \lesssim \lambda^2 \mu^{-1} \| R_\mu\|_{L^\infty_t H^1}
 \|Y\|_{L^\infty_t L^{\infty}} \|R_\alpha \|_{L^2_t H^{-\frac32}_{loc}} \lesssim \mu^{-1} \lambda^2  c_\mu \epsilon M,
\end{equation}
which again suffices. This corresponds to using local energy for $R_\alpha$ and the uniform 
control norm for the remaining factors.

Finally, the third term in \eqref{Fh-Ralpha} is negligible since we are multiplying two holomorphic functions, 
 so the output at low   frequency is exponentially small.

\bigskip

\textbf{ I.b. The contribution of $F_h$ to  $G_2$.}
This is given by $F^{[2]}_h W_\alpha$. We use again the Littlewood-Paley trichotomy,
\begin{equation}\label{FWalpha}
P_\lambda (F^{[2]}_h W_\alpha) = F^{[2]}_{h,\lambda}  W_{<\lambda,\alpha} +  F^{[2]}_{h,<\lambda}  W_{\lambda,\alpha} +
\sum_{\mu > \lambda}  P_\lambda(F^{[2]}_{h,\mu}  W_{\mu,\alpha}).
\end{equation}
The first term is easy for the $F^{[2],2}_h$ component, where we use the 
local energy norm for $\Im W$.

For the $F^{[2],1}_h$ component  we again expand $F^{[2],1}_h$ as a bilinear form in $R$ and $Y$
which contains only high-low interactions, obtaining a trilinear form. As before we have two contributions, 
of which we describe the worst, namely 
\[
\sum_{\lambda > 1} \P(\bar R_{<\lambda}Y_{\lambda}) W_{<\lambda,\alpha}.
\] 
This is estimated  by
\begin{equation}\label{RmuYmu-low+}
 \left\| \P((\bar R_{<\lambda}Y_{\lambda}) W_{<\lambda,\alpha} \right\|_{L^2_t H^{-\frac12}_{loc}} 
\lesssim \| R\|_{L^2_t H^{-\frac12}_{loc}}  \|Y_{\lambda}\|_{L^\infty_t L^{\infty}_{\alpha}} \|W_\alpha \|_{L^\infty_tL^{\infty}_{\alpha}} \lesssim c_\lambda \epsilon M.
\end{equation}

Consider now the second term in \eqref{FWalpha}.
The  bound is easy for the $F^{[2],1}_h$ component. For the $F^{[2],2}_h$ 
component we need to consider the sum
\[
\sum_{\mu_1 < \mu,\lambda} P_{\mu_1} ( R_\mu \bar Y_{\mu}) W_{\alpha,\lambda}
= \sum_{\mu < \lambda} ( R_\mu \bar Y_{\mu}) W_{\alpha,\lambda}+\sum_{\lambda \leq \mu} P_{<\lambda} ( R_\mu \bar Y_{\mu}) W_{\alpha,\lambda}.
\]
Again we estimate the two terms differently. For the first sum where 
$\mu < \lambda$ we compute by H\"older and Bernstein's inequalities at fixed time
\[
\left\| \sum_{1 \leq \mu < \lambda}  R_\mu \bar Y_{\mu} W_{\alpha,\lambda} \right\|_{L^2_t H^{-\frac12}_{loc}} 
\lesssim \lambda^{-\frac12} \mu^\frac12 
\| R  \|_{L^2_t \dot H^{-\frac12}_{loc}} \|Y\|_{L^\infty_tL^{\infty}_{\alpha}} \| W_{\lambda,\alpha}\|_{L^\infty _tH^\frac12_h} \lesssim
  \lambda^{-\frac12} \mu^\frac12 c_\lambda \epsilon M,
\]
using the local energy bound for $R$ and the uniform control norm for
the other two factors. Similarly, for the second sum, where $\mu \gtrsim \lambda$,  we have the fixed
time bound
\[
\left\| P_{<\lambda} ( R_\mu \bar Y_{\mu}) W_{\alpha,\lambda} \right\|_{L^2_t H^{-\frac12}_{loc}} 
\lesssim \lambda \mu^{-1} \| R_\mu \|_{L^\infty_t H^1_h} \|Y\|_{L^\infty_t L^{\infty}_{\alpha}} \| W\|_{L^2_t L^2_{loc}} \lesssim
  \lambda \mu^{-1} c_\mu \epsilon M,
\]
which again suffices.

Finally, the third term in \eqref{FWalpha} is negligible as the two factors are 
holomorphic.

\bigskip

\textbf{ I.c. The contribution of $F_h$ to  $G_3$.}
Here we will estimate directly the product $F_h \Im W$, as  $F_h W_\alpha$ is easily seen to satisfy the same bounds as $F_h$. 
We use again the Littlewood-Paley trichotomy,
\begin{equation}\label{FimW}
P_\lambda (F_h W) = F_{h,\lambda}  \Im W_{<\lambda} +  F_{h,<\lambda}  W_{\lambda} +
\sum_{\mu > \lambda}  P_\lambda(F_{h,\mu}  W_{\mu}).
\end{equation}

The contribution of $R$ to the first term is easy to bound in $L^2_t
H^{-\frac12}_{loc}$, and so can be included in $G_3^{1,1}$. Consider now
the expression
\[
\sum_{\lambda \geq 1} F^{[2]}_{h,\lambda} \Im W_{<\lambda}.
\]
Here it is easy to estimate the contribution of $F^{[2],2}_h$,
using the local energy bound for $\Im W$.
Hence we consider the contribution of $F^{[2],1}_h$, which contains 
the high-low interactions of $R$ and $Y$ in $F^{[2]}$ . We expand this as a trilinear
form, obtaining two terms depending on whether $R$ or $Y$ is at high frequency.
The better term is
\[
R_\lambda Y_{< \lambda} W_{<\lambda},
\]
where the second factor is harmless so this is no different than the corresponding 
contribution of $R$.

The worst term is
\[
Y_{\lambda} \bar R_{<\lambda} \Im W_{<\lambda}.
\]
To bound it we consider several cases depending on the frequencies 
of $R$ and $\Im W$:
\medskip

(i) Both frequencies $\geq 1$. Then we have the fixed time estimate
\[
\begin{split}
\| Y_{\lambda}  \bar R_{[1,\lambda)} \Im W_{[1,\lambda)} \|_{ L^2_t H^\frac12_{loc}} 
\lesssim & \  \|Y_\lambda \|_{L^\infty_t H^\frac12_h}
 ( \| R\|_{L^\infty_t \dot H^1_h} \|\Im W\|_{L^2_t L^2_{loc}} + \| R\|_{L^2_t H^{-\frac12}_{loc}} \|\Im W\|_{L^\infty_t H_{\alpha}^\frac32})
\\
\lesssim & \ c_\lambda \epsilon M,
\end{split}
\]
where we balance norms depending on which of the frequencies of $R$ and $W$
is larger. This contribution is 
included in $G_3^{1,1}$.

\medskip 

(ii) One frequency $\geq 1$, and one $\leq 1$. Here the same argument as above
applies, where we bound the low frequency factor in $L^2_t L^\infty_{loc}$.
Again here we use $G_3^{1,1}$.

\medskip

(iii) Both frequencies $\leq 1$. This is the more difficult term, 
where we need the parameter $\mu$ and the $G_3^{\mu,2}$ component.
 This is where we differentiate depending on the 
frequency of $R$. If the frequency of $R$ is less than $\mu$ then we use 
the local energy bound for $R$, and add that contribution to $G_3^{\mu,2}$,
If the frequency of $R$ is larger than $\mu$ then we use 
the energy bound for $R$, and add that contribution to $G_3^{\mu,1}$.

The low-high case, i.e. the second term in \eqref{FimW}, is 
similar to the like one for $G_1$, and goes into $G_3^{1,1}$.

Unlike in the case of $G_1$ or $G_2$, here the high-high to low case
is also nontrivial. We consider it next. For the two components of $F^{[2]}_h$ we estimate 
for $1 < \lambda \lesssim \mu$
\[
\| P_{\lambda} (F^{[2],1}_{h,\mu} \Im W_{\mu})\|_{L^2_t H^\frac12_{loc}} \lesssim \lambda^\frac12 \mu^{-\frac12} 
\| F_{h,\mu}^{[2],1} \|_{L^2_t H^{-\frac12}_{loc}} 
\| \Im W_\mu \|_{L^\infty_t H^\frac32_h} \lesssim \lambda^\frac12 \mu^{-\frac12} c_\mu M,
\]
respectively 
\[
\| P_{\lambda} (F^{[2],2}_{h,\mu} \Im W_{\mu}) \|_{L^2_t H^\frac12_{loc}} \lesssim\lambda^\frac12 \mu^{-\frac12}  \| F_{h,\mu}^{[2],2} \|_{L^\infty_t H^{1}_h}
 \| \Im W\|_{L^2_t L^2_{loc}} \lesssim \lambda^\frac12 \mu^{-\frac12} c_\mu M,
\]
both of which suffice after $\mu$ summation.  Both of these components go into $G_3^{1,1}$.

The same estimates also apply for $\lambda = 1$ when $P_1$ is replaced by $P_{\leq 1}$.
This addresses the low frequency bounds $\leq 1$ in  $G_3^{1,1}$ .

\bigskip

\textbf{II. The low frequencies of $F$.} Here we consider the low frequencies
 \[
F_l: = F_{[h^{-1},1]}.
\]
Our main tool  will be the decomposition for $F^{[2]}_l$ provided
by Lemma~\ref{l:ImFb}. One consequence of  Lemma~\ref{l:ImFb} is the bound
\begin{equation}\label{Fl-full}
\| F_l \|_{L^2_t L^\infty_{loc}} \lesssim M,
\end{equation}
which will be used to handle with the contribution of $F_l$ to the high frequencies of $G_1$, $G_2$ and $G_3$.

\medskip

\textbf{ II.a. The contribution of $F_l$ to $G_1$.} 
The contribution of $R$ is easy to estimate using Lemma~\ref{l:R-LE}.
The high frequencies $\geq 1$ are in turn directly estimated using \eqref{Fl-full}.

Here it remains to bound the expression
\[
\sum_{\lambda_1, \lambda_2 \leq 1} F^{[2]}_{l,\lambda_1} R_{\lambda_2,\alpha}
\]
in $L^2_t L^\infty_{loc}(A_\lambda)$.  Restricting to $A_\lambda$ limits the frequencies
$\lambda_1,\lambda_2$ to $[1/h,\lambda]$ with only  exponentially decaying tails at higher 
frequencies. We consider two cases:

If $\lambda_1 \leq \lambda_2$ then  we can use \eqref{Fl1-le}  and \eqref{Fl2-le}, and combine this with 
the uniform control bound for $R$ and Bernstein's inequality,
\[
\| F^{[2]}_{l,\lambda_1} R_{\lambda_2,\alpha}\|_{L^2_t L^\infty_{loc}(A_\lambda)} \lesssim 
\|  F^{[2]}_{l,\lambda_1}\|_{L^2_t L^\infty_{loc}(A_\lambda)} \|  R_{\lambda_2,\alpha}\|_{L^\infty_t L^\infty_{\alpha}}
\lesssim \lambda_1^\frac12 \lambda_2^{\frac12} c_{\lambda_2} M,
\]
which suffices after $\lambda_1,\lambda_2$ summation.

If $\lambda_1 > \lambda_2$ we can still estimate the contribution
of $F_l^{[2],2}$ using  \eqref{Fl2-e} combined
with  the pointwise bound for $ R_{\lambda_2,\alpha}$ derived from local energy
in Lemma~\ref{l:R-LE},
\[
\| F^{[2]}_{l,\lambda_1} R_{\lambda_2,\alpha}\|_{L^2_t L^\infty_{loc}(A_\lambda)} \lesssim 
\|  F^{[2]}_{l,\lambda_1}\|_{L^\infty_t L^{\infty}_{\alpha}} \|  R_{\lambda_2,\alpha}\|_{L^2_t L^\infty_{loc}(A_\lambda)}
\lesssim \lambda_1^{-\frac12} \lambda_2^{\frac32} c_{\lambda_1} M.
\]

This leaves us  only with the contribution of $F_l^{[2],1}$, which we expand to a trilinear
expression, arriving at  an expression of the form
\[
\sum_{\lambda_2, \lambda_3  \leq \lambda_1 \leq \lambda}  R_{\lambda_3} Y_{\lambda_1} R_{\lambda_2,\alpha}
+  R_{\lambda_1} Y_{\lambda_3} R_{\lambda_2,\alpha}.
\]
Here we apply the local energy bound for the factor with the lowest frequency
$\lambda_{\min}$, and use the uniform control norm for the two highest frequencies. 
Estimating as above this yields a bound 
\[
\| R_{\lambda_3} Y_{\lambda_1} R_{\lambda_2,\alpha}
+  R_{\lambda_1} Y_{\lambda_3} R_{\lambda_2,\alpha}\|_{L^2 _tL^\infty_{loc}(A_\lambda)}
\lesssim    \lambda_{min}^{\frac12} \lambda_1^{\frac12} c_{\lambda_1} M,
\]
where we have off-diagonal decay for the  summation.

\medskip

\textbf{ II.b. The contribution of $F_l$ to $G_2$}. 
The contribution of $R$ is easy to estimate using Lemma~\ref{l:R-LE}.
The high frequencies $\geq 1$ are in turn directly estimated using \eqref{Fl-full}.

It remains to estimate the low frequency contribution of $F^{[2]}_l$,
for which we consider the decomposition in Lemma~\ref{l:ImFb}. This
time the contribution of $F_l^{[2],1}$ is easy to bound, using the pointwise
estimate for $W_\alpha$,
\[
\| F_{l,<\lambda}^{[2],1} W_{<\lambda,\alpha}\|_{L^2_tL^\infty(A_\lambda)} \lesssim \lambda^{\frac12} c_\lambda M.
\]

 This leaves us with the contribution of $F_l^{[2],2}$, i.e., with terms of the form
\[
\sum_{\lambda_1, \lambda_2 < \lambda} F_{l,\lambda_1}^{[2],2} W_{\lambda_2,\alpha}.
\]
Here we consider two cases. 

\medskip
a) If $\lambda_1< \lambda_2$ then we use \eqref{Fl2-le} for the first
factor combined with the pointwise bound derived from the control norm  for the second, which yields 
\[
\| F_{l,\lambda_1}^{[2],2} W_{\lambda_2,\alpha}\|_
{L^2_tL^\infty_{\alpha}(A_\lambda)} \lesssim \lambda_1^\frac12 \epsilon
c_{\lambda_2} M,
\]
with off-diagonal decay which insures the summation with respect to $\lambda_1,\lambda_2 < \lambda$.

\medskip

b) If $\lambda_1\geq  \lambda_2$ then we use \eqref{Fl2-e} for the first
factor combined with local energy for the second, which yields 
\[
\| F_{\lambda_1}^{[2],2} W_{\lambda_2,\alpha}\|_
{L^2_tL^\infty_{\alpha}(A_\lambda)} \lesssim \lambda_1^{-\frac12} \lambda_2  \epsilon
c_{\lambda_1} M.
\]
This again suffices.

\medskip

\textbf{ II.c. The contribution of $F_l$ to $G_3$.} 
For all terms except a single one, it suffices to use only $ G_3^{1,1}$.
We consider first the contribution of $R$, which is
\[
\HD(\Im W\cdot \Re R_l).
\]
Here we use the Littlewood-Paley trichotomy, combining a
local energy bound for the low frequency factor with the uniform control norm bound
for the high frequency factor. The estimates follow from H\"older's and Bernstein's inequalities. We briefly describe the estimates:
\medskip

a) In the high-low case $\lambda > \mu$ we have
\[
\| \Im W_\lambda \Re R_\mu\|_{L^2_t L^\infty_{loc}(B_\lambda)} \lesssim \|\Im W_\lambda\|_{L^\infty_tL^{\infty}_{\alpha}(B_\lambda)} 
\|  \Re R_\mu\|_{L^2_t L^\infty_{loc}(B_\lambda)} \lesssim \mu^\frac12 \lambda^{-1} c_\lambda M. 
\]
\medskip

b) In the low-high case  we have
\[
\| \Im W_{<\lambda} \Re R_\lambda\|_{L^2_t L^\infty_{loc}(B_\lambda)} \lesssim \|W_{<\lambda}\|_{L^2_t L^\infty_{\alpha}(B_\lambda)} 
\|  \Re R_\lambda\|_{L^\infty_tL^{\infty}_{\alpha}(B_\lambda)} \lesssim \lambda^{-\frac12} c_\lambda M. 
\]

\medskip 
c) In the high-high case we have
\[
\begin{split}
\| P_\lambda ( \Im W_\mu \Re R_\mu) \|_{L^2_t L^\infty_{loc}(B_\lambda)}
\lesssim & \ \lambda \| \Im W_\mu \Re R_\mu \|_{L^2_t L^1_{loc}(B_\lambda)}
\\
\lesssim  & \  \lambda \| \Im W_\mu\|_{L^\infty_t L^2_{\alpha}} \| \Re R_\mu \|_{L^2_t L^2_{loc}(B_\lambda)}
\\
\lesssim  & \ \lambda^{\frac12} \mu^{-1} c_\mu M.
\end{split}
\]

Next we consider the contribution of $F^{[2]}_l$ using the Littlewood-Paley trichotomy:

a) The expression $W_{<\lambda} F^{[2]}_{l,\lambda}$. This is the most delicate case. We can easily dispense 
with $F^{[2],2}_{l,\lambda}$ via \eqref{Fl2-e} combined with the local energy bound for $W$ 
in Lemma~\ref{l:theta-LE-loc}. It remains to consider the contribution of  $F^{[2],1}_{l,\lambda }$, which, we recall,
is produced from unbalanced interactions of $R$ and $Y$. Then we are left with trilinear 
expressions of two types. 
\medskip

a.1) The trilinear form
\[
W_{<\lambda} R_\lambda \bar Y_{< \lambda}.
\]
Here we use local energy for $W$ and the uniform control norm 
 for the remaining two factors; this is identical to case (b) before since $Y$ is bounded.

\medskip

a.2)  The trilinear form 
\[
W_{< \lambda} R_{< \lambda} \bar Y_{\lambda},
\]
where using one local energy bound does not seem to suffice.  
It is only here that the decomposition $G^{\mu,1}_3+ G_3^{\mu,2}$ is needed. 
We consider three cases depending on how the two low 
frequencies $\lambda_1,\lambda_2$ compare with $\mu$.

\medskip

a.2.i)  $ \lambda_1,\lambda_2 < \mu $.
Then we group terms as
\[
R_{<\mu} W_{<\mu} \bar Y_\lambda,
\]
and use the pointwise bound derived from local energy for the  first two factors to obtain
\[
\| R_{<\mu} W_{<\mu} \bar Y_\lambda\|_{L^1 _tL^\infty_{loc}(B_\lambda)}    \lesssim \mu^{-1} M^2 c_\lambda  .
\]
 This term is placed in  $G_3^{\mu,2}$.

a.2.ii) $\lambda_1 > \mu $.  Then also $\lambda > \mu$.
We group terms as
\[
R_{\lambda_1}  W_{<\lambda} \bar Y_\lambda.
\]
Then we use the local energy bound for $W$ and the control norm for $R$
to get a bound of 
\[
\|  R_{\lambda_1}  W_{<\lambda} \bar Y_\lambda\|_{L^2_t L^\infty_{loc}(B_\lambda)} \lesssim  
\lambda_1^{-\frac12} c_\lambda c_{\lambda_1}M,
\]
 where we use the summation for both $\lambda_1$  and $\lambda$. This term is placed in  $ G_3^{\mu,1}$.

\medskip

a.2.iii) $\lambda_1 <  \mu  <  \lambda_2$. Now we switch roles and 
use local energy for $R$ and the control norm for $W$. The estimate is similar to the previous case but better. This term is also placed in 
$ G_3^{\mu,1}$.

\medskip

b) The expression $W_{\lambda} F^{[2]}_{l,<\lambda}$. This is easier, using the control norm for $W_\lambda$ 
and local energy \eqref{Fl1-le} \eqref{Fl2-le} for  $F^{[2]}$. 

\medskip

c) The expression $P_\lambda(W_{\nu} F^{[2]}_{l,\nu})$  is similar to the above,
using either local energy or the control norm for $W$ corresponding to the 
two components of $F^{[2]}$. This term is also placed in 
$ G_3^{1,1}$.

\bigskip

\textbf{III. The very low frequencies of $F$.} Here we consider the very low frequencies
 \[
F_{vl} = \Re (F-R)_{< \frac{1}{h}}.
\]
We freely omit the imaginary part of $F$, as well as $R$, which fit
within the purview of the analysis in the low frequency case.

The size of $F_{vl}$ depends on the choice of the constants, but  its derivative does not, so we 
estimate that first:

\begin{lemma}
The function $\nabla F_{vl}$   satisfies the bound
\begin{equation}\label{Fvl1-le}
 \| \nabla F_{vl} \|_{L^2_t L^\infty_{loc}(A_{1/h})} \lesssim  h^{-\frac32} M \epsilon .
\end{equation}
\end{lemma}
This is an immediate consequence of Lemma~\ref{l:ImFb} and the proof
is omitted. This estimate allows us to estimate the contribution of
$F_{vl} - F_{vl}(\alpha_0)$ to Proposition~\ref{p:nonlin}. By direct
integration, this function satisfies
\begin{equation}\label{Fvl}
| F_{vl} (z) -  F_{vl}(\alpha_0)| \lesssim g(t) |z - \alpha_0|, \qquad \| g\|_{L^2_t} \lesssim  h^{-\frac32} M \epsilon .
\end{equation}
We now consider its effect on Propositions~\ref{p:nonlin}, \ref{p:nonlin2}.

\bigskip

\textbf{ III.a. The contribution of $F_{vl}$ to $G_1$.} 
This is easily estimated combining the pointwise estimate \eqref{Fvl}
with the pointwise bound for $R_\alpha$ derived from the control norm.

\bigskip

\textbf{ III.b. The contribution of $F_{vl}$ to $G_2$.} 
This is also straightforward using the pointwise estimate \eqref{Fvl}
together with the pointwise bound for $W_\alpha$ derived from the control norm.

\bigskip

\textbf{ III.c. The contribution of $F_{vl}$ to $G_3$.} Here 
we need to consider the expression
\[
G_{3,vl} = (F_{vl} (\alpha) -  F_{vl}(\alpha_0)) \Im W (1+W_\alpha).
\]
To estimate $G_3$ at frequency $\lambda$ 
we use Bernstein's inequality to bound $\Im W$ and $W_\alpha$
in $L^\infty$ in terms of the control norm. To do this we take into
account the fact that the lowest frequency must be at least $h^{-1}$, the highest frequency must be at least
$\lambda$, as well as the fact that from \eqref{Fvl} we get a
factor of $\lambda^{-1} h^{-\frac32}$ in the region $B_\lambda(x_0)$.
The worst case scenario is when $\nu < \lambda$ and we estimate in $A_\lambda(\alpha_0)$
\[
| P_\lambda   [(F_{vl} (\alpha) -  F_{vl}(\alpha_0)) \Im W_{\nu} W_{\lambda,\alpha})]    | \lesssim f(t)\lambda^{-1} h^{-\frac32}
\mu^{-1} \lesssim f(t) \lambda^{-\frac12},
\]
with trivial $\ell^1$ summation. Thus this contribution is directly placed in 
$G_3^1$.

Here $\mu$ is limited below by $1/h$ because we use the inhomogeneous norms in $X$.

\bigskip

\textbf{IV. The constant in $F$.} 
We denote the constant by $c(t)$ which we will simply estimate 
via \eqref{const-bd}, which we recall here:
\[
|c(t)| \lesssim |R(\alpha_0)|.
\]
To evaluate the contribution of $c$ we will use the following 

\begin{lemma}
For each $\lambda > h^{-1}$ we have a decomposition
\[
c = c_\lambda^1 + c_\lambda^2,
\]
where
\begin{equation}\label{c1}
\| c_\lambda^1\|_{L^2_t} \lesssim \lambda^\frac12 M ,
\end{equation}
respectively 
\begin{equation}\label{c2}
\| c_\lambda^2\|_{L^\infty_t} \lesssim \lambda^{-\frac12} c_\lambda, 
\end{equation}
with additional $\ell^1$ summability at low frequency in the last bound.
\end{lemma}
\begin{proof}
This corresponds to the decomposition 
\[
R = R_{\leq \lambda} + R_{\geq \lambda},
\]
where for the first term we use the local energy bounds and for the second the $X$ bound.
\end{proof}
We now evaluate the effect of $c$ in Propositions~\ref{p:nonlin}, \ref{p:nonlin2}.

\bigskip

\textbf{ IV.ab. The contribution of $F_{vl}$ to $G_1, G_2$.}
For $G_1$ and $G_2$
we use the above decomposition to estimate $F R_\lambda$, respectively $F W_{\lambda,\alpha}$.
For the $c^1_\lambda$ term we use the control norm for its co-factor, and for 
the $c^1_\lambda$ term we use local energy for its co-factor. 

\bigskip

\textbf{ IV.c. The contribution of $F_{vl}$ to $G_1, G_2$.}
The same idea as above applies 
the $c \Im W$ component of $G_3$. 
Finally, for the term 
\[
c \Im W \cdot W_\alpha
\]
we apply  a similar argument, but splitting $c$ 
depending on the lowest of the two frequencies.

\appendix 
\section{Nonlinear computations}\label{appendix:nonlinear}

In this appendix, we prove another Morawetz's inequality which holds under a very mild smallness assumption on the free surface elevation $\eta$ (and without restriction on $\psi$). 
The proof is entirely different. 
It exploits the positivity of the pressure to deduce through a virial type argument a control of the kinetic energy. As a result, we obtain a bound of the local energy, which is a quadratic quantity, in terms of the momentum density $I_1$, which contains a linear term. However, by so doing, we loose  the uniformity in the depth $h$ as well as the control of the low-frequency component of the velocity potential.

\begin{theorem}\label{ThmG}
Let $g\in (0,+\infty)$. 
Let $s>5/2$ and $T$ be an arbitrary positive real number. 
Consider any solution $(\eta,\psi)\in C^0([0,T];H^{s}(\xR)\times H^s(\xR))$ 
of the water-wave system \eqref{systemT}. Given $\eps>0$ and $r>1/2$, set 
$$
m(x)=\int_0^x \frac{\dsigma}{(1+\eps^2\sigma^2)^r}.
$$
Assume that
\begin{alignat*}{2}
&(i)\qquad &&\inf_{(t,x)\in [0,T]\times\xR}\eta(t,x)\ge -\frac{h}{2},\\
&(ii) &&\sup_{(t,x)\in [0,T]\times \xR} \la \eta_x(t,x)\ra\leq \frac{1}{3},\\
& (iii) && \eps r\Big(h+\lA \eta\rA_{L^\infty}\Big)\leq \frac{1}{42}.
\end{alignat*}
Then there holds
\begin{equation}\label{nMg-alt}
\begin{aligned}
&\int_0^T\biggl\{\int m_x(x) \eta^2(t,x)\,dx +\iint_{\Omega(t)} m_x(x)\la\nabla_{x,y}\phi(t,x,y)\ra^2\,dydx\biggr\}\dt\\
&\qquad\qquad\qquad\qquad\leq 14\int_\xR m(x) I_1(t,x) \,dx\ST +2\int m(x)I_2(t,x)\,dx \ST,
\end{aligned}
\end{equation}
where
$$
I_1(t,x)=\int_{-h}^{\eta(t,x)}\phi_x(t,y)\,dy,\quad I_2(t,x)=\eta(t,x)\psi_x(t,x),
$$
and where we used the notation 
$\int f(t,x)\dx \ST=\int f(T,x)\,dx-\int f(0,x)\,dx$.
\end{theorem}
\begin{proof}

The proof is in two different steps. We first 
estimate the local kinetic energy by using the momentum density $I_1$ and the positivity of the pressure. Then we 
estimate the local potential energy by using the momentum density of $I_2$. 

\noindent\textbf{Step 1: kinetic energy.} 
We begin by proving that
\begin{equation}\label{nMg-kinetic}
\int_0^T\iint_{\Omega(t)} m_x(x)\la\nabla_{x,y}\phi(t,x,y)\ra^2\,dydxdt\leq 
7\int_\xR m(x) I_1(t,x) \,dx\ST.
\end{equation}
To do so we use the local conservation law
$\partial_t I_1+\partial_x S_1=0$ where recall that 
$$
S_1(t,x)\defn -\int_{-h}^{\eta(t,x)}(\partial_t \phi+gy)\,dy+\mez \int_{-h}^{\eta(t,x)}\big(\phi_x^2-\phi_y^2\big)\,dy.
$$
By multiplying the equation $\partial_t I_1+\partial_x S_1=0$ by $m=m(x)$ and integrating by parts, one obtains that
$$
\iint_{Q_T} S_1(t,x) m_x \,dxdt =\int m I_1\,dx \ST,
$$
where $Q_T=[0,T]\times \xR$. 
We will prove a stronger result than \eqref{nMg-kinetic}. Namely, we will prove that
\begin{equation}\label{n17}
\begin{aligned}
\iint_{Q_T} S_1(t,x) m_x \,dxdt &\ge \uq \int_0^T \iint_{\Omega(t)} m_x \la \nabla_{x,y}\phi\ra^2\,dydxdt\\
&\quad +\int_0^T \iint_{\Omega(t)} m_x P \,dydxdt\\
&\quad +\frac{h}{2}\int_0^T\int m_x   \phi_x^2(x,-h)\,dxdt.
\end{aligned}
\end{equation}
This will imply \eqref{nMg-kinetic} since the third term in the right-hand side of \eqref{n17}
is obviously positive and since the second one also since $P\ge 0$ 
(this classical result follows from the maximum principle, the fact that 
$P$ is sub-harmonic and the boundary condition on the bottom; see Lannes~\cite{LannesJAMS}).

To obtain \eqref{n17}, we start from
$$
\partial_{t} \phi +\mez \la \nabla_{x,y}\phi\ra^2 +P +g y = 0,
$$
which allows us to write $S_1$ under the form
$$
S_1(t,x)\defn \mez \int_{-h}^{\eta(t,x)}\Big(\la\nabla_{x,y}\phi\ra^2 +P\Big)\,dy
+\mez \int_{-h}^{\eta(t,x)}(\phi_x^2-\phi_y^2)\,dy.
$$
Then, to obtain \eqref{n17}, the key point is to prove that 
$$
\iint_{\Omega(t)}(\phi_x^2-\phi_y^2)\,dydx
$$
can be written as the sum of a positive term and a remainder term. This will be deduced from the following identity.

\begin{notation}
From  now on we use the shorthand notations
$$
\iint f\dxdt=\iint_{Q_T} f(t,x)\,dxdt, ~~ \iiint f\dy\dxdt=\int_0^T\iint_{\Omega(t)}f(t,x,y)\,dy dxdt.
$$
\end{notation}
\begin{lemma}\label{L3.3}
For any function $w=w(x)$ we have
\begin{multline}\label{n100}
\iint w(\phi_x^2-\phi_y^2)\,dydx=
\int w(h+\eta) \phi_x^2(x,-h)\,dx\\
\quad -2\iint w\eta_x \phi_x \phi_y\,dydx+2\iint w_x(y-\eta)\phi_x\phi_y\,dydx.
\end{multline}
\end{lemma}
\begin{proof}
This identity is proved in \cite{Boundary} when $w=1$. The time variable is seen as a parameter and we skip it. Set
$$
u(x,y)=-w(x)(y-\eta(x))\phi_y(x,y)^2.
$$
Then $u(x,\eta(x))=0$ and $u(x,-h)=0$ so $\int_{-h}^{\eta(x)} \py u\, dy=0$. 
On the other hand
$$
\py u=-2w (y-\eta)\phi_y \phi_{yy}-w\phi_y^2,
$$
so integrating on $y\in [-h,\eta(x)]$ and then on $x$ we obtain, 
remembering that $\phi_{yy}=-\phi_{xx}$,
$$
0=\iint u_y=-\iint w\phi_y^2 +2\iint w(y-\eta)\phi_y \phi_{xx}.
$$
Since $\phi_y=0$ on $y=-h$, 
by integrating by parts we infer that
$$
0=-\iint w\phi_y^2 -\iint w(y-\eta)\py \phi_x^2
+2\iint w \eta_x \phi_x\phi_y-2\iint w_x(y-\eta)\phi_x\phi_y.
$$
Thus
\begin{align*}
0&=-\iint w\phi_y^2 -\iint \py\big( w(y-\eta)\phi_x^2\big)\\
&\quad+\iint w \phi_x^2+2\iint w\eta_x \phi_x\phi_y-2\iint w_x(y-\eta)\phi_x\phi_y.
\end{align*}
Since
$$
\int_\xR\int_{-h}^\eta \py\big( w(y-\eta)\phi_x^2\big)\dydx=\int_\xR w(h+\eta)\phi_x^2(x,-h)\,dx,
$$
this proves the desired result.
\end{proof}

Set
$$
\Sigma\defn \mez \iiint m_x \la \nabla_{x,y}\phi\ra^2\dydxdt +\mez\iiint m_x\big(\phi_x^2-\phi_y^2\big)\,dydxdt.
$$
It follows from the previous lemma that
$\Sigma =\Sigma_1+\Sigma_2$ with 
\begin{align*}
\Sigma_1&=\iiint \Big(\frac{m_x}{2}-m_x\eta_x+m_{xx}(y-\eta)\Big) \la \nabla_{x,y}\phi\ra^2\,dydxdt,\\
\Sigma_2&=\iint m_x(h+\eta) \phi_x^2(x,-h)\,dxdt.
\end{align*}
Now we assume that $\eta\ge -h/2$. Then 
$$
\Sigma_2\ge \frac{h}{2}\iint m_x \phi_x^2(x,-h)\,dxdt.
$$
Now recall that by definition,
$$
m(x)=\int_0^x \frac{\dsigma}{(1+\eps^2\sigma^2)^r},
$$
with $r>1/2$ and where $\eps$ has to be chosen. Then
$$
m_{xx}(x)=-r\frac{2\eps^2 x}{(1+\eps^2 x^2)^{r+1}}=\eps C(\eps,x)m_x(x)\quad \text{with}\quad 
C(\eps,x)=-2r\frac{\eps x}{1+\eps^2 x^2}.
$$
Since $\la C(\eps,x)\ra\leq r$, we obtain that $\la m_{xx}(x)\ra \leq \eps r m_x(x)$. 
As a result,
$$
\la m_{xx}(y-\eta)\ra\leq \eps r(h+\lA \eta\rA_{L^\infty})m_x.
$$
Since, on the other hand, one has $\la \eta_x\ra\leq \frac{1}{3}$ by assumption, we conclude that
$$
\frac{m_x}{2}-m_x\eta_x+m_{xx}(y-\eta)\ge \Big(\frac{1}{6}-\eps r(h+\lA \eta\rA_{L^\infty})\Big)m_x.
$$
Then, assuming that
$$
\eps r\Big(h+\lA \eta\rA_{L^\infty}\Big)\leq \frac{1}{42},
$$
we conclude that
$$
\Sigma_1\ge \frac{1}{7}\iiint m_x \la \nabla_{x,y}\phi\ra^2\dydxdt,
$$
which completes the proof of \eqref{n17} and hence the proof of \eqref{nMg-kinetic}.

\bigbreak

\noindent\textbf{Step 2: estimate of the potential energy.} In light of \eqref{nMg-kinetic}, to prove Theorem~\ref{ThmG}, it is sufficient to 
prove the following estimate about the potential energy:
\begin{equation}\label{nMg2}
\begin{aligned}
\int_0^T \int gm_x(x) \eta^2(t,x)\,dx \dt&\leq \iint_{\Omega(t)} m_x(x)\la\nabla_{x,y}\phi(t,x,y)\ra^2\, dydxdt\\
&\quad+2\int m(x)I_2(t,x)\,dx \ST,
\end{aligned}
\end{equation}
where recall that $I_2(t,x)=\eta(t,x)\psi_x(t,x)$.

We now work with the density momentum $I_2$ and the associated flux force $S_2$. 
Recall that
$$
S_2= -\eta\psi_t-\frac{g}{2}\eta^2+\mez \int_{-h}^{\eta}(\phi_x^2-\phi_y^2)\,dy,
$$
Again, it follows from the local conservation law 
$\partial_t I_2+\partial_x S_2=0$ that, for any weight 
$m=m(x)$ and any time $T$, one has
\begin{equation}\label{n210}
\iint_{Q_T} S_2(t,x) m_x \,dxdt =\int_\xR m(x)I_2(T,x)\dx-\int_\xR m(x)I_2(0,x)\,dx,
\end{equation}
where $Q_T=[0,T]\times \xR$.

Let us introduce a notation. Set 
\[
N(\eta)\psi=\mez \psi_x^2-\mez \frac{(G(\eta)\psi+\eta_x\psi_x)^2}{1+\eta_x^2},
\]
so that the Bernouilli equation reads
\[
\partial_t \psi+g\eta +N(\eta)\psi=0.
\]

We begin by reporting the expression for $\partial_t\psi$ given by \eqref{systemT} to obtain
$$
S_2=\frac{g}{2}\eta^2+\eta N(\eta)\psi
+\mez \int_{-h}^{\eta(t,x)}(\phi_x^2-\phi_y^2)\,dy.
$$

Let us recall a lemma from \cite{Boundary} which allows to handle 
the integral involving $N(\eta)\psi$.

\begin{lemma}\label{L3.2}
For any function $\mu=\mu(x)$ there holds
\begin{equation}\label{C6-bis}
\int_\xR \mu N(\eta)\psi\dx
=-\iint_{\Omega} \mu_x\phi_x\phi_y \,dydx +\mez \int \mu \phi_x^2\arrowvert_{y=-h}\,dx.
\end{equation}
\end{lemma}
\begin{proof}
One can check that
\[
N(\eta)\psi=\mathcal{N} \big\arrowvert_{y=\eta}
\] with
\[\mathcal{N}=\mez\phi_x^2-\mez \phi_y^2+\eta_x \phi_x \phi_y\label{t90}.
\]
%Here, for convenience, we also recall the expression of the mean %curvature 
%\[
%H(\eta)=\px\left(\eta_x / \sqrt{1+\eta_x^2}\right).
%\] 
The proof then relies on the following identity
$$
\partial_y \big( \phi_y^2-\phi_x^2\big)+2\partial_x \big(\phi_x\phi_y\big)=
2\phi_y \Delta_{x,y}\phi,
$$
which implies that, since $\phi$ is harmonic and $\partial_y \mu=0$,
$$
\partial_y \big( \mu \phi_y^2-\mu \phi_x^2\big)+2\partial_x \big(\mu \phi_x\phi_y\big)=
2\mu_x\phi_x\phi_y.
$$
We deduce that the vector field $X\colon \Omega\rightarrow \xR^2$ defined by $X=(-\mu \phi_x\phi_y;\frac{\mu}{2}\phi_x^2-\frac{\mu}{2}\phi_y^2)$ 
satisfies
$\cn_{x,y} \big(X\big)=-\mu_x\phi_x\phi_y$. 
Since $\nabla_{x,y}\phi$ belongs to $C^1(\overline{\Omega})$ and since one has 
the boundary conditions
$$
\phi_y\arrowvert_{y=-h}=0,
$$
an application of the divergence theorem gives that
\begin{align*}
-\iint_\Omega \mu_x\phi_x\phi_y \,dydx&=\iint_\Omega \cn_{x,y}X\,dydx\\
&=\int_{\partial\Omega}X\cdot n \dsigma= \int \mu \mathcal{N} \big\arrowvert_{y=\eta} \,dx
-\mez\int \mu\phi_x^2\arrowvert_{y=-h}\,dx.
\end{align*}
This completes the proof.
\end{proof}

By combining this result with Lemma~\ref{L3.3}, we conclude that
\begin{equation}\label{n211}
\begin{aligned}
\iint m_x S_2\, dxdt&=\iint \frac{g}{2}m_x \eta^2\, dxdt\\
&\quad +\iint m_x \left(\frac{h}{2}+\eta\right)\phi_x^2\arrowvert_{y=-h}\,dxdt\\
&\quad +\iiint \left(m_{xx}y-2\eta m_{xx}-2m_x\eta_x\right)\phi_x\phi_y\,dydxdt.
\end{aligned}
\end{equation}
Now, by assumptions, one has
$$
\frac{h}{2}+\eta\ge 0,\quad \la m_{xx}\ra\leq \eps r \la m_x\ra,\quad 
\sup \la \eta_x\ra\leq \frac{1}{3}.
$$
Consequently,
\begin{align*}
\la m_{xx}y-2\eta m_{xx}-2m_x\eta_x\ra 
&\leq \left( \eps r (h+\lA \eta\rA_{L^\infty})+2 \eps r \lA \eta\rA_{L^\infty}+\frac{2}{3}\right)\la m_x\ra\\
&\leq \left(3\eps r(h+\lA \eta\rA_{L^\infty})+\frac{2}{3}\right)\la m_x\ra\\
&\leq \left(\frac{3}{42}+\frac{2}{3}\right)\la m_x\ra\leq \la m_x\ra.
\end{align*}
So, \eqref{n211} implies that
$$
\iint \frac{g}{2}m_x \eta^2\,dxdt\leq  \iint m_x S_2\,dxdt 
+\mez \iiint m_x \la \nabla_{x,y}\phi\ra^2\,dydxdt.
$$
The desired result \eqref{nMg2} then follows from \eqref{nMg-kinetic} and \eqref{n210}. 

This completes the proof of Theorem~\ref{ThmG}.
\end{proof}

\vspace{10mm}

\noindent\textbf{Thomas Alazard}\\
\noindent CNRS and CMLA, \'Ecole Normale Sup{\'e}rieure de Paris-Saclay, Cachan, France

\vspace{3mm}

\noindent\textbf{Mihaela Ifrim}\\
\noindent Department of Mathematics, University of 
Wisconsin--Madison, USA

\vspace{3mm}

\noindent\textbf{Daniel Tataru}\\
\noindent Department of Mathematics, University  of California, Berkeley, USA

\end{document}